\documentclass[11pt,leqno,twoside]{amsart}
\usepackage{fullpage}
\usepackage{amsmath, amsfonts, amsthm, amssymb, amscd, stmaryrd, url, amstext, amsxtra, amsopn,hyperref,cancel,todonotes,booktabs}
  \numberwithin{equation}{section}
\usepackage{mathrsfs}
\usepackage{longtable,makecell}
\usepackage{multirow}
\usepackage[capitalize]{cleveref}

\usepackage{caption}

\DeclareMathOperator*{\Res}{Res}
\DeclareMathOperator{\Tr}{Tr}
\DeclareMathOperator{\supp}{supp}
\DeclareMathOperator{\sgn}{sgn}
\newcommand{\mymod}[1]{(\operatorname{mod} #1)}
\renewcommand{\d}{{\delta}}
\renewcommand{\i}{{\mathrm{i}}} 
\newcommand{\I}{1\!\!1} 
\renewcommand{\O}{{\mathcal{O}}}
\newcommand{\R}{{\mathbb{R}}}
\newcommand{\C}{{\mathbb{C}}}
\newcommand{\Q}{{\mathbb{Q}}}
\newcommand{\Z}{{\mathbb{Z}}}
\newcommand{\N}{{\mathbb{N}}}
\newcommand{\B}{{\bf B}}
\newcommand{\M}{\mathcal{M}}
\renewcommand{\Re}{{\mathfrak{Re}}}
\renewcommand{\Im}{{\mathfrak{Im}}}
\newcommand{\w}{\omega}
\renewcommand{\a}{\alpha}
\renewcommand{\b}{\beta}
\newcommand{\g}{\gamma}
\renewcommand{\l}{\lambda}
\renewcommand{\t}{\theta}
\newcommand{\s}{\sigma}
\newcommand{\del}{\delta}
\newtheorem{theorem}{Theorem}
\newtheorem{corollary}{Corollary}[theorem]
\newtheorem{lemma}[theorem]{Lemma}

\newtheorem{proposition}[theorem]{Proposition}
\theoremstyle{remark}
\newtheorem*{remark}{Remark}
\reversemarginpar %forces margin note to appear on the inside of the margin

\theoremstyle{theorem}
\newtheorem{innercustomthm}{Theorem}

\theoremstyle{corollary}
\newtheorem{innercustomcor}{Corollary}

\theoremstyle{proposition}
\newtheorem{innercustomprop}{Proposition}

\theoremstyle{lemma}
\newtheorem{innercustomlem}{Lemma}

\newcommand{\erfc}{\textnormal{\text{erfc}}}

\begin{document}
\title{Sharper bounds for the Chebyshev function $\theta(x)$} 
\author{Samuel Broadbent, Habiba Kadiri, Allysa Lumley, Nathan Ng, Kirsten Wilk}
%
%\address{Mathematics and Computer Science Department, University of Lethbridge, 4401 University Drive, Lethbridge, Alberta, T1K 3M4 Canada}
%%
%\email{email@uleth.ca}

\address{University of Lethbridge \\ Department of Mathematics and Computer Science \\ 4401 University Drive \\ Lethbridge, AB \ T1K 3M4 \\ Canada}
\email{sam.broadbent@uleth.ca}
\email{habiba.kadiri@uleth.ca}
\email{lumley@crm.umontreal.ca}
\email{nathan.ng@uleth.ca}
\email{kirsten.wilk@uleth.ca}

\begin{abstract}
In this article, we provide explicit bounds for the prime counting functions $\theta(x)$ for all ranges of $x$. 
The bounds for the error term for $\theta (x)- x$ are of the shape $\varepsilon x$ and $\frac{c_k x}{(\log x)^k}$, for $k=1,\ldots,5$.
%and new bounds for $\psi(x)$ which are the best known for most values of $x$.  
%This is based on the classical work of Rosser and Schoenfeld \cite{RS2}, \cite{RS1} and improves recent articles of Dusart \cite{Dus5}, \cite{Dus3}, and \cite{Dus4}.
%The main new inputs  for the bounds for $\psi(x)$ is a recent explicit zero-density result due to Kadiri, Lumley, and Ng \cite{KLN}. 
%Using these bounds for $\psi(x)$ and also B\"{u}the's bounds for $\psi(x)$ \cite{But},  we can deduce our bounds for $\theta(x)$.  
Tables of values for $\varepsilon$ and $c_k$ are provided. %ables of Moreover, we provide extensive tables with bounds for $\theta(x)$.
\end{abstract}
\maketitle 
%%
%remove anything other than RS62
\section{Introduction}
\subsection{History}
%%%
%We recall that 
%$\pi(x)$ is the prime counting function $\# \{p\le x: p \text{ prime}\}$ and that $\psi(x)$ and $\theta(x)$ are the Chebyshev functions:
%\[\theta(x) = \sum_{p\le x} \log p , \ \psi(x) = \sum_{p^k\le x, k\ge 1} \log p, \]
%%where $\Lambda(n)$ is the Von-Mangold function: $\Lambda(n)=\log p$ if $n$ is a prime power, $\Lambda(n)=0$ otherwise.
In 1852 Chebyshev \cite{Cheb1} proved that if $x$ is large enough, then 
\[
0.9212  \frac{x}{\log x} \le \pi(x) \le 1.1056 \frac{x}{\log x}\ \text{ as}\ x\to \infty,
\]
where $\pi(x)$ denotes the number of primes less than or equal to $x$. 
It was the first major step towards the prime number theorem. 
He introduced what are  now referred to as the  Chebyshev functions: 
\[
\theta(x) = \sum_{p\le x} \log p , \ \text{and}\ \psi(x) = \sum_{p^k\le x, k\ge 1} \log p, 
\]
%where $\Lambda$ is the Von Mangoldt function supported on prime powers $p^k$ and defined by $\Lambda(p^k) = \log p$. 
and he proved that for all $x \ge 30$
\[
% \label{Cheb-bnd-psi}
 Ax - \tfrac52 \log x -1 < \psi(x) < \tfrac{6}{5}Ax+ \tfrac{5}{4 \log 6}(\log x)^2 + \tfrac54 \log x +1 ,
\]
and 
\[
% \label{Cheb-bnd-theta}
Ax - \tfrac{12}{5}Ax^{\frac12} -\tfrac{5}{8\log 6} (\log x)^2 - \tfrac{15}{4}\log x - 3 < \theta(x) < \tfrac65 Ax - Ax^{\frac12} + \tfrac{5}{4 \log 6}(\log x)^2
 + \tfrac52 \log x +2
\]
where $A= \log ( 2^{\frac12}3^{\frac13}5^{\frac15} 30^{-\frac{1}{30}} ) = 0.9212 \ldots$ and $\tfrac65 A =1.1055\ldots$ \footnote{There is a typo in the definition of $A$ in \cite{Cheb1} on p. 376. We have
given a corrected definition of A.}. \\
%{\commentcolor{red} We think Chebyshev made a typo in his 
%definition of A, we changed the definition here to be accurate.}
As a consequence, there exists $x_0 > 0$ such that 
\[
 0.9212  x \le \psi(x) \le  1.1056 x \text{ for all } x \ge x_0
\]
and that there exists $x_1  > 0$ such that 
\[
 0.9212 x \le \theta(x) \le 1.1056 x \text{ for all } x \ge x_1 .
\]
Such bounds are now known as Chebyshev bounds. Over the years, many other elementary arguments have yielded improved
bounds:
% \footnote{These are the elementary bounds we are aware of. If readers know of others please let us know. } 
\begin{center}
\begin{tabular}{|c|c|c|c|c|c|c|}
\hline
 \multirow{2}{*}{Author} & \multicolumn{3}{|c|}{Bounds for $\psi(x)/x$} & \multicolumn{3}{|c|}{Bounds for $\theta(x)/x$} \\ 
	& upper & lower  & range &  upper  & lower  & range \\ \hline
 Erd\"os (1932) \cite{Erdos} & $1.38629$ & - & $x > 0$ & - & - & - \\ \hline
 Hanson (1972) \cite{Han} & - & - & - & $1.09861$ & - & $x>0$ \\ \hline
 Grimson \& Hanson  & \multirow{2}{*}{$1.09861$} & \multirow{2}{*}{-}  & \multirow{2}{*}{$x > 0$} & \multirow{2}{*}{$1.0508$} & \multirow{2}{*}{-} & \multirow{2}{*}{$x > 0$} \\ 
 (1977) \cite{GrmHan} & & &  & &  & \\ \hline
 Deshouillers & $1.07715$ & - & $x > 0$ & - & - & - \\ 
 (1977) \cite{Des} & - & $0.92129$ & $x \ge 59$ & - & - & -\\ \hline
\end{tabular}
\end{center}
%%
%\begin{center}
%\begin{tabular}{|c|c|c|c|c|c|c|}
%\hline
% \multirow{2}{*}{Author} & \multicolumn{3}{|c|}{Bounds for $\psi(x)/x$} & \multicolumn{3}{|c|}{Bounds for $\theta(x)/x$} \\ 
%	& upper & lower  & range &  upper  & lower  & range \\ \hline
% Erd\"os (1932) \cite{Erdos} & $1.38629$ & & $x > 0$ & & & \\ \hline
%Rosser (1941) \cite{Ros} & & & & $1.019$ & $0.980$& $x > e^{20}$  \\ \hline
%\multirow{3}{*}{Rosser, Schoenfeld (1962)  \cite[Thm. 9, 10, 12]{RS1} }%\\ \cite[Thm. 10]{RS1} }%\\ \cite[Thm. 12]{RS1} }
%&                   &  &                    & $1.01624$ &  & $x> 0$ \\ 
%&                   &  &                    & & $0.84$   & $x\ge 101$  \\
%& $1.03883$ &  & $x> 0$        &                   &  &             \\  \hline
% Hanson (1972) \cite{Han} & & & & $1.09861$ & & $x>0$ \\ \hline
% Deshouillers & $1.07715$ & & $x > 0$ & & & \\ 
% (1977) \cite{Des} & & $0.92129$ & $x \ge 59$ & & & \\ \hline
% Grimson \& Hanson  & \multirow{2}{*}{$1.09861$} & & \multirow{2}{*}{$x > 0$} & \multirow{2}{*}{$1.0508$} & & \multirow{2}{*}{$x > 0$} \\ 
% (1977) \cite{GrmHan} & & & & & & \\ \hline
%                      & & & & $1.001884$ & & $x > 0$ \\ 
%  Costa Pereira (1989) \cite{Cos2} & & & & & $0.985$ & $x > 11\ 927$ \\ 
%                    & & & & & $0.998$ & $x > 487\ 381$ \\ \hline
%\end{tabular}
%\end{center}
%
The Prime Number Theorem, as proven independently by de la Vall\'{e}e Poussin \cite{dlVP} and Hadamard  \cite{Had} in 1896, states that the number of primes up to $x$ satisfies 
\[
\pi(x)\sim \frac{x}{\log x}\ \text{as}\ x\to \infty.
\]
It can also be easily reformulated as 
\[ 
\psi(x) \sim x    \text{ and }  \theta(x) \sim x \text{ as }  x\to \infty.
\]
%where $\psi(x)$ and $\theta(x)$ are the Chebyshev functions
%\[
%\theta(x) = \sum_{p\le x} \log p , \ \psi(x) = \sum_{p^k\le x, k\ge 1} \log p. 
%\]
%The proof establishes the equivalent statement for the weighted version $\theta(x)$:
%\[\theta(x) = \sum_{p\le x} \log p  \sim \ \text{as}\ x\to \infty,\]
%which is also equivalent to 
%\[\psi(x) = \sum_{p^k\le x, k\ge 1} \log p \sim x\ \text{as}\ x\to \infty.\]
%%Rewriting $\psi$ using the Von-Mangold function $\Lambda(n)=\log p$ if $n$ is a prime power, $\Lambda(n)=$ otherwise:
%%\[\psi(x) = \sum_{n\le x} \Lambda(n),\]
%%then 
%%These are smoothed versions of $\pi(x)$:
%%\[\psi(x) = \sum_{p^k\le x, k\ge 1} \log p, \ \theta(x) = \sum_{p\le x} \log p.\]
The essence of the proof as suggested by Riemann is to relate $\psi(x)$ to the zeta function $\zeta(s)$.
This allows one  to use the properties of the zeros of $\zeta$, and in particular their location in the complex plane. For instance, Hadamard and de la Vall\'{e}e Poussin proved that $\zeta(s)$ does not vanish on the vertical $1$-line.
By refining this result, de la Vall\'ee Poussin proved in $1899$ that the error term in estimating $\pi(x)-\text{li}(x)$ (and $\theta(x)-x$  and $\psi(x)-x$) is asymptotically of size
%\left| \frac{\psi(x)-x}{x} \right| \ll  
$x \exp ( -c \sqrt{\log x} )$.
%%
%{\color{red} first psi - analytical, then theta, both bounds epsilon constant, then $exp(-c\sqrt{\log x})$ by Shoenfeld 76}\\
%The prime number theorem, 
%implies that 
%\[
% \psi(x) \sim x   \text{ and }
% \theta(x) \sim x \text{ as } x\to \infty.
%%\]
%We can rephrase these results explicitly: 
%for every $\varepsilon > 0$, there exists $x_0 >0$ such that 
%\[
%  (1-\varepsilon)x \le \psi(x) - \le (1+\varepsilon) x \text{ for all } x \ge x_0, 
%\]
%for every $\varepsilon' > 0$, there exists $x_1  >0$ such that 
%\[
%  (1-\varepsilon')x \le \theta(x)  \le (1+\varepsilon') x \text{ for all } x \ge x_1 ,
%\]
%and 
%for every $\varepsilon'' > 0$, there exists $x_2 >0$ such that 
%\[
%\le \left| \theta(x) - x \right| \le A x \exp ( -C  \sqrt{\log x} ) \text{ for all } x \ge x_1 ,
%\]
%These are asymptotic results and do not tell us what the value of $x_0$ is for a given $\varepsilon$.  
%Then, until the 1970's Rosser joined forces with Schoenfeld and significantly improved Rosser's earlier work.
%There have been many partial previous results of the type of Theorem \ref{MainResult:theta} spread throughout the literature.
%Rosser and Schoenfeld (RS) were pioneers in exploiting the potential of 
Between 1941 and 1976, Rosser and Schoenfeld (together or separately) developed a program of determining explicit results for the Chebyshev functions, 
%Change here. wording. 
as well as various finite sums and products over primes, including Mertens sums, and the size of the $n^{th}$ prime.
We provide a sample of the numerous inequalities they established:
\begin{eqnarray*}
 0.980x \le \theta(x) \le 1.019x, \text{ for all } x \ge e^{20} & \text{\cite[Equation (12)]{Ros}}, 
 \\ 
\frac{x}{\log x} < \pi(x) < 1.25506\frac{x}{\log x}, \ \text{for all}\ x\ge 17 & \text{\cite[Corollary 1]{RS1}},
\\ 
|\psi(x)-x| , |\theta(x)-x| < x (\log x)^{\frac{1}{2}}  \exp\big(- ( \tfrac{\log x}{R})^{\frac{1}{2}} \big), \ \text{for all}\ x\ge 2 & \text{\cite[Theorem 11]{RS1}} ,
\\  
|\psi(x)-x| , |\theta(x)-x| < 0.024 2269 \frac{x}{\log x}, \text{ for all }\ x \ge 10^8 & \text{\cite[Theorem 7]{RS2}}
%\\\left| \sum_{p\le x} \frac{\log p}{p} - \left( \log x -E\right)\right| < \frac1{2\log x} & \text{\cite[Theorem 6]{RS1}} 
\end{eqnarray*}
with $R = 17.51\ldots$. %, and $E = -\gamma - \sum_{n\ge 2} \sum_p \frac{\log p}{p^n} = -1.33258\ldots$ where $\gamma$ is Euler's constant. %\ \text{with}\ R = 17.51\ldots \\
%Finally in 
%there are positive constants
%$a,b$, and $x_0$ such that 
%\[
%a\frac{x}{\log x} \le \pi(x) \le b\frac{x}{\log x}\ \text{for all }\ x \ge x_0.
%\] 
%Here $a,b$ are computable constants which values depends on how large $x_0$ is. \\   
%In \cite[Theorem 6]{RS1} they established that for all $x\ge 319$
%\[
%\left| \sum_{p\le x} \frac{\log p}{p} - \left( \log x -E\right)\right| < \frac1{2\log x} ,
%\]
%where $E = -\gamma - \sum_{n\ge 2} \sum_p \frac{\log p}{p^n} = -1.33258\ldots$ and $\gamma$ is Euler's constant.
%
%They also established explicit bounds for Chebyshev's prime counting functions $\psi(x)$ and $\theta(x)$ 
%In a 1976 article of Schoenfeld many interesting applications from the works of Rosser and Schoenfeld were derived, such as sharp bounds for the $n$-th prime number $p_n$.  
%Other interesting applications included bounds for various quantities involving finite sums and products of prime numbers. 
%Many of these bounds are still in use but only few have been improved. 
Here is a non-exhaustive list that the interested reader can consult: Axler \cite{Axler17_2}, B\"uthe \cite{But2}, Costa-Pereira \cite{Cos2}, Dusart \cite{Dus4}, Faber and Kadiri \cite{FaKa}, and Trudgian \cite{Tru}. 
%. We are hoping to provide a more modern approach together with sharper explicit results. This would better reflect the progress in our knowledge about the zeros of the Riemann zeta functions and in our more efficient numerical tools. 
%Using our results for $\psi(x)$ we were able to improve  work of Axler and Dusart.   \\
%%%
\subsection{Main Theorem}
Our goal is to give a comprehensive and complete description of how to obtain an explicit bound for the error term for $\theta(x)$ of the form $\frac{x}{(\log x)^k}$, no matter the size of $x$ and for values of $k$ that are most widely used. 
\begin{theorem}\label{MainResult:theta}
Let $k$ be an integer with $0 \le k \le 5$.  For any fixed $X_0 \ge 1$, there exists $m_k> 0$ such that, for all $x \ge X_0 $
\begin{equation}
  \label{mklowerbound}
x \left( 1- \frac{m_k}{(\log x)^k} \right) \le \theta(x) .
\end{equation}
For any fixed $X_1 \ge 1$, there exists $M_k>0 $ such that, for all $x \ge X_1 $
\begin{equation}
  \label{Mkupperbound}
\theta(x) \le x \left( 1+ \frac{M_k}{(\log x)^k} \right).
\end{equation}

%
%
%\begin{align*}
%  x \left( 1- \frac{m_k}{(\log x)^k} \right) \le \ &\theta(x)  & &\text{ for all } x \ge X_0 
%\\ \text{and} \qquad \qquad
%  &\theta(x) \le x \left( 1+ \frac{M_k}{(\log x)^k} \right)  && \text{ for all } x \ge X_1.
%\end{align*}
In the case $k=0$ and $X_0, X_1 \ge e^{20}$, we have 
\[
 m_0 = \varepsilon(\log X_0) + 1.03883 ( X_0^{-1/2} + X_0^{-2/3} + X_0^{-4/5} ) \quad \text{ and } \quad M_0 = \varepsilon(\log X_1).
\]
See Table \ref{ThetaEpsilon} for values of $m_0$ and $M_0$, and Table \ref{Table:MasterResults} for values of $m_k$ and $M_k$, for $k \in \{1,2,3,4,5\}$.
\end{theorem}
\pagebreak %This pagebreak exists for Table 1 Placement
Here $\varepsilon(b)$ is a positive constant associated to $\psi(x)$, defined in the next theorem. \\
The proof of Theorem \ref{MainResult:theta} is given in Section \ref{thetalemmataproofs}.
In addition, the reader will find there the first formal algorithm to automatically deduce new bounds for $\theta(x)$ every time new bounds are generated for $\psi(x)$.  In particular, we describe how the values for $m_k$ and $M_k$ depend on $\varepsilon(b)$.  We have also produced extended versions of  Tables \ref{Wedpsixvals} - \ref{Table:MasterResults}
  in \cite{BKLNW} which will be made available on 
our personal webpages and on the arXiv. 

B\"uthe's \cite[Theorem 2, (1.7)]{But2} implies that, for all $x<10^{19}$, $\theta(x)<x$ , giving $M_k=0$ for all $k$ and $x$ in this range.
In addition we did direct calculations for values of $x$ up to $7\cdot10^{11}$ (see Table \ref{Table-Dk}).
Thus Theorem \ref{psixmainthm} gives relevant values for $m_k,M_k$ for $X_0,X_1>7\cdot10^{11}>e^{27}$ as listed in Table \ref{Table:MasterResults}.
For more extensive calculations of $m_k,M_k$, we refer the reader to Tables in \cite{BKLNW}. %\footnote{where do we put this?}
For instance, we obtain from \cite[Table 14 and Table 15]{BKLNW} respectively that, for all $x\ge  10^{19}$, 
\begin{equation*}
\begin{split}
 1.9338\cdot 10^{-8} < \frac{\theta(x)-x}{x}  & < 1.9667 \cdot 10^{-8} \text{ and }
\\
%and from Table \cite[Table 15]{BKLNW} that
 \left| \frac{\theta(x)-x}{x} \right|  & <   \frac{3.79\cdot 10^{-5}}{(\log x)^2}.
\end{split}
\end{equation*}
The following result gives explicit bounds for $\psi(x)$ and is based on the articles 
\cite{But}, \cite{But2} and \cite{PT2019}. 
\begin{theorem}[B\"uthe, Platt-Trudgian]
\label{psixmainthm}
Let $b >0$.  Then there exists a positive constant $\varepsilon(b)$ such that 
\begin{equation}
   \left| \frac{\psi(x)-x}{x}\right| \le \varepsilon(b),\ \text{for all}\ x\ge e^b.
\end{equation}
\end{theorem}
% where $m_0 = \varepsilon(\log X_0) + c_0\left( X_0^{-1/2} + X_0^{-2/3} + X_0^{-4/5} \right)$ and $M_0 = \varepsilon(\log X_1)$.
%
%
We use \cite[Theorem 2]{But}, \cite[Theorem 1]{But2} and \cite[Theorem 1]{PT2019}  to compute a more exhaustive list of values for $\varepsilon(b)$ 
which we need for our calculations of $m_k,M_k$.
%The articles \cite{But2} and \cite{PT2019} only provide values of $\varepsilon(b)$ for select values of $b$ and we require  
In particular details of the calculations of $\varepsilon(b)$ are provided in Appendix \ref{psixboundssection} and tables of values 
are given  in Table \ref{Wedpsixvals} of Appendix \ref{Section:Tables}.   
The table lists the best values obtained using either of these techniques. B\"{u}the's uses a smoothing technique as introduced in \cite{FaKa} with a weight arising from the Logan function, while Platt and Trudgian use a truncated Perron's formula combined with the zero density obtained in \cite{KLN}. B\"uthe's technique leads to better bounds when $x< e^{2300}$ while Platt and Trudgian's works better for larger values of $x$.
For instance for all $x\ge e^{3000}$, the  method from \cite{PT2019} gives $\varepsilon(b) = 4.60 \cdot 10^{-14}$, and for all $x\ge e^{46}$, method from \cite{But2} gives $\varepsilon(b) = 6.95 \cdot 10^{-9}$. 
%
%In this article, we use this new result to improve \cite{Dus4} {\color{red}(see Theorem \ref{} and Section \ref{}. }
%In addition we compare all known methods (between B\"uthe's \cite{But2} and 
%The bounds calculated  bounds for $\psi(x)$
% by using an improved zero-density result of .
%
%We now state some of the main theorems of this article.
Improved estimates for $\psi(x)$ may be used to derive 
%Change here. Wording. 
new bounds for $\theta(x)$.   %We recall the general argument. 
More precisely, we have %The first step is to write 
%\[
%  \theta(x)-x = (\psi(x)-x)+ (\theta(x)-\psi(x)) ,
%\]
%%and then apply the triangle inequality, and the fact that 
%and $\psi(x)-\theta(x) \ge 0$, then 
\[
%  \label{keyidentity}
  |\theta(x)-x| \le  |\psi(x)-x| + \psi(x)-\theta(x),
\]
where $\psi(x)-\theta(x)$ introduces an error term of size $\sqrt{x}$. 
We study this term in Section \ref{psithetaresults}, and Theorem \ref{psi-theta:GeneralThm} provides a refinement to \cite{RS2} and \cite{Dus4}.
% 
%In Section \ref{psithetaresults}, we prove there exist $a_1$ and $a_2$ such that
%\begin{equation}
%  \label{psithetadiff}
% \psi(x) - \theta(x) \le a_1x^{\frac12} + a_2x^{\frac13} ,
%\end{equation}
%for all $x \ge x_0$, with $a_1,a_2$ depending on $x_0$.
% Theorem \ref{psi-theta:GeneralThm} refined 
%%we derive a bound for $  |\theta(x)-x|$ from Theorem \ref{psitheorem} about $|\psi(x)-x|$
%and from Theorem \ref{psi-theta:GeneralThm} about $\psi(x)-\theta(x)$. 
%The bounds for $|\psi(x)-x|$ will be taken from Section \ref{psixboundssection} and from \cite[Theorem 1]{But}.
%Recall that 
%\[
% \psi(x) =\theta(x) + \theta(x^{\frac{1}{2}})+ \theta(x^{\frac{1}{3}})+ \cdots, 
%\]
%so that
%\[
%%   \label{psithetadiff} - Commented out as this label is never called until later, after it is redefined.
%  \psi(x)-\theta(x) = \sum_{k=2}^{ \lfloor \frac{\log x}{\log 2} \rfloor }
%  \theta(x^{\frac{1}{k}}). 
%\]Thus 
%
%we can derive for every $x_0 > 0$ that there exist $a_1$ and $a_2$ such that
%\begin{equation}
%  \label{psithetadiff}
% \psi(x) - \theta(x) \le a_1x^{\frac12} + a_2x^{\frac13} \quad \text{ for all } x \ge x_0.
%\end{equation}
%
% and slight improvements to the previous values for $a_1$ and $a_2$ will be given.  
%Using Theorem \ref{psitheorem} and \eqref{psithetadiff} we are able to derive new bounds for $\theta(x)$.
%Change here. below corrected. 
We give just below a non-exhaustive historical recollection of bounds of the type \eqref{mklowerbound} and \eqref{Mkupperbound}. 
The values in Tables \ref{tablek0}-\ref{tablek4} make use of explicit formula techniques.

\begin{center}
\begin{table}[!htbp] 
 \caption{Case $k=0$}
 \begin{tabular}[h]{|l|rr|rrr|}
 \hline
 Author & $m_0$ & $X_0$ & $M_0$ & $X_1$  & \\ 
 \hline
 \multirow{4}{*}{Rosser (1941) \cite{Ros} }
				& - & - & $0.0376$ & $ 1$ & \\
				& $0.0393$ & $e^{13.8}$ & $0.0376$ & $e^{13.8}$ & \\
				& $0.0328$ & $e^{15}$ & $0.0321$ & $e^{15}$ & \\
				& $0.02$ & $e^{20}$ & $0.0199$ & $e^{20}$ & \\
 \hline
\multirow{4}{*}{Rosser $\&$ Schoenfeld (1962) \cite{RS1}   }
	                              	   & $0.16$ & $101$ & - & - & \\
		                            & $0.05$ & $1\,427$ & - & - & \\
	                                    & $0.02$ & $7\,481$ & - & - & \\           
                                            & - & - & $0.01624$ & $1$ & \\
\hline
\multirow{5}{*}{Rosser $\&$ Schoenfeld (1975) \cite{RS2}	}
                  & $0.015$ & $11\ 927$ & - & - &\\ 
                  & $0.010$ & $32\ 057$ & - & - & \\
		 & $0.005$ & $89\ 387$ & - & - &\\
          	 & $0.002$  & $487\ 381$ & - & - & \\
          	 & $0.001316$ & $1\ 319\ 007$ & - & - &\\
		 & - & - & $0.001102$ & $1$ &\\
\hline
\multirow{3}{*}{Schoenfeld (1976) \cite{Sch}}
		 & $0.001303$ & $1\ 155\ 901$ & - & - &\\
		 & - & - & $0.001093$ & $1$ &\\
\hline	
\multirow{1}{*}{Platt $\&$ Trudgian (2016) \cite{PT2016}} & - & - & $7.5 \cdot 10^{-7}$ & $0$ & \\
\hline
 \end{tabular}
  \label{tablek0}
 \end{table}
\end{center}

\pagebreak %This pagebreak exists for Table 1 placement

In the case $k=0$ and $X_0=1$ of Theorem \ref{MainResult:theta} we are able to  reduce the 
%\todo{In which case? I think either we need to reference Table 1 or write the case $k=0$}
 bound $1.75 \cdot 10^{-7}$  due to Platt and Trudgian (\cite{PT2016}, see Table \ref{tablek0}) to $1.94 \cdot 10^{-8}$: 
% Change here. above. 
%
\begin{corollary}
\label{cor:k0intro}
We have 
\[
 \theta(x) \le (1+ 1.93378 \cdot 10^{-8})x \qquad  \text{for all } x \ge 0
\]
\end{corollary}
\begin{proof}
We combine together the fact that $\theta(x) < x$ for $0 < x \le 10^{19}$ (\cite[Theorem 2, (1.7)]{But2})
and that $M_0 = \varepsilon(19 \log 10)= 1.93378 \cdot 10^{-8}$ (from \cref{Wedpsixvals} with $k = 0$ and $X_1 = 10^{19}$).
% , 
% we obtain Thus from \cref{Wedpsixvals}, we can see that 
%  \[
%   M_0 = 1.93378 \cdot 10^{-8} \text{ for all } x \ge 0. 
%  \]
\end{proof}
%\ \newline
%
\begin{center}
 \begin{small}
\begin{table}[h]
 \caption{Case $k=1$}
 \begin{tabular}[h]{|l|rr|rrr|}
 \hline
 Author & $m_1$ & $X_0$ & $M_1$ & $X_1$  & \\ 
 \hline
 \multirow{2}{*}{Rosser (1941) \cite{Ros} } 
				& $2.85$ & $2$ & $2.85$ & $2$ & \\
				& $0.96$ & $e^{2000}$ & $0.96$ & $e^{2000}$ & \\
 \hline
\multirow{2}{*}{Rosser $\&$ Schoenfeld (1962) \cite{RS1}   }
	                                   & $0.50$ & $563$ & $0.50$ & $1$  & \\ 
	                                   & $0.47$ & $569$ & $0.47$ & $569$ & \\ 
\hline
\multirow{2}{*}{Rosser $\&$ Schoenfeld (1975) \cite{RS2}	}
		& $0.02500$ & $678\ 407$ & $0.02500$ & $678\ 407$ & \\ 
		& $0.02424$ & $758\ 699$ & $0.02424$ & $525\ 752$ &\\
\hline
\multirow{1}{*}{Schoenfeld (1976) \cite{Sch}}
		& $0.02400$ & $758\ 711$ & $0.02400$ & $758\ 711$ &\\
		%& $0.007763$ & $e^{22}\simeq 3.584\cdot 10^{9}$ & $0.007763$ & $e^{22}\simeq 3\,584\cdot 10^{9}$ &\\
\hline      
\multirow{1}{*}{Dusart (1999) \cite{Dus}} 
               & $6.788\cdot10^{-3}$ & $10\ 544\ 111$ & $6.788\cdot10^{-3}$ & $10\ 544\ 111$ &\\
\hline	      
 \multirow{1}{*}{Dusart (2010) \cite[Table 6.4-6.5]{Dus5} }%(Arxiv 2010),}  
		      & $3.888\cdot10^{-5}$ & $e^{35}$ & $3.888\cdot10^{-5}$ & $e^{35} $ & \\
		      \hline
 \end{tabular}
 \label{tablek1}
 \end{table}
\end{small}
\end{center}
% For the values of $X_0,X_1$ up to $7\cdot10^{11}$ (like $10\ 544\ 111$, etc), comparison can be done with direct calculations as recorded in Table \ref{Table-Dk}. For instance for $X_0 =X_1=10\ 544\ 111$, $m_1=$ and $M_1=$.
%\todo{In which case? I think either we need to reference Table 2 or write the case $k=1$}
%Change here. Below.
In the case $k=1$ and $X_0=e^{35}\approx 1.586\cdot10^{15}$ of Theorem  \ref{MainResult:theta}, we reduce the bound 
$3.888\cdot10^{-5}$ due to Dusart  (\cite{Dus5}, see Table \ref{tablek1}) 
to %improve Dusart's \cite{Dus5} from  $m_1 = M_1 = 3.888\cdot10^{-5}$ to 
$1.0778 \cdot 10^{-6}$.
\pagebreak %This pagebreak exists for Table 3 Placement
 %\ \newpage
\begin{center}
 \begin{small}
\begin{table}[h]
 \caption{Case $k=2$}
 \begin{tabular}[h]{|l|rr|rrr|}
 \hline
 Author & $m_2$ & $X_0$ & $M_2$ & $X_1$  & \\ 
 \hline
 \multirow{1}{*}{Rosser and Schoenfeld (1975) \cite{RS2}}
			      & $8.6853$ & $1$ & $8.6853$ & $1$ & \\
\hline
\multirow{1}{*}{Schoenfeld (1976) \cite{Sch}}
		& $8.0720$ & $1$ & $8.0720$ & $1$ &\\
\hline      
\multirow{1}{*}{Dusart (1999) \cite{Dus}} 
	      & $0.2$ & $3\ 594\ 641$ & $0.2$ & $3\ 594\ 641$ &\\
\hline	      
 \multirow{1}{*}{Dusart (2010) \cite[Table 6.4-6.5]{Dus5} }%(Arxiv 2010),}  
		      & $0.140\cdot10^{-3}$ & $e^{35}$ & $0.140\cdot10^{-3}$ & $e^{35} $ & \\
%  \multirow{4}{*}{Dusart (2010) \cite[Theorem 5.2]{Dus5} }%(Arxiv 2010),}  
% 		      & $3.965$ & $1$ & $3.965$ & $1$ & \\
% 		      & $0.2$ & $3\ 594\ 641$ & $0.2$ & $3\ 594\ 641$ & \\
% 		      & $0.05$ & $122\ 568\ 683$ & $0.05$ & $122\ 567\ 683$ & \\
% 	 	      & $0.01$ & $7\ 713\ 133\ 853$ & $0.01$ & $7\ 713\ 133\ 853$ & \\
\hline
Trudgian (2016) \cite[Lemma 1]{Tru} 
& $0.450 \cdot 10^{-3}$ & $e^{35}$ & $0.450 \cdot 10^{-3}$ & $e^{35}$ & \\ 
\hline		      
 \end{tabular}
 \label{tablek2}
 \end{table}
\end{small}
\end{center}
% 
%Change here. 
%\todo{In which case? I think either we need to reference Table 3 or write the case $k=2$} 
Note that in the case $k=2$ and $X_0 =e^{35}$,  Dusart's calculation \cite{Dus5} of $0.140\cdot10^{-3}$ was based on assuming Wedenevski-Gourdon's verification of the Riemann Hypothesis up to height $10^{13}$ \cite{Wed}
while Trudgian's \cite{Tru} $0.450 \cdot 10^{-3}$ is ``worse'' as based on Platt's rigorous verification at the lower height of $3\cdot10^{10}$ \cite{Pla}.
We improve both of these results by obtaining $m_2 = M_2 = 5.9771 \cdot 10^{-5}$ for $X_0=X_1=e^{35}$. 
\begin{center}
 \begin{small}
\begin{table}[h]
 \caption{Case $k=3$}
 \begin{tabular}[h]{|l|rr|rrr|}
 \hline
 Author & $m_3$ & $X_0$ & $M_3$ & $X_1$  & \\ 
 \hline
 \multirow{1}{*}{Rosser and Schoenfeld (1975) \cite{RS2} }
	      & $11\ 762$ & $1$ & $11\ 762$ & $1$ & \\
\hline
\multirow{1}{*}{Schoenfeld (1976) \cite{Sch}}
		& $10\ 644$ & $1$ & $10\ 644$ & $1$ &\\
\hline      
\multirow{1}{*}{Dusart (1999) \cite{Dus}} 
	      & $515$ & $1$ & $515$ & $1$ &\\
\hline	      
%  \multirow{4}{*}{Dusart (2010) \cite[Table 6.4-6.5]{Dus5} }%(Arxiv 2010),}  
%                        & $20.83$ & $1$  & $20.83$ & $1$ & \\
% 		      & $10$ & $32\ 321$ & $10$ & $32\ 321$ & \\
% 		      & $1$ & $89\ 967\ 803$ & $1$ & $89\ 967\ 803$ & \\
% 		      & $0.78$ & $158\ 822\ 621$ & $0.78$ & $158\ 822\ 621$ & \\
 \multirow{1}{*}{Dusart (2010) \cite[Table 6.4-6.5]{Dus5} }%(Arxiv 2010),}  
		      & $0.35$ & $e^{30}$ & $0.35$ & $e^{30}$ & \\
% \hline
% Axler (2016) \cite[Proposition 3.2]{Axler16}  
%  & $0.35$ & $e^{30}$ & $0.35$ & $e^{30}$ & \\ 
% \hline
% \multirow{1}{*}{Dusart (2018) \cite[Theorem 4.2]{Dus4}}% Dusart (Ramanujan 2018) 
%  & $0.50$ & $767\ 135\ 587$ & $0.50$ & $767\ 135\ 587$ & \\
% \hline
% Axler (2018) \cite[Theorem 2.4]{Axler17_2} 
%  & $0.15$ & $19\, 035\, 709\, 163>e^{23}$ & $0.15$ & $1$ & \\ 
\hline		      
 \end{tabular}
 \label{tablek3}
 \end{table}
\end{small}
\end{center}
In the case $k=3$ and $X_0=X_1= e^{30}$ of Theorem  \ref{MainResult:theta}, we obtain $  m_3 = M_3 = 0.0244$.
Observe that this improves Dusart's \cite{Dus5} bound of 0.35 (see Table \ref{tablek3} below). 
In addition, we recover Axler's \cite[Theorem 1.1]{Axler17_2}: 
for $ x \ge 19\, 035\, 709\, 163>e^{23}$, $  m_3 = 0.15$ and for $ x \ge 1$, $ M_3 = 0.15$.
In particular,  $0.15$ is attained at the prime $p_{841\,508\,302}=19\,035\,709\,163$. 
%\todo[inline]{In which case? I think either we need to reference Table 4 or write the case $k=3$}
%\ \newline 
%
\begin{center}
 \begin{small}
\begin{table}[h]
\caption{Case $k=4$}
 \begin{tabular}[h]{|l|rr|rrr|}
 \hline
 Author  & $m_4$ & $X_0$ & $M_4$ & $X_1$  & \\ 
 \hline
 \multirow{1}{*}{Rosser and Schoenfeld (1975) \cite{RS2}}
				& $ 1.8559 \cdot 10^{7}$ & $1$ & $ 1.8559 \cdot 10^{7}$ & $1$ & \\
\hline
\multirow{1}{*}{Schoenfeld (1976) \cite{Sch}}
& $16\,570\,000$ & $1$ & $16\,570\,000$ & $1$ &\\ 
\hline      
\multirow{1}{*}{Dusart (1999) \cite{Dus}} 
& $1\, 717\, 433$ & $1$ & $1\, 717\, 433$ & $1$ &\\ 
\hline	      
 \multirow{1}{*}{Dusart (2010) \cite[Tables 6.4-6.5]{Dus5} }%(Arxiv 2010),}  
& $1\,300$ & $1$ & $1\,300$ & $1$ & \\ 
\hline
% \multirow{1}{*}{Dusart (2018) \cite[Theorem 4.2]{Dus4}}% Dusart (Ramanujan 2018) 
% & $151.3$ & $1$ & $151.3$ & $1$ & \\ 
% \hline
\end{tabular}
 \label{tablek4}
\end{table}
\end{small}
\end{center}
In the case $k=4$ and $X_0=X_1=1$ of Theorem  \ref{MainResult:theta}, we have $m_4=M_4=151.3$: 
as noticed by Dusart \cite{Dus4}, the value $151.2235\ldots$ is a max attained at the prime $1\,423$.
For $X_0=X_1=7\cdot10^{11}$, we find $m_4=M_4=57.184$. 
%
% {\commentcolor{red} VALUES FOR $X_0,X_1$.}
% The values for $X_0$ and $X_1$ are $e^{j}$ with $j$ ranging through various integers in the interval $[20,10000]$. 
% From this we obtain the following.  In the range $x \le 7.0 \cdot 10^{11}$
%
%
%
%% \begin{proof}
%%  We use  \cref{k0upper-k0implower} with $b=19\log 10$, as we have $\theta(x) < x$ for $0 < x \le 10^{19}$ from  \cite[Theorem 2, (1.7)]{But2}.
%%  Then we apply Theorem \ref{psixmainthm} with $b_0 = 19\log 10$ for $\varepsilon(19\log 10)=1.93378 \cdot 10^{-8}$ given in Table \ref{Wedpsixvals}.
%% \end{proof}
%
%% This is proven in Section \ref{thetalemmataproofs}, following directly from bounds on $\psi(x)$.
\subsection{The conjectural size of $\theta(x)$.}  
In this article, we have attempted to establish the best-known Chebyshev-type bounds for $\theta(x)$. Note that de la Vall\'{e}e Poussin's proof of the prime number theorem \cite{dlVP} actually yields
$\theta(x) = x + O (x \exp(-c (\log x)^{\frac{1}{2}}) )$
for some $c >0$.
% Furthermore, Dusart \cite[Theorem 1.1]{Dus3} explicitly showed that if there is a zero-free region of the form $\Re(s) \ge 1 - \frac{1}{R \log| \Im(s)|}$ 
%  for some $R > 0$, then
%\[
% |\theta(x) -x| \le \sqrt{\frac{8}{\pi R^{\frac12}}} x (\log x)^{\frac14} \exp\left(-\sqrt{\frac{\log x}{R}}\right)
% \ \text{for all}\ x \ge \exp\left( R \left( \max\! \left( 8.36,\frac{8}{R} \right) \right)^2 \right).
%\]
Furthermore, Platt and Trudgian \cite{PT2019} have established for $x_0 \ge 1000$, 
there exist positive constants $A,B,C$ such that for all $x \ge e^{x_0}$
\begin{equation}
    |\theta(x) -x| \le  A ( \tfrac{\log x}{R} )^{B} \exp\left( -C  (\tfrac{\log x}{R})^{\frac{1}{2}} \right).
\end{equation}
%We should also mention what are the actual conjectured sizes of these functions.
Recently, B\"{u}the \cite{But} has shown that under partial RH (true when $1\le |\gamma|\le T$), then  
\begin{equation}
 \label{thetaRH}
 |\theta(x) -x| \le \frac{1}{8 \pi} \sqrt{x} (\log x)^2 
\end{equation}
for all $x\ge 599$ satisfying $4.92 (\frac{x}{\log x})^{\frac{1}{2}} \le T$. 
For instance, for Platt's $T = 3.061 \cdot 10^{10}$, \eqref{thetaRH} holds for $599 \le x \le 1.89 \cdot 10^{21}$. 
Conditionally on RH being true, Schoenfeld \cite[Theorem 10]{Sch} has shown that \eqref{thetaRH} holds for all $x \ge 599$.
% On the Riemann hypothesis, it has been established by Schoenfeld \cite[Theorem 10]{Sch} that 
% \begin{equation}
%   \label{thetaRH}
%   |\theta(x) -x| \le \frac{1}{8 \pi} \sqrt{x} \log^2 x \text{ for all } x \ge 599. 
% \end{equation}
These are effective versions of a theorem of von Koch \cite{Koch}. To date, the constant $\frac{1}{8 \pi}$ has not been improved. 
% Recently, B\"{u}the \cite{But} has shown that under partial RH (for $1\le |\gamma|\le T$), then ...
% \\the condition that $4.92(\frac{x}{\log x})^{\frac{1}{2}} \le H_0$, \eqref{thetaRH} holds for $x \le 1.89 \cdot 10^{21}$ if we take Platt's $H_0=3.061 \cdot 10^{10}$; 
% %for $x \le 2.1\cdot10^{22}$ if we take J. Franke, Th. Kleinjung, J. B\"uthe, and A. Jost's $H_0=10^{11}$; 
% and finally for  $x \le 1.4 \cdot 10^{25}$
% if we take Wedeniwski's  $H_0=2.445 \cdot 10^{12}$.
% % In particular the numerical verification in [Pla15] ($T \approx 3.061 \cdot 10^{10}$ ) gives these bounds for 
% % $x \le 1.89 \times 10^{21}$, the result in [FKBJ] ($T = 10^{11}$) gives them for $x \le 2.1?10^{22}$ and the result in [Gou04] ($T \approx 2.445 \times10^{12}$) gives them for $x \le 1.4 \times 10^{25}$.
% \\
% Furthermore, Vinogradov \cite{Vin} and Korobov's \cite{Kor} improvements to the zero-free region, yield 
% \begin{equation}
%  \label{thetavino}
%   \theta(x) = x + O \Big( x  \exp \Big( -c \frac{(\log x)^{\frac{2}{3}}}{ (\log \log x)^{\frac13} }  \Big) \Big).
% \end{equation}
% {\commentcolor{red} Explicit versions of \eqref{thetadlvp} and \eqref{thetavino} (see Ford 2001?). 
% % (Gourdon's??) 
% }
% {\commentcolor{red} FIX WORDING. REWRITE.}
It may be asked, what is the true size of the error term on the 
right hand side of \eqref{thetaRH}. The explicit formula of Riemann tells us that on the Riemann hypothesis
\[
   \frac{\theta(x)-x}{\sqrt{x}} = -1 -  2 \Re \Big( \sum_{\gamma >0} \frac{x^{i \gamma}}{\frac{1}{2}+i \gamma} 
   \Big) + \cdots 
\]
where $\frac{1}{2}+i \gamma$ ranges through the non-trivial zeros of zeta. 
It is known under the Linear Independence Hypothesis (LI) \footnote{LI is the conjecture that the positive ordinates 
of the zeros of $\zeta(s)$ are linearly independent over $\mathbb{Q}$.}
that the distribution of values of $(\theta(x)-x)/\sqrt{x}$ is the same as that of the 
%Moreover, it may be shown that $(\theta(e^y)-e^y) e^{-y/2}$ possesses a limiting distribution $\mu$ (a probability measure on $\mathbb{R}$).   Under the assumption that the positive ordinates of the zeros of the zeta function are linearly independent over $\mathbb{Q}$, it may be shown that $\mu(B) = P(X \in B)$ where $X$ is 
random variable
$
  X = 2 \Re \big( \sum_{ \gamma >0}  \frac{ X_{\gamma}}{|\frac{1}{2}+i \gamma|} \big)$
where the $X_{\gamma}$ are independent random variables, uniformly distributed on the unit circle. 
By giving sharp estimates for the probability of the tail of $X$, it may be shown 
that $\exp(-c_2\sqrt{V} e^{\sqrt{2 \pi V}}) \le P(x \ge V) \le \exp(-c_1 \sqrt{V} e^{\sqrt{2 \pi V}})$, for some $c_1, c_2 > 0$.
 This suggests that 
\[
%    \label{thetaconjecture} 
   \limsup_{x \to \infty} \frac{\theta(x)-x}{\sqrt{x} (\log \log x)^2} = \frac{1}{2 \pi}
   \text{ and }
    \liminf_{x \to \infty} \frac{\theta(x)-x}{\sqrt{x} (\log \log x)^2} = -\frac{1}{2 \pi}.
\]
This conjecture in the case of $\psi(x)$ is due to Montgomery \cite{Mont}.  
\section{Bounding $\psi(x)-\theta(x)$.}
\label{psithetaresults}
In this section we give bounds for $\psi(x)-\theta(x)$ of the shape
\begin{equation}
  \label{psithetageneral}
  \psi(x)-\theta(x) \le a_1 x^{\frac{1}{2}} + a_2 x^{\frac{1}{3}} \text{ for all } x \ge x_0
\end{equation}
where $a_1$ and $a_2$ depend on $x_0$. 
Rosser and Schoenfeld \cite[Theorem 6]{RS2} established this with $a_1=1.001102$, $a_2=3$, 
and $x_0 = 1$.
Recently, Dusart \cite[Corollary 4.5]{Dus4} improved this bound to $a_1 = 1 + 1.47 \cdot 10^{-7}$ and $a_2 = 1.78$ for $x_0=1$.
% As we are interested in applying such a bound for various other values of $x_0$, $a_1$ and $a_2$ can be reduced.
As we apply such a bound to various other values of $x_0$, we are able to reduce the values of $a_1$ and $a_2$.
\begin{proposition}%\cite[Proposition 4.4]{Dus4}
\label{Prop:7}
%  For $x > 0$ we have
% \begin{equation}
%  \psi(x) - \theta (x) - \theta (x^{\frac12}) < 1.777745 x^{\frac13}.
% \end{equation}
% Let $j \in \mathbb{Z}_{\ge 0}$.  
% Let $x_0 \ge 2^{10}$. 
Let $x_0 \ge 2^{9}$. Let $\a > 0$ exist such that $\theta(x) \le (1+ \a)x$ for $x > 0$. Then for $x \ge x_0$, 
% \[
%  \psi(x) - \theta (x) - \theta (x^{\frac12}) < \delta x^{\frac13}
% \]
\begin{equation}
\label{Prop9:equ1}
 \sum_{k=3}^{\big\lfloor \frac{\log x}{\log 2} \big\rfloor} \theta(x^{\frac{1}{k}}) \le \eta x^{\frac13}
\end{equation}
where 
\begin{equation}
 \label{defn:eta}
  \eta = \eta(x_0) =(1+\a) \max \bigl(  f(x_0), f(2^{\big\lfloor \frac{\log x_0}{\log 2} \big\rfloor+1})   \bigr)
\end{equation}
with %$b_0 = \lfloor \frac{\log x_0}{\log 2} \rfloor$ and
%\footnote{ $\lceil t \rceil$  denotes the least integer greater than or equal to $t$.}, and
\begin{equation} 
  \label{fndefn}
 f(x) := \sum_{k=3}^{\big\lfloor \frac{\log x}{\log 2} \big\rfloor} x^{\frac{1}{k}-\frac{1}{3}}.
\end{equation}
\end{proposition}
\begin{proof}
Let $x \ge x_0 \ge 2^{9}$. 
% We have  $\psi(x) - \theta(x) -  \theta (x^{\frac12})  = 
% \sum_{k=3}^{\big\lfloor \frac{\log x}{\log 2} \big\rfloor} \theta (x^{\frac{1}{k}})$.
% Observe that this sum vanishes for $x<8$ and thus we may assume $x \ge 8$. 
% Now $j \in \mathbb{Z}_{\ge 0}$ and  assume $x \ge e^j$.  
Bounding each $\theta(x^{\frac1k})$ term of \eqref{Prop9:equ1} by $(1+\a)x^{\frac1k}$ yields
\[
 \frac{\psi(x) - \theta(x) -  \theta (x^{\frac12})}{x^{\frac13}} \le (1+\a) \sum_{k=3}^{\big\lfloor \frac{\log x}{\log 2} \big\rfloor}
 x^{\frac{1}{k} - \frac{1}{3}}. 
\]
% , so that 
% \begin{align*}
%   R(x) & \le 1.000081 \sum_{k=3}^{L(x)} x^{\frac{1}{k}} = 1.000081 x^{\frac{1}{3}}  \Big(1 + \sum_{k=4}^{L(x)} x^{\frac{1}{k}-\frac{1}{3}} \Big). 
% \end{align*}
% Trivially, we have 
% \[
%    R(x) \le 1.000081 x^{\frac{1}{3}}  \Big(1 + \frac{\frac{\log x}{\log 2} -3}{x^{\frac{1}{12}}} \Big).  
% \]
% Since the expression in brackets is decreasing with $x$, we obtain for $x \ge e^{1000}$
% \[
%    R(x)  \le 1.000081 c_0 x^{\frac{1}{3}} = 1.000081 (1+9.26966\cdot10^{-34})x^{\frac13} = 1.000082x^{\frac13} 
% \]
Next, we divide the interval $[x_0, \infty)$ as follows.
If $x_0$ is not a power of $2$, then $2^ {\big\lfloor \frac{\log x_0}{\log 2} \big\rfloor + 1}$ is the least power of 2 in $[x_0, \infty)$.
Thus,  we have 
\[
  [x_0, \infty) = [x_0,2^{\big\lfloor \frac{\log x_0}{\log 2} \big\rfloor+1}) \cup \bigcup_{n=\big\lfloor \frac{\log x_0}{\log 2} \big\rfloor+1}^{\infty} [2^n,2^{n+1}).
\]
%If $x \in [2^n, 2^{n+1})$,  then $L(x)=n$ and $R(x) = f_n(x)$. 
% Let $x \in [2^n, 2^{n+1})$ for arbitrary fixed $n$. 
Observe that $f(x)$ decreases on $[2^n, 2^{n+1})$ and thus $f(x) \le f(2^n)$ for every $x \in [2^n, 2^{n+1})$.
Note that $f(2^n) = 1+u_n$ where $u_n =  \sum_{k=4}^n 2^{\frac{n}{k} -\frac{n}{3}}$. 
% We now show that $u_n \le u_{n-1}$ for $n \ge 10$.
We now show that $u_{n+1} < u_n$ for $n \ge 9$.
% that is,  $u_n-u_{n-1} \le 0$ for $n \ge 10$. 
We have 
% \begin{equation}
%   \label{undifference}
%  u_n-u_{n-1}
% % &= \sum_{k=4}^{n-1} \Bigl( 2^{\frac{n}{k}-\frac{n}{3}} - 2^{\frac{n-1}{k}-\frac{n-1}{3}} \Bigr) + 2^{\frac{n}{n}-\frac{n}{3}} \\
%  = \sum_{k=4}^{n-1} 2^{\frac{n}{k}-\frac{n}{3}} (1-2^{\frac13 -\frac{1}{k}} ) + 2^{1-\frac{n}{3}} 
%  = 2^{-\frac{n}{3}} \Big( 2- \sum_{k=4}^{n-1} 2^{\frac{n}{k}} ( 2^{\frac13 - \frac1k} -1 )  \Big).
% \end{equation}
\begin{equation}
\label{undifference}
 u_{n+1} - u_n
 = \sum_{k=4}^{n} 2^{\frac{n+1}{k}-\frac{n+1}{3}} (1-2^{\frac13 -\frac{1}{k}} ) + 2^{1-\frac{n+1}{3}}
  = 2^{-\frac{n+1}{3}} \Big( 2- \sum_{k=4}^{n} 2^{\frac{n+1}{k}} ( 2^{\frac13 - \frac1k} -1 )  \Big).
\end{equation}
%Thus it suffices to show that the sum in brackets is $\ge 2$, for $n \ge 10$.
%Calling $\varphi_n(t) =  2^{\frac{n+1}{t}}(2^{\frac13-\frac{1}{t}}-1)$, a straightforward calculus calculation shows that 
%$\varphi_n(t)$ decreases with $t\ge4$ as long as $n \ge 17$. 
Observe that if $n \ge 20$, then 
\[
  \sum_{k=4}^{n} 2^{\frac{n+1}{k}} ( 2^{\frac13 - \frac1k} -1 )
  > 2^{\frac{n+1}{4}} (2^{\frac{1}{3}-\frac{1}{4}}-1) \ge 2^{\frac{21}{4}} (2^{\frac{1}{12}}-1)
  > 2
\]
and it follows that $u_{n+1}-u_n < 0$ for $n \ge 20$.  Finally, a numerical calculation verifies that the right hand side of \eqref{undifference}  is negative for $9 \le n \le 19$.  Therefore it follows that $f(2^n) > f(2^{n+1})$ for $n \ge 9$. 
Therefore $f(x) \le f(2^{\big\lfloor \frac{\log x_0}{\log 2} \big\rfloor+1})$ on $[2^{\big\lfloor \frac{\log x_0}{\log 2} \big\rfloor+1},\infty)$. 
Similarly, we see that $f(x) \le f(x_0)$ for $x \in [x_0,2^{\big\lfloor \frac{\log x_0}{\log 2} \big\rfloor+1})$, since $f(x)$ decreases on
 $[2^{\big\lfloor \frac{\log x_0}{\log 2} \big\rfloor}, 2^{\big\lfloor \frac{\log x_0}{\log 2} \big\rfloor+1})$.
%  and $x_0 \in [2^{\big\lfloor \frac{\log x_0}{\log 2} \big\rfloor},
% 2^{\big\lfloor \frac{\log x_0}{\log 2} \big\rfloor+1})$.
In summary,  $
  f(x) \le \max \bigl(  f(x_0), f(2^{\big\lfloor \frac{\log x_0}{\log 2} \big\rfloor+1})   \bigr)$.
% By a computer calculation with Maple we obtain the values for $\delta_j$ in the table. 
%From these calculations, it seems that 
%$f(x_0) > f(2^{j_0})$ for $j \le 100$.
\end{proof}

We now apply the case where $x_0 =e^{b}$ to obtain the following corollary.
\begin{corollary}
  \label{deltabcorr}
 Let $b \ge 7$. Assume $x \ge e^b$. Then we have
 \[
  \psi(x) - \theta(x) - \theta(x^{\frac12}) \le \eta x^{\frac13}
 \]
where
\begin{equation}
  \label{deltab}
  \eta = (1+1.93378 \cdot 10^{-8}) \max \Big( f(e^b) , f(2^{\big\lfloor \frac{b}{\log 2} \big\rfloor+1})  \Big)
%   b_0  = \Big\lfloor \frac{b}{\log 2} \Big\rfloor.
\end{equation}
and $f$ is defined by \eqref{fndefn}.
\end{corollary}
\begin{proof}
 We apply Proposition \ref{Prop:7} with $\a=1.93378 \cdot 10^{-8}$ from Corollary \ref{cor:k0intro}
 and $x_0 =e^b$ where we observe that $x_0=e^b \ge e^7 >2^9$. 
 %   from Corollary \ref{k0expupper} 
\end{proof}

The next result is a general version of \cite[Theorem 6, eq (5.3)]{RS2} and \cite[Corollary 4.5]{Dus4}.
The two inputs we take are a Chebyshev Bias constant $x_1$ such that $\theta(x) < x$ for $x \le x_1$ and 
a bound for $|\psi(x)-x|$ for $x \ge y_0$ for every $y_0 >0$.  These are used in conjunction with \cref{deltabcorr} to prove Theorem \ref{psi-theta:GeneralThm}.
\begin{proposition}\label{psi-theta:Prop2}
Let $b \ge 7$ and assume that for fixed $b$, there exists a positive constant $\varepsilon(b)$ such that 
\begin{equation}
  \label{psibound}
 |\psi(x)-x| \le \varepsilon(b) x \text{ for all } x \ge e^b. 
\end{equation}
Assume there exists $x_1  \ge e^{7}$ such that 
\[
%   \label{x1defn}
  \theta(x) < x \text{ for all } x \le x_1. 
\]
If $b \le 2 \log x_1$, then we have
\begin{equation}
  \label{thetasqrtxbd1}
 \theta(x^\frac12) < (1+ \varepsilon(\log x_1))x^{\frac12} \text{ for }  x \ge e^b. 
\end{equation}
If $b > 2 \log x_1$, then we have 
\[
 \theta(x^\frac12) < (1+\varepsilon(b/2))x^{\frac12}  \text{ for }  x \ge e^b. 
\]
\end{proposition}

\begin{proof}
We bound $\theta(x^{\frac12})$ by cases depending on the range of $x$ we are considering. \\
{\it Case 1: $e^b \le  x_1^2$}.   If $e^b \le x \le x_1^2$, then  $x^{\frac12} \le x_1$, and thus 
\[
 \theta(x^\frac12) < x^{\frac12} \text{ for }e^b \le x \le x_1^2. 
\]
On the other hand, if $x^{\frac12} > x_1=e^{\log x_1}$, then we have by \eqref{psibound}
\[
 \theta(x^\frac12) \le  \psi(x^{\frac12}) \le (1 + \varepsilon(\log x_1)) x^{\frac12}, 
\]
since $\log x_1 \ge 7$.  The last two inequalities for $\theta(x^{\frac{1}{2}})$ combine to establish \eqref{thetasqrtxbd1}.  \\
{\it Case 2: $e^b  >  x_1^2$}.  As in the above subcase, we have for $x \ge e^b$
%
%bound $\theta(x^\frac12)$ by $\psi(x^\frac12)$, subsequently bounding $\psi(x^\frac12)$ by 
% $(1+\varepsilon(b/2))x^{\frac12}$
% since $x^\frac12 > e^{\frac{b}{2}} > x_1 \ge e^7$. 
% Thus we have
\[
 \theta(x^\frac12) 
 \le \psi(x^\frac12) \le 
  (1 +\varepsilon(b/2))x^{\frac12}, 
\]
since  $x^\frac12 > e^{\frac{b}{2}} > x_1 \ge e^7$.
\end{proof}
Using the previous general results we obtain the following explicit bounds for $\psi(x)-\theta(x)$ of the shape
\eqref{psithetageneral}.  This generalizes Rosser and Schoenfeld's \cite[Theorem 6]{RS2} and Dusart's \cite[Corollary 4.5]{Dus4} results.
% Note that either of the previous results can be used without the other if desired.
%
\begin{theorem}\label{psi-theta:GeneralThm}
Let $\alpha > 0$ exist such that 
%  Let $b$ and $\alpha$ exist such that
 \[
  \theta(x) \le (1+\alpha)x \text{ for all }x > 0.
 \]
 Assume for every $b \ge 7$ there exists a positive constant $\varepsilon(b)$ such that 
%  and for every $b$ there exists a $\varepsilon(b)$ such that 
 \[
  \psi(x) - x \le \varepsilon(b)x \text{ for all } x \ge e^b.
 \]
 Assume there exists $x_1 \ge e^7$ such that 
 \begin{equation}
  \label{defn:x1}
  \theta(x) < x \text{ for }x \le x_1.
 \end{equation}
 Let $b \ge 7$.  
 Then, for all $x\ge e^b$ we have 
 \begin{equation*}
 \psi(x) - \theta(x) < a_1 x^{\frac12} + a_2 x^{\frac13},
 \end{equation*}
 where
 \begin{equation*}
 a_1 = 
 \begin{cases}
  1+ \varepsilon(\log x_1) &\ \text{if}\ b \le 2\log x_1,\\
  1 + \varepsilon(b/2) &\ \text{if}\ b > 2\log x_1,
 \end{cases}
 \end{equation*}
 and
 \begin{equation*}
  a_2 = (1+\a) \max \Big( f(e^b) , f(2^{\lfloor \frac{b}{\log 2} \rfloor+1})  \Big).
 \end{equation*}
\end{theorem}
\begin{proof}
 We have $\psi(x) - \theta(x) = \theta (x^{\frac12}) + \sum_{k=3}^{\lfloor \frac{\log x}{\log 2} \rfloor} \theta (x^{\frac{1}{k}})$. For any $b \ge 7$, 
  setting $x_0 = e^b$ in Proposition \ref{psi-theta:Prop2}, we bound $\sum_{k=3}^{\big\lfloor \frac{\log x}{\log 2} \big\rfloor} \theta (x^{\frac{1}{k}})$ by
 $\eta x^\frac13$ as defined in \eqref{defn:eta}. We bound $\theta (x^{\frac12})$ with Proposition \ref{psi-theta:Prop2} by taking either 
 $a_1 = 1+ \varepsilon(\log x_1)$ for $b \le 2\log x_1$ or $a_1 = 1 + \varepsilon(b/2)$ for $b > 2\log x_1$.
%   for
%  \begin{align*}
%   \psi(x) - \theta(x) &= \theta (x^{\frac12}) + \sum_{k=3}^{\big\lfloor \frac{\log x}{\log 2} \big\rfloor} \theta (x^{\frac{1}{k}}) \\
%   \psi(x) - \theta(x) &< a_1 x^\frac12 + \delta x^\frac13
%  \end{align*}
\end{proof}

Using the best known values for $x_1$, $\alpha$ and $\varepsilon(b)$, we have the following Corollary:
\begin{corollary}\label{psi-theta:ExplicitCor}
 Let $b \ge 7$. Then for all $x\ge e^b$ we have 
 \begin{equation}
 \psi(x) - \theta(x) < a_1 x^{\frac12} + a_2 x^{\frac13},
 \end{equation}
 where 
 \begin{equation}
 \label{def-a1}
 a_1 = a_1(b)=
 \begin{cases}
 1+ 1.93378 \cdot 10^{-8} &\ \text{if}\ b \le 38\log 10,\\
 1 + \varepsilon(b/2) &\ \text{if}\ b >   38\log 10,
 \end{cases}
 \end{equation}
 \begin{equation}\label{def-a2}
 a_2 = a_2(b)= (1+1.93378 \cdot 10^{-8}) \max \Big( f(e^b) , f(2^{\lfloor \frac{b}{\log 2} \rfloor+1})  \Big),
 \end{equation}
where $f$ is defined by \eqref{fndefn} and values for $\varepsilon(b/2)$ are from Table \ref{Wedpsixvals}.
\end{corollary}
\begin{proof}
 We apply Theorem \ref{psi-theta:GeneralThm} with $x_0 = e^b$ and $\a=1.93378 \cdot 10^{-8}$ from Corollary \ref{cor:k0intro}. Thus we get 
  $a_2 =  (1+1.15177 \cdot 10^{-8}) \max \Big( f(e^b), f(2^{\big\lfloor \frac{b}{\log 2} \big\rfloor+1}) \Big)$. In Proposition \ref{psi-theta:Prop2}, we take 
  $x_1=10^{19}$ from the work of B\"uthe \cite[Equation (1.7)]{But2}, and get $\varepsilon(\log x_1) = 1.93378 \cdot 10^{-8}$ from Table \ref{Wedpsixvals}.
  Thus for $b \le 38\log 10$ we have $a_1 = 1+1.93378 \cdot 10^{-8}$ and for $b > 38\log 10$ we take $\varepsilon(b/2)$ from Table \ref{Wedpsixvals}.
  In the case that $b/2$ is not included in the table, we bound $b/2$ by the greatest value smaller than $b/2$.
\end{proof}

We have the following values for $a_2$.
% {\color{red} Updated Aug 27 - Sam}

\begin{center}
\begin{tabular}{|*{8}{c|}}
 \hline
 $b$        & $20$     & $25$     & $30$     & $35$      & $40$      & $43$      & $50$ \\ \hline
 $a_2$ & $1.4263$ & $1.2196$ & $1.1211$ & $1.07086$ & $1.04320$ & $1.03253$ & $1.01718$ \\ \hline
\end{tabular}
\\
\begin{tabular}{|*{6}{c|}}
 \hline
 $b$        & $100$                     & $150$                     & $200$                     & $250$                     & $300$ \\ \hline
 $a_2$ & $1 + 2.421 \cdot 10^{-4}$ & $1 + 3.749 \cdot 10^{-6}$ & $1 + 7.712 \cdot 10^{-8}$ & $1 + 2.024 \cdot 10^{-8}$ & $1 + 1.936 \cdot 10^{-8}$ \\ \hline
\end{tabular}
\end{center}

% For various values of $b$, we have the following values for $a_2$.
% \begin{center}
% \begin{tabular}{|*{8}{c|}}
%  \hline
%  $b$   & $20$     & $25$     & $30$     & $35$      & $40$      & $43$      & $50$ \\ \hline
%  $a_2$ & $1.4263$ & $1.2196$ & $1.1211$ & $1.07086$ & $1.04320$ & $1.03252$ & $1.01718$ \\ \hline
% \end{tabular}
% \\
% \begin{tabular}{|*{6}{c|}}
%  \hline
%  $b$   & $100$                     & $150$                     & $200$                     & $250$                     & $300$ \\ \hline
%  $a_2$ & $1 + 2.421 \cdot 10^{-4}$ & $1 + 3.740 \cdot 10^{-6}$ & $1 + 6.930 \cdot 10^{-8}$ & $1 + 1.242 \cdot 10^{-8}$ & $1 + 1.154 \cdot 10^{-8}$ \\ \hline
% \end{tabular}
% \end{center}
% In addition to the computed values above, we have numerically verified the following bounds, as in these cases, $\delta_b \ge 1.7778$
% \begin{center}
% \begin{tabular}{|*{4}{c|}}
%  \hline
%  $b$ & $5$ & $10$ & $15$ \\ \hline
%  ${\delta'_b}$ & $1.7778$ & $1.7395$ & $1.4733$  \\ \hline
% \end{tabular}
% \end{center}
\section{Bounds for $\theta(x)-x$ of the form $\frac{x}{(\log x)^k}$}\label{thetalemmataproofs}
%%
% We prove {\commentcolor{red} \quad Include conditions on all Lemmas?}
% \begin{theorem}\label{GeneralAxler}
%  Let constants exist such that the conditions of Lemmas \ref{lemma1}, \ref{lemma2}, \ref{trivialupper}, \ref{LemmaxN_2}, \ref{But2Lem1}, and
%  \ref{lem} are met, then if $[X_0,\infty) = [G',G'')\cup[E',E'']\cup[D',D'']\cup[B',B'']\cup[A',\infty)$ and $[X_1,\infty) = (C',C'']\cup[B',B'']\cup[A',\infty)$,
%  we have 
%  \begin{align}
%  && \left( 1- \frac{m_k}{\log^kx} \right) x < \ &\theta(x) &&  \text{for all } x \ge X_0 \\
%  &&  &\theta (x) < \left( 1 + \frac{M_k}{\log^k x} \right)x &&  \text{for all }x > X_1 
%  \end{align}
%  for $k \in \N$ where $m_k > \max{(A_k, B_k, D_k, E_k, G_k)}$, $M_k > \max{(A_k, B_k, C_k)}$ and \\
%  $[A',\infty),[B',B''],(C',C''],[D',D''],[E',E''],[G',G'')$ are the ranges acquired respectively from Lemmas \ref{lemma1},
%  \ref{lemma2}, \ref{trivialupper}, \ref{LemmaxN_2}, \ref{But2Lem1}, and Section \ref{Lowernumeric}.
% \end{theorem}
\noindent
In this section we prove Theorem \ref{MainResult:theta}, our main theorem for giving estimates for $\theta(x)-x$ of size $\frac{x}{(\log x)^k}$. 
More precisely, we prove that for any $k =0,\ldots, 5 $ and any $X_0, X_1 \ge 1$, there exists $m_k, M_k>0$ such that 
\begin{equation}
 \label{thetalb}
 x \left( 1- \frac{m_k}{(\log x)^k} \right) \le \theta(x) \qquad \text{ for all } x \ge X_0 
\end{equation}
and 
\begin{equation}
 \label{thetaub}
 \theta(x) \le x \left( 1+ \frac{M_k}{(\log x)^k} \right) \qquad \text{ for all } x \ge X_1.
\end{equation}
% \begin{align}
%   && \left( 1- \frac{m_k}{(\log x)^k} \right) x < \ &\theta(x) &&  \text{for all } x > X_0  \label{thetalb} \\
%   && &\theta(x) < \left( 1 + \frac{M_k}{(\log x)^k} \right)x &&  \text{for all } x > X_1 \label{thetaub}
%  \end{align}
% \begin{theorem}\label{GeneralAxler}
%  Let $k \in \N$ and fix $X_0, X_1  \ge 1$. Then there exists $m_k(X_0), M_k(X_1) > 0$ such that
%  \begin{align}
%   && \left( 1- \frac{m_k}{\log^k x} \right) x < \ &\theta(x) &&  \text{for all } x > X_0  \label{thetalb} \\
%   && &\theta(x) < \left( 1 + \frac{M_k}{\log^k x} \right)x &&  \text{for all } x > X_1 \label{thetaub}
%  \end{align}
% \end{theorem}
The values for $X_0, X_1, m_k$, and $M_k$ may be found in Table \ref{ThetaEpsilon} for $k = 0$, and in Table \ref{Table:MasterResults} for $k=1,\ldots,5$.
Theorem \ref{MainResult:theta} is a generalization of Axler's \cite[Theorem 1]{Axler17_2}. 
% 
% \subsection{Theorem 2 for $k=1,\ldots,5$}\label{Section:knotzero}
% \quad \\
% Our strategy for proving Theorem \ref{MainResult:theta} for $k \in \lbrace1,2,3,4,5\rbrace$ is to follow and generalize the proof of 
%   Axler \cite[Theorem 1]{Axler17_2}.
% Our strategy for proving Theorem \ref{MainResult:theta} for $k \in \lbrace1,2,3,4,5\rbrace$ is to follow the result of Axler \cite{Axler17_2}
% which gives the values {\commentcolor{red}the last row of the above Table.}
%We prove a lower bound on $\theta(x)$ for $x \ge X_0$ and an upper bound for $x \ge X_1$.
We separate the cases $k=0$ and $k =1,\ldots, 5 $ (with $k=0$ being treated at the end of this section).
%For $k=0$, ???\footnote{comment for k=0}
For $k =1,\ldots, 5 $, we subdivide the interval $[1, \infty)$ (for the range of  $x$) as follows: 
\[
  [1,\infty) = [1,e^{J_0}) \cup[e^{J_0},e^{J}) \cup [e^J, e^K) \cup [e^{K}, \infty).
\]
We now explain how the values of $J_0$, $J$ and $K$ are chosen. 
%depending on certain explicit results.
For shorthand, we respectively call $[1,e^{J})$, $[e^J, e^K)$, and $[e^{K}, \infty)$ the ``small'', ``middle'', and ``large'' ranges of $x$.
\begin{itemize}
 \item In the large range of $x$, we apply bounds for $\theta(x)$ of the shape $x (\log x)^{c}  \exp\left(-C \sqrt{\log x} \right).$
 \item 
 % In the large range of $x$, our bounds follow from a classical de la Vall\'{e}e-Poussin bound on $\theta(x)$.
% \eqref{thetalb} and \eqref{thetaub}.
In the middle range of $x$, we subdivide the interval $[e^J, e^K)$ into smaller consecutive intervals $[e^{b},e^{b'})$. 
In each such subinterval, 
we make use of bounds for $\psi(x)$ of the shape $|\psi(x)-x| \le \varepsilon x$ and for the difference $\psi(x)-\theta(x)$ (established respectively in Appendix \ref{psixboundssection} and Section \ref{psithetaresults}).
% $\psi(x)-\theta(x)$ from Section \ref{psithetaresults} and the bounds
% for $\psi(x)$ from Section \ref{psixboundssection} and  from \cite{But}.
By numerical experimentation we choose
\[K=25\,000.\]
%{\color{red} IS THIS STILL TRUE? (This value of $K$ actually depends on Mossinghoff-Trudgian \cite{TrudMoss}'s zero-free region and on Wedenevski \cite{Wed}'s partial verifications of the Riemann Hypothesis.)}
\item In the small range of $x$, upper bounds for $\theta(x)$ are the result of direct calculations, 
namely that $\theta(x) < x$ is known for all $x<e^J$. 
Here B\"uthe's \cite[Theorem 2]{But2} allows us to take 
\[J=19\log 10=43.74\ldots.\]
% easy to obtain due to the Chebyshev Bias phenomenon \cite{RubSar}. % which asserts that $\theta(x) < x$ 
% for $x < x_0$ where $x_0 \approx e^{316}$ \cite{BH}, \cite{PT2016}.
\item In the small range of $x$, we subdivide $[ 1, e^J )$ at $e^{J_0}$ to obtain lower bounds for $\theta(x)$:
\begin{itemize}
% $[X_0,e^J) = [X_0,e^{J_0}) \cup [ e^{J_0}, e^J )$ with $J_0=27$.
\item We do direct calculations up to $e^{J_0}$ where we use
\[J_0 = \log(7 \cdot 10^{11}).\]
\item For $x \in [ e^{J_0}, e^J )$ we use a little known comparison of $\psi(x)$ with $\theta(x)$ due to Costa Pereira \cite{Cos},
together with numerical bounds for $(\psi(x)-x)/\sqrt{x}$, computed by B\"{u}the \cite{But2}. 
\end{itemize}
\end{itemize}
\subsection{Upper and lower bounds for $\theta(x)$ in the large range $x \ge  e^{K}$} \label{theta:xlarge}
% ($x \in [A', \infty)$)}
%%%%%%%%%%%%%%%%%%%%%%%%%%%%%%%%%%%%%%%%%%%%%%%%%%%%%%%%%%%
% For $x \ge e^{K}$, bounds for $\theta(x)$ are derived from the large $x$ bounds for
% $\theta(x)$ which arise from the classical zero-free region.
The following lemma derives a bound of the form $x/(\log x)^k$ for $\theta(x)-x$ 
from the classical de la Vall\'{e}e Poussin-Hadamard bound. 
\begin{lemma}\label{lemma1}
Suppose there exists $c_1, c_2, c_3, c_4 > 0$ such that
  \begin{equation}
\label{equ:1}
%  && |\theta(x) - x| < c_1 (\log x)^{c_2}x \exp (-c_3 \sqrt{\log x} ) &&  \text{for all } x \ge c_4
  |\theta(x) - x| \le  c_1 x (\log x)^{c_2}  \exp (-c_3 (\log x)^{\frac{1}{2}} ) \ \text{for all}\ x \ge c_4.   
  \end{equation}
% Let $k \in [1,5]$ and let $x_2 \ge \max \left(c_4, \exp ( \tfrac{4(c_2+k)^2}{c_3^2} )\right)$.
Let $k >0$ and let $b  \ge \max ( \log c_4, \log ( \tfrac{4(c_2+k)^2}{c_3^2} ) )$.
Then for all $x \ge e^b $ we have
\begin{equation}
|\theta(x) - x| \le \frac{A_k(b ) x}{(\log x)^k}, 
\end{equation}
where 
\begin{equation}
  \label{Akdefn}
 A_k(b ) = c_1 \cdot   b ^{c_2+k} e^{-c_3 \sqrt{b }}.
\end{equation}
\end{lemma}
\begin{proof}
We denote $g(x) = (\log x)^{c_2+k} \exp (-c_3 (\log x )^{\frac{1}{2}} )$. 
By (\ref{equ:1}), $ |\theta(x) - x| < \frac{c_1 g(x) x}{(\log x)^k}$ for all $x \ge c_4$.
It suffices to bound $g$: by calculus, $g(x)$ 
% has critical points at $x = 1$ and $x = \exp \left( \frac{4(c_2+k)^2}{c_3^2} \right)$. It follows that $g$ 
decreases when $x\ge  \tfrac{4(c_2+k)^2}{c_3^2}$. 
Therefore  $ |\theta(x) - x| \le  \frac{c_1 g(e^b ) x}{(\log x)^k}$.
Note that $c_1g(e^b) =  A_k(b )$ and the condition on $b$ follows from the conditions $e^b \ge c_4$
and $e^b \ge \tfrac{4(c_2+k)^2}{c_3^2}$. 

%, completing the proof.
\end{proof}
The current best explicit version of \eqref{equ:1} is due to Platt and Trudgian. Details are given 
in Corollary \ref{thetaxlargeABC} in Appendix \ref{psixboundssection}.
\begin{theorem}\cite[Theorem 1]{PT2019} \label{thetaxlarge}
Let $x_0 \ge 1000$ and 
let $R$ be a formal constant such that there exists a zero-free region of the form $\Re (s) \ge 1 - \frac{1}{R  \log | \Im (s)|}$ for $|\Im (s)| \ge 2$.  There exist positive constants $(A,B,C)$ such that for all $x \ge e^{x_0}$ 
\[
 |\theta (x) - x| < A x ( \tfrac{\log x}{R} )^{B} \exp ( -C  (\tfrac{\log x}{R})^{\frac{1}{2}}  ).
\]
\end{theorem}

Using this we obtain the following corollary. 
\begin{corollary}\label{Cor:Ak}
Let $k >0$, $x_0 \ge 1000$,  and  $b \ge \max(\log x_0, \log(4R (\frac{B+k}{C} )^2) ) $. 
% $b \ge \max \left( 4R\left(\frac14+k\right)^2, R  \max (8.36, \frac{8}{R })^2 \right)$. 
Then 
\[
     |\theta (x) - x| \le \frac{\mathcal{A}_k(b) x}{(\log x)^k} \qquad  \text{for all } x \ge e^{b}
\]
where 
\begin{equation}
   \label{AkDusart}
    \mathcal{A}_k(b) = \frac{A}{R^B} \cdot b^{B +k} \cdot \exp \left(- C\sqrt{\frac{b}{R}}\right)
 \end{equation}
 and $R = 5.573412$.
Values for $\mathcal{A}_k(b)$ for $1 \le k \le 5$ are displayed in Table \ref{Aktable}.
\end{corollary}

\begin{proof}
 We apply Lemma \ref{lemma1} with $R=5.573412$ (using \cite{TrudMoss}), and values from Theorem \ref{thetaxlarge}, namely
 $c_1=\frac{A}{R^B}$, $c_2=B$, $c_3=\frac{C}{\sqrt{R}}$, and $c_4= x_0$.
We complete the proof by noticing% the following regarding the condition for $b$:
 \[
   \frac{4(c_2+k)^2}{c_3^2} =4R \Big(\frac{B+k}{C} \Big)^2. 
 \]

%  We apply Lemma \ref{lemma1} with the $c_1, c_2, c_3, c_4$ values from Theorem \ref{thetaxlarge}, $K = b$, and Mossinghoff and Trudgian's $R=5.573412$.
%  \begin{center}
%   \[
%    c_1 = \sqrt{\frac{8}{\pi R^{\frac12}}}, \quad c_2 = \frac14, \quad c_3 = \frac{1}{\sqrt{R}}, \quad R = 5.573412.
%   \]
%  \end{center}
%  which are values from Corollary \ref{DusCor1.2}.
%  Thus we obtain $\mathcal{A}_k(b)$.
%  as given by \eqref{Ak}.
\end{proof}
% \ \\ 
% {\color{blue} This table updated June 18, 2018 by Sam}\\
%%%%%%%%%%%%%%%%%%%%%%%%%%%%%%%%%%%%%%%%%%%%%%%%%%%%%%%%%%%%%%%%%%%%%%
\subsection{Upper and lower bounds for $\theta(x)$ in the middle range $e^{J} \le x < e^{K}$} \label{section:xmiddle}
In this range we combine bounds for $\psi(x)-\theta(x)$ established in 
Corollary \ref{psi-theta:ExplicitCor} with the current best 
known bounds for $\psi(x)$ as derived in Appendix \ref{psixboundssection} to produce a bound for $\theta(x)$.
We begin with a general result. 
 %and from B\"uthe \cite{But}
% . The next lemma gives a general result which takes a bound for $\psi(x)-\theta(x)$ and bounds for $\psi(x)$
 
\begin{lemma}\label{lemma2}
Let $k=1,\ldots,5$. Assume there exist a positive integer $n$, real numbers $a_{\ell} \ge 0$ for every $\ell \in \{1, 2, \ldots, n \}$, 
 and $x_0 > 0$ such that 
 \begin{equation}
  \label{psithetadiff}
  \psi (x) - \theta (x) \le \sum_{\ell=1}^n a_{\ell} x^{\frac{1}{\ell+1}} \qquad \text{ for all } x \ge x_0.
 \end{equation}
 Let $b' > b \ge 2k$, $e^b \le x_0$, and assume that there exists $\varepsilon(b)>0$ such that
  \begin{equation}
  \label{psixdiff}
  |\psi (x) - x| \le \varepsilon(b)x \qquad \text{for all }x \ge e^{b}.
 \end{equation}
%  where $2k \le b < \tilde{b}$. 
% where $b_j$ denotes the $j$-th entry of column 1 of  Table \ref{Plapsixvals1}
%or Table \ref{Wedpsixvals}. 
Then we have 
 \begin{equation}
  \label{thetakBbd}
   |\theta(x)-x| \le   \frac{B_k x}{(\log x)^k} \qquad  \text{for all }x \in [e^{b}, e^{b'}]
 \end{equation}
where 
\begin{equation}
  \label{defn:B}
  B_k = B_k(b,b') =
  \max_{x \in [e^b, e^{b'}]}  
   \Big(  \sum_{\ell=1}^n a_{\ell} (\log x)^k x^{-\frac{\ell}{\ell+1}} +  \varepsilon(b) (\log x)^k  \Big).
%   b^{k} 
% \sum_{\ell=1}^n a_{\ell} \exp\left(-\frac{\ell b}{\ell+1}\right) + \varepsilon(b) \tilde{b}^k.
%  \sum_{\ell=1}^n a_{\ell} e^{-\frac{\ell \cdot b}{\ell+1}} + \varepsilon \cdot (\tilde{b})^k.
\end{equation}
Note that 
\begin{equation}
  \label{defn:tildeB}
  B_k \le \widetilde{B}_k = \widetilde{B}_k(b,b') =  b^{k} 
 \sum_{\ell=1}^n a_{\ell} \exp\left(-\frac{\ell b}{\ell+1}\right) + \varepsilon(b) (b')^k.
\end{equation}
% (ii)  Let $p,q \in \N$ exist such that $b_j > 0$ is ordered for $j \in \{ p, p+1, \cdots, q\}$ such that 
% \begin{equation}
%  |\psi (x) - x| < \varepsilon_j \cdot x \qquad  \text{for all }x \in [e^{b_j}, e^{b_{j+1}}]. \label{psixdiff}
% \end{equation}
% Then we have
%  Let $k,l,n,p,q \in \N$, for $j \in \{l, l+1, \cdots n\}$, $a_j \ge 0$, for $i \in \{ p, p+1, \cdots, q, q+1\}$ $b_i > 0$ such that 
%  $y < z \Rightarrow b_y < b_z$, and $\frac{k(l+1)}{l} \le b_p$. Then, if
%  \begin{equation}
%   \psi (x) - \theta (x) < \sum_{j=l}^n a_j x^{\frac{1}{j+1}}, \label{psithetadiff}
%  \end{equation}
% and
% \begin{equation}
%  |\psi (x) - x| < \varepsilon_i \cdot x \qquad  \text{for all }x \in [e^{b_i}, e^{b_{i+1}}], \label{psixdiff}
% \end{equation}
%\begin{equation}
% |\theta (x) - x| \le \frac{ \mathcal{B}_k x}{\log^k x}  \qquad  \text{for all }x \in [e^{b_p}, e^{b_{q+1}}]
%\end{equation}
%where
%\begin{equation}
% \mathcal{B}_k = \max_{p \le j \le q}\{ B_{j,k,} \}, \qquad B_{j,k} := \left( b_{j}^{k} 
% \sum_{\ell=1}^n a_{\ell} e^{-\frac{\ell \cdot b_j}{\ell+1}} + \varepsilon_j \cdot b_{j+1}^k \right).
%\end{equation}
 \end{lemma}
 \noindent {\bf Remark}.  Note that the value for $B$ is  slightly smaller than $\tilde{B}$. 
 However, at times we make use of the weaker value given by \eqref{defn:tildeB}.
\begin{proof} By the triangle inequality and the non-negativity of $\psi(x)-\theta(x)$, we have 
 \begin{align*}
  |\theta (x) - x| 
%  &= |\theta (x) - \psi (x) + \psi (x) - x| \\
%  &\le |\psi(x) - \theta (x)| + |\psi (x) - x| \\
  & \le  \psi (x) - \theta (x) + |\psi(x) - x|.
 \end{align*}
 Bounding these terms by (\ref{psithetadiff}) and (\ref{psixdiff}), we have for $x \ge e^b$,
%  for  $x \in [e^{b}, e^{\tilde{b}}]$,
 \begin{align*}
  |\theta(x) - x| &\le
   \frac{x}{(\log x)^k} \Big(  \sum_{\ell=1}^n a_{\ell} (\log x)^k x^{-\frac{\ell}{\ell+1}} +  \varepsilon(b) (\log x)^k  \Big).
%   =\frac{x}{\log^k x} ( r(x)+ \varepsilon(b) \log^k x )
%  \\
%  &= \frac{x}{\log^k x} \left( \sum_{j=l}^n a_j \frac{\log^k x}{x^{\frac{j}{j+1}}} + \varepsilon_i \log^k x \right)
 \end{align*}
%  where $r(x) = \log^k x  \sum_{\ell=1}^n a_{\ell}  x^{-\frac{\ell}{\ell+1}}$.
This immediately implies \eqref{thetakBbd} holds with \eqref{defn:B}.  
Observe that since $x\ge e^b > e^{2k} \ge e^{\frac{k(\ell+1)}{\ell}}$, then each $a_\ell (\log x)^k x^{-\frac{\ell}{\ell + 1}}$ decreases with $x$. 
On the other hand, $\varepsilon(b) (\log x)^k$ increases with $x$ and thus we have the inequality 
\eqref{defn:tildeB}.
\end{proof}
\begin{corollary}\label{Cor:Bk}
%(i) 
Let $k \in \{ 1,\ldots,5 \}$, and let $b$ and $b'$ be any consecutive entries of column 1 of Table \ref{Wedpsixvals} such that $b < b'$. i.e. 
we assume that there exists $\varepsilon(b)>0$ 
such that 
\[
 |\psi(x) - x| \le \varepsilon(b)x \qquad \text{for all }x \in [ e^{b} , e^{b'} ] .
\]
Thus 
\begin{equation}
 \label{B:ExpSubinterval}
 |\theta(x)-x| \le  \frac{B_k(b,b') x}{(\log x)^k} \qquad   \text{for all }x \in [e^{b}, e^{b'}],
\end{equation}
where
\begin{equation}
 \label{Bbbprime2}
 B_k(b,b') = a_1(b) b^k e^{-\frac{b}{2}} + a_2(b) b^k e^{-\frac{2b}{3}}+  (b')^k \varepsilon(b),
\end{equation}
% \begin{equation}
%  \label{Bbbprime2}
%  B_k(b,b') = \max_{ b \le b \le b_{j+1} } \left\{a_1(b_j) b^k e^{-\frac{b}{2}} + a_2(b_j) b^k e^{-\frac{2b}{3}}+ \varepsilon(b') b^k \right\},
% \end{equation}
%\footnote{ def of $B_k(b,b',n)$ used to be $B_{j,k}$ which was $a_1(b_j) b_j^k e^{-\frac{b_j}{2}} + a_2(b_j) b_j^k e^{-\frac{2b_j}{3}}+ \varepsilon(b_j) b_{j+1}^k$ }
%
and $a_1,a_2$ are defined in \cref{psi-theta:ExplicitCor}. 
% , and the values for the $\widetilde B_{j,k}$ are displayed in Table \ref{BkiTableWedeniwski} which correspond to Wedeniwski's H_0 value. \\
% \ref{EtaTable1Buthe}, \ref{EtaTable2Buthe}, \ref{EtaTable3Allysa}, and \ref{EtaTable4Allysa}.  \\
\\ In addition, let $b_0$ be any entry in column 1 of \cref{BkiTableWedeniwski}. Then, 
% (ii) Let $k=1,\ldots,5 $, and let $b_j$ denote the $j$-th entry of column 1 of Table \ref{Wedpsixvals}%{BkMaxTableWedeniwski}
% , i.e we assume that there exists $\mathcal{B}_{j,k} > 0$ such that \eqref{Bbbprime2} is occuring.
% Thus
\begin{equation}
 \label{bound:mathcalB}
 |\theta (x) - x| \le \frac{\mathcal{B}_k(b_0) x}{(\log x)^k}  \qquad \text{for all }x \in [e^{b_0}, e^K]
 %\text{ or }  x \in [e^{b_j}, e^{b_{13\,900}}]
\end{equation}
% and there exists $\widetilde{\mathcal{B}}_{j,k} > 0$ such that 
% \[
%  |\theta (x) - x| \le \frac{\widetilde{\mathcal{B}}_{j,k} x}{\log^k x}  \qquad  \text{ for }  x \in [e^{b_j}, e^{13\,900}]
% \]
where $K$ is the last entry in Column $1$ of \cref{BkiTableWedeniwski}, and
\begin{equation}
\label{MathcalBbbprime2}
\mathcal{B}_k(b_0) = \max_{b,b' \atop b_0 \le b < b'}   B_k(b,b'). 
\end{equation}
Values for $B_k(b,b')$ and $\mathcal{B}_k(b_0)$ are respectively displayed in Tables \ref{BkiTableWedeniwski} and \ref{BkMaxTableWedeniwski}.
% \ref{B_kTable3Allysa} and \ref{B_kTable4Allysa}.  
\end{corollary}
%%%%%%%%%%%%%%%%%%%%%%%%
\begin{proof} %(i) 
We apply Lemma \ref{lemma2} with $k \in \{1,2,3,4,5\}$, $b_0 = b$, and $n=2$ and obtain \eqref{Bbbprime2}. For we take
$B_k(b,b') = \widetilde{B}_k(b,b',2) $ with $a_1=a_1(b)$ and $a_2=a_2(b)$ as defined in \eqref{def-a1} and \eqref{def-a2} respectively.
\\ %(ii) 
The inequality \eqref{bound:mathcalB} follows from \eqref{B:ExpSubinterval} together with the fact that 
$[e^{b_0},e^K] = \bigcup_{b \in [b_0, K)} [e^{b},e^{b'}]$. 
%where $b_{k^*+1}=K=13\,900$.  
%Observe that by part (i) of the lemma $ |\theta(x)-x| \le  \frac{B_{i,k} x}{(\log x)^k}$ and thus we derive \eqref{bound:mathcalB}.
\end{proof}
%\begin{remark}  We can recover  a result of Axler \cite[Theorem 1.1]{Axler17_2}.
%He shows that 
% \begin{equation}
%  |\theta (x) - x| \le \frac{0.15x}{\log^3 x} \qquad  \text{for all } x \in [e^{32}, e^{5000}].
% \end{equation}
%Using our notation, Axler proves Lemma \ref{lemma2} for the following values
% \begin{align*}
%  \B_3 = 0.15; \qquad k = 3; \qquad  \text{for all }  e^{32} \le x \le  e^{5000}.
% \end{align*}
% {\bf EXPLAIN HOW TABLES GIVES THIS RESULT.}. 
% ($\B_3 = 0.1458 < 0.15$,
% as Axler announced.)
%\end{remark}
% {\commentcolor{red} MOVE TO SECTION 3}.
% Application:
% \[a_1=1+1.82348\cdot10^{-10},\ a_2 = 1.777745 \implies \]
% question: can Dusart's $a_2= 1.777745$be improved?
%%%%%%%%%%%%%%%%%%%%%%%%%%%%%%%%%%%%
%\subsection{Costa-Pereira's approach to upper and lower bounds for $\theta(x)$ in the medium range $e^{J} < x \le e^K$}
\noindent {\it Remark} 1. 
Note that $B_k(b,b')$ is essentially $\varepsilon(b)(b')^k$ and thus any refinement on the other terms only brings minor improvements.
For instance, Costa-Pereira's \cite[Theorem 1]{Cos} estimates for $\psi(x)-\theta(x)$ affects digits much further than those we display.
For every $x>0$,
\begin{equation}\label{CP1}
\psi(x) - \psi (x^{\frac12} )- \psi (x^{\frac13} )- \psi (x^{\frac15} ) \le 
\theta(x) \le \psi(x) - \psi(x^{\frac12})- \psi(x^{\frac13})- \psi(x^{\frac17}).
\end{equation}
2.  While we use the bound
\[
  B_k(b,b') \le a_1(b) b^k e^{-\frac{b}{2}} + a_2(b) b^k e^{-\frac{2b}{3}}+ \varepsilon(b) (b')^k,
\]
for any given $b$ and $k$, a more specific $\max$ in \eqref{Bbbprime2} can be computed exactly using calculus.
\subsection{Upper bounds for $\theta(x)$ in the small range $x < e^{J}$}
\label{trivial:upper}
%We take $J$ such that $\theta(x) < x$ for all $x \le e^J$. 
It is has been proven that the first sign change of $x - \theta(x)$ occurs before $1.3972 \cdot 10^{316}$ \cite[Lemma 9.4]{SaTrDe}.
In fact, B\"uthe \cite[Theorem 2, (1.7)]{But2} has shown that 
\begin{equation}
  \label{Chebbias}
        \theta(x) < x - 0.05 \sqrt{x}  \text{ for all } x \le 10^{19}.
\end{equation}
It follows that 
%\todo{are these both $J$ or $J_0$? This is correct. $J$ and $J_0$ are different. NN}
\begin{equation}
 \label{ThetaLex}
 \theta(x) - x \le M_k \ \text{for all } x \le e^{J_0}, \ \text{with } M_k = 0 \ \text{ and }  J=19\log10. 
\end{equation}
% The argument for the lower bound with $x \le e^{43}$ will be divided into two cases; we use numerical verification for the first values of $x$. For $x \le e^{20}$, a direct calculation 
% will be employed.
\subsection{Lower bounds for $\theta(x)$ in the small range $e^{J_0} \le x < e^{J}$} \label{section:xsmall}
% We now discuss the lower bound in this region. 
%%%%%%%%%%%%%%%%%%%%%%%%%%%%%%%%%%%%%%%%%%%%%%%%%%%%%%%%%%%%%%%%%%%%%%%
%%%%%%%%%%%%%%%%%%%%%%%%%%%%%%%%%%%%%%%%%%%%%%%%%%%%%%%%%%%%%%%%%%%%%%%
%%\quad \\
% \begin{lemma}\label{LemmaxCk}
%  Let $c$, $N'$, $N''> 0$, and let $k \in \N$ exist such that
%  \begin{equation}
%   \theta (x) \ge x - cx^{\frac12} \qquad  \text{for all } N' \le x \le N''.\label{cxhalf}
%  \end{equation}
%  then, for $C_k=1,\ldots,5$, we have
%  \begin{equation}
%   \theta (x) \ge x - \frac{xC_k}{\log^k x} \qquad  \text{for all }x \in (\max\{N',N_{2}\}, N''),\label{LowerboundN2}
%  \end{equation}
%  where $N_2$ is the solution to $\frac{c}{C_k} = \frac{N_2^{\frac12}}{\log^k N_2}$, satisfying $N_2 \ge e^{2k}$ and $N_2 \le N''$.
% \end{lemma}
% %37
% \begin{proof}
%  We directly bound $cx^{\frac12}$ in \eqref{cxhalf} with $\frac{xC_k}{\log^k x}$ after proving the bound is valid. To do so, we need
%  \begin{equation*}
%   \frac{xC_k}{\log^k x} \ge cx^{\frac12} \Leftrightarrow \frac{c}{C_k} \le \frac{x^{\frac12}}{\log^k x}.
%  \end{equation*}
%  Applying the first derivative test to $\frac{x^{\frac12}}{\log^k x}$, we find that it is increasing for all $x > e^{2k}$. Thus by assumption that $N_2 \ge e^{2k}$, we have
%  that
%  \begin{equation}
%   \frac{x^{\frac12}}{\log^k x} \ge \frac{c}{C_k} \qquad  \text{for all } x > N_2.
%  \end{equation}
%  Which then from \eqref{cxhalf} implies \eqref{LowerboundN2}
% \end{proof}
% %%
% 
%  The following lemma provides a $D_k$ value when given $N_{2}$, rather than $D_k$ providing a $N_{2}$ value.
We provide here an improvement to B\"uthe's \cite[Lemma 1]{But2}. 
%In this range, we bound $\theta(x)$ below through the combination of two results;
%To do so, we combine Costa Pereira's \cite[Theorem 1, Equation (1)]{Cos} about $\psi(x) - \theta(x)$ and B\"{u}the \cite{But2}'s numerical bounds of the shape
\[
  -c \le  \frac{x-\psi(x)}{\sqrt{x}} \le C.
\]
%  with $x \le e^{43}$. 
%   , and we use the Chebyshev bounds derived in Section \ref{psixboundssection}  for $\psi(x^{\frac{1}{n}})$ for $n=3,5$. 
\begin{lemma}
 \label{lemma3}
 Let $1 \le u < v$. Assume there exist $c = c_{u,v} > 0$ and $C= C_{u,v} > 0$ such that 
 \begin{equation}
 \label{equ:c-Psi-C}
  -c \le \frac{x - \psi(x)}{\sqrt{x}} \le C \qquad \text{for every }x \in [u,v].
 \end{equation}
Assume that there exists $c_0>0$ such that 
 \begin{equation}\label{defn1:c0}
  \psi(x) < c_0 x \text{ for all } x > 0.
 \end{equation}
If $u^2<v$, then 
  \begin{equation}
  \label{eqn:C-Section-Post-Lemma}
  \theta(x) \ge x - (C + 1)x^{\frac12} - c_0 x^\frac13 - c x^{\frac14} - c_0 x^\frac15 \  \text{ for all } x \in [u^2,v].
 \end{equation}
% where $c_0=1.03883$.
%  \begin{equation}
%   \label{eqn:C-Section-Post-Lemma}
%   \theta(x) \ge x - (C + 1)x^{\frac12} - \widetilde{c} x^\frac13 - c x^{\frac14} - \widetilde{c} x^\frac15 \qquad  \text{for all }x \in [u^2,v]
%  \end{equation}
%  where
%  \begin{equation}
%   \label{defn:ctilde}
%   \widetilde{c} = \widetilde{c}_{u,v} =  \max_{u^\frac15 \le t \le v^\frac13} \frac{\psi(t)}{t}.
%  \end{equation}
\end{lemma}
\begin{proof}
%  Let $c,C$ exist such that
%  \[
%   c < \frac{x - \psi(x)}{\sqrt{x}} < C \qquad  \text{for all }x \in [a,b].
%  \]
Costa-Pereira \cite[Theorem 1, Equation (1)]{Cos} proved that 
 \[
%    \label{CP}
  \psi(x) - \theta(x) \le \psi(x^{\frac12}) + \psi(x^{\frac13}) + \psi(x^{\frac15})\  \text{ for all } x >0.
%  -\theta(x) &\le -\psi(x) + \psi(x^{\frac12}) + \psi(x^{\frac13}) + \psi(x^{\frac15}) \\
%  \theta(x) &\ge \psi(x) - \psi(x^{\frac12}) - \psi(x^{\frac13}) - \psi(x^{\frac15}) 
 \]
Together with \eqref{defn1:c0}, it follows 
%  From \eqref{defn:ctilde}, we have
%  We have the following result due to Rosser and Schoenfeld \cite[Theorem 12]{RS1}.
%  \[
%   \psi(x) < 1.03883 x \text{ for all } x > 0.
%  \]
% Thus for $c_0 = 1.03883$ we have
 \[
  \psi(x) - \theta(x) \le \psi(x^{\frac12}) +  c_0 x^{\frac13} +c_0 x^{\frac15}  \text{ for all } x \in [u,v].
 \]
%  Note that  $x^{\frac13} \in [a^\frac13,b^\frac13]$ and $x^{\frac15} \in [a^\frac15,b^\frac15]$. 
%  A direct computation in PARI/GP of the maximum of $\psi(x)$ over $[a^\frac15,b^\frac13]$ yields
%  \[
%     \psi(t) \le \widetilde{c} t \qquad  \text{for all } t \in  [a^\frac15,b^\frac13].
%  \]
%  which rearranges to  
%  \[
%      \theta(x) \ge \psi(x) - \psi(x^{\frac12}) -  \widetilde{c} (x^{\frac13} + x^{\frac15}) \text{ for all } x \in [a,b].
%  \]\noindent
This may be rewritten as 
 \begin{equation}
  \label{thetainequality}
   \theta(x) \ge x + x^{\frac12} \left( \frac{\psi(x) - x}{x^{\frac12}} \right) - x^{\frac12} + x^{\frac14} \left( \frac{x^{\frac12}
    - \psi(x^{\frac12})}{x^{\frac14}} \right) - c_0 x^{\frac13} - c_0 x^{\frac15}.
 \end{equation}
% since $x^{\frac{1}{3}},x^{\frac{1}{5}} \in [u^{\frac{1}{5}},v^{\frac{1}{5}}]$. 
We conclude using \eqref{equ:c-Psi-C}: $\frac{\psi(x) - x}{x^\frac12} \ge -C$ and $\frac{x^\frac12 - \psi(x^\frac12)}{x^\frac14} \ge -c$ for every $x \in [u^2, v]$.
% Thus we have \eqref{eqn:C-Section-Post-Lemma} .
%  \[
%   \theta(x) \ge x - (1 + C)x^\frac12 - c x^\frac14 - \widetilde{c} (x^\frac13 + x^\frac15).
%  \]
%  for $x \in [a^2,b]$.
\end{proof}
Rosser and Schoenfeld \cite[Theorem 12]{RS1} proved \eqref{defn1:c0} with
 \begin{equation}\label{defn:c0}
  c_0= 1.03883.
 \end{equation}
Note that this value cannot be improved since the bound \eqref{defn1:c0}
is achieved for $x=113$. We shall use this value throughout the article. 
%Change here.  Wording. 
%This is the value we use in this paper as it is optimal (the bound is attained at $x=113$). \todo{This sentence is a bit clumsy, I read it a few times. Maybe better to say ``This $c_0$ is optimal (and thus used in this paper) since this bound is achieved at $x=113$.''  Or something...}
Recently, B\"uthe has proven many bounds like \eqref{equ:c-Psi-C}. From \cite[Equation (6.2), Table 1]{But2}, we have: 
 \begin{center}
  \begin{tabular}{|*{4}{c}|}
   \hline
   $u$ & $v$ & $c$ & $C$ \\ \hline
   $100$ & $5 \cdot 10^{10}$ & $0.8$ & $0.81$ \\
   $100$ & $32 \cdot 10^{12}$ & $0.88$ & $0.86$ \\
   $100$ & $10^{19}$ & $0.94$ & $0.94$ \\ \hline
  \end{tabular}
 \end{center}
In the following corollary, we restrict $x$ to $[e^b,v]$, a subset of $[u^2, v]$. 
% Thus we can compute $\widetilde{c}$ over a shorter range:
% \[
%  \widetilde c_{e^b,v} = \max_{e^\frac{b}{10} \le t \le v^\frac13} \frac{\psi(t)}{t}.
% \]
% {\color{red} Previous version had $e^{\frac{b}{5}}$ instead of  $e^{\frac{b}{10}}$. Fix values in table below.}
% While $\widetilde c_{e^b,v}$ only needs to be computed over $[e^\frac{b}{10},v^\frac13]$, computations show that values of $\widetilde c_{e^b,v}$ remain unchanged 
% through computation up to $10^{19}$. Thus we have:
% %    {\color{red} Define the endpoints by non-integers?}
%  \begin{center}
%   \begin{longtable}{|c|*{10}{|c}|}
%    \caption{Values of $\widetilde c_{e^b,v}$ valid for $e^b \le x \le 10^{19}$}
%    \label{tab:TildeC}
%    \\
%    \hline
%    $b$ & $20$ & $24$ & $27$ & $29$ & $31$ & $33$ & $37$ & $38$ & $40$ & $42$ \\ \hline
%    \vphantom{$\big)$}$\widetilde c_{e^b,v}$ & $1.03883$ & $ 1.03591$ & $ 1.02728$ & $ 1.02116$ & $ 1.01802$ & $ 1.01363$ & $ 1.01196$ & $ 1.00990$ & $ 1.00661$ & $ 1.00649$ \\
%    \hline
%   \end{longtable}
%  \end{center}
\begin{corollary}\label{Cor:Ck}
 Let $(v,c,C) \in \{(5 \cdot 10^{10}, 0.8, 0.81), (32 \cdot 10^{12}, 0.88, 0.86), (10^{19}, 0.94, 0.94) \}$. Let $k \ge 0$ and let $b$ satisfy 
 $\max(10^4, e^{2k}) \le e^{b} \le v$. Then %there exists $\mathcal{C}_{b,k}$ such that 
 \begin{equation}
  \label{eqn:C-Section-Post-Corollary}
  \theta(x) \ge x - \frac{\mathcal{C}_{b,k}x}{(\log x)^k} \qquad  \text{for all } x \in [e^b,v]
 \end{equation}
 where 
 \begin{equation}
  \label{defn:mathcalCbk}
  \mathcal{C}_{b,k} = b^k ( (C+1)e^{-b/2} + c_0e^{-2b/3}+ c e^{-3b/4}+ c_0 e^{-4b/5}  ),
 \end{equation}
and where $c_0$ is defined in \eqref{defn:c0}. 
%  where $\widetilde c_{e^b,v}$ is defined by \eqref{defn:ctilde} with specific values given in Table \ref{tab:TildeC}. 
Values of $\mathcal{C}_{b,k}$ can be found in Table \ref{CkValues}.
\end{corollary}
\begin{proof}
We apply \eqref{eqn:C-Section-Post-Lemma} with $u=e^{\frac{b}{2}} $:
%for all $x \in [e^b,v]$, we have
 \[
   \theta(x) \ge x-(C +1)x^{\frac12} - cx^{\frac14}  - c_0 x^{\frac13} - c_0 x^{\frac15} \ \text{for all}\ x \in [e^b,v].
 \]
% In order to obtain \eqref{eqn:C-Section-Post-Corollary}, we see that 
We now set 
%prove
% $$(C +1)x^{\frac12} + c x^{\frac14}  +c_0 x^{\frac13} + c_0 x^{\frac15} \le \frac{\mathcal{C}_{b,k}x}{(\log x)^k}
%   \text{ for } 
% x \in [e^b,v].
% $$
% This is equivalent to finding the maximum of
 \begin{equation}
  \label{eqn:C-Section-prebound}
   \mathcal{C}_{b,k} = \max_{x \in [e^b,v]  } \Big\{ 
  (C+1)\frac{(\log x)^k}{x^\frac12} + c_0\frac{(\log x)^k}{x^\frac23} + c\frac{(\log x)^k}{x^\frac34} + c_0\frac{(\log x)^k}{x^\frac45}
  \Big\}. 
 \end{equation}
We find that this equals the expression in \eqref{defn:mathcalCbk} by observing that 
%  \[
%  %   \label{phibound}
%     \phi(x) \le \mathcal{C}_{b,k} \text{ for all } x \in [e^b,v],
%  \]
%  where 
%  \[
%    \phi(x) = (C_b+1)\frac{\log^k x}{x^\frac12} - c_b\frac{\log ^k x}{x^\frac34} + \widetilde c_{e^b,v}\frac{\log ^k x}{x^\frac23}+ \widetilde c_{e^b,v}\frac{\log^k x}{x^\frac45}.
%  \]
for  $a \in \{ \frac12, \frac23, \frac34, \frac45 \}$, $\frac{(\log x)^k}{ x^{a}}$ is decreasing for  $x \ge e^b$ as long as $e^{b} \ge e^{k/a}$. 
This last inequality leads to the condition $b \ge 2k$. 
%Thus \eqref{eqn:C-Section-prebound} is non-increasing on 
% $\max \{e^{2k},e^{3k/2},e^{4k/3},e^{5k/4} \} = e^{2k}$. Since $b \ge 2k$, 
% %we have that $\phi(x)$ is decreasing on $x \in [e^b, \infty)$. Since
% \eqref{eqn:C-Section-prebound} takes its maximum at $x = e^b$, giving \eqref{defn:mathcalCbk}.
%%  is decreasing on $[e^b,\infty)$, when we consider the range $x \in [e^b,e^{b+1}]$, we have $\phi(x) \le \phi(e^b)$. Thus we may take
%%  $\mathcal{C}_{b,k} = \phi(e^b)$, valid for $x \in [e^b,e^{b+1}]$ as desired.
%%  Notice
%%  \begin{equation*}
%%   \frac{(C_j +1)x^{\frac12} + c_jx^{\frac14}  -d_j(x^{\frac13}+x^{\frac15})}{D_{j,k}} = \frac{x}{e^j} \Bigl( \frac{(C_j+1)x^{-1/2}-c_jx^{-3/4}-d_j(x^{-2/3}+x^{-4/5})}{(C_j+1)(e^j)^{-1/2}-c_j(e^j)^{-3/4}+d_j( (e^j)^{-2/3}+(e^j)^{-4/5} )} \Bigr)
%%  \end{equation*}
%%  Let $f(x) = \frac{x}{\log ^k x}$. By calculus we can see that $f(x)$ is increasing for $x \ge e^k$. Thus if $e^j \ge e^k$, then 
%%  $f(x) \ge f(e^j)$. Notice $f(e^j)$ is equal to the left hand side of the inequality we wish to prove, thus we arrive at the desired result. 
%% 
%%  Thus that it suffices to prove $r(x)x^{\frac12} \le \frac{D_{j,k} x}{\log^k x}$ for $x \in [e^j,e^{j+1}]$.  This is equivalent to showing 
%%  that $\frac{r}{D_{j,k}} \le f(x)$ for $x \in [e^j,e^{j+1}]$ where $f(x) =  \frac{x^{\frac12}}{\log^k x}$.  Note that $f$ has a critical point
%%  $x=e^{2k}$ and is increasing if $x \ge e^{2k}$.    Thus if $e^j \ge e^{2k}$, then $f(x) \ge f(e^j)=r(x)/D_{j,k}$, by \eqref{Dk}.
\end{proof}
\subsection{Lower bounds for $\theta(x)$ for $x < e^{J_0}$}\label{sec:numerical}
The following lemma gives a condition to obtain a lower bound 
%provides numerical calculations for a lower bound 
for $\theta(x)$ for the first values of $x$. 
\begin{lemma}\label{lem:numerical}
 Let $k=1,\ldots,5$. Let $0 < a < b$ such that $a > e^{k+1}$. Let
%  \[
%   \label{x0condition}
%    x_0 > e^{k+1}.
%  \]
%   Let  
%  \[
%   d_k(n) := \frac{\log^k(p_{n}) \cdot (p_{n} - \theta(p_{n-1}))}{p_{n}}
%  \]
 $p_n$ denote the $n$-th prime, with $p_{n_0}$ and $p_{n_1}$ being the smallest primes greater than $a$ and $b$ respectively. 
Let 
\begin{equation}
\label{defn:mathcalD}
\mathcal{D}_k(a,b) = \max_{n_0 \le n \le n_1} \frac{(\log p_{n})^k \cdot (p_{n} - \theta(p_{n-1}))}{p_{n}} .
\end{equation}
If 
\begin{equation}
\label{Dkcondition}
%   \label{Gkcondition}
    \mathcal{D}_k(a,b) < (k+1)^{k+1},
 \end{equation}
then 
 \begin{equation}
  \label{bound:mathcalD}
  \theta(x) \ge x- \frac{\mathcal{D}_k(a,b) x}{(\log x)^k}\ \text{when}\ a \le x \le b.
 \end{equation}
\end{lemma}
% \begin{lemma}\label{lem:numerical}
%  Let $k$ be a non-negative integer.  Let $0 < X_0 < y_0$ and assume that $X_0 > e^{k+1}$. Let
% %  \[
% %   \label{x0condition}
% %    x_0 > e^{k+1}.
% %  \]
% %   Let  
%  \[
%   d_k(n) := \frac{\log^k(p_{n}) \cdot (p_{n} - \theta(p_{n-1}))}{p_{n}}
%  \]
%  where $p_n$ denotes the $n$-th prime.
%  Let $p_{n_0}$ be the smallest prime $\ge X_0$, and let $p_{n_1}$ be the smallest prime $> y_0$, and 
%  \[
%   \mathcal{D}_k = \mathcal{D}_k(X_0,y_0) = \max_{n_0 \le n \le n_1} d_k(n). 
%  \]
%  Assume that 
%  \[
% %   \label{Gkcondition}
%     \mathcal{D}_k < (k+1)^{k+1}. 
%  \]
%  Then we have 
%  \[
%   \theta(x) \ge x- \frac{\mathcal{D}_k x}{(\log x)^k} \text{ for all } x_0 \le x \le y_0.
%  \]
% \end{lemma}
%
\begin{proof}
 Fix $n_0 \le n \le n_1$ and set
 \begin{equation}
  \label{defn:littleD}
  D_k(n) = \frac{(\log p_n)^k (p_n - \theta(p_{n-1}))}{p_n}.
 \end{equation}
 Let $x\in [p_{n-1} , p_n)$ and observe that
 \[
  \theta(x) = \theta(p_{n-1}) = p_n \left( 1 - \frac{D_k(n)}{(\log p_n)^k} \right),
 \]
 and that the function $\varphi(x)=x(1-c/(\log x)^k)$ increases with $x$, as long  as $c <(k+1)^{k+1}$.  
 This is since $\varphi'(x)= 1+ \frac{c}{(\log x)^{k+1}}(k-\log x)$ and $\varphi''(x) =  \frac{ck}{x (\log x)^{k+2}}(\log x - (k+1)) $.
It follows that $\varphi'$ has a minimum at $x=e^{k+1}$.  Therefore  $\varphi$ is increasing since
\[
  \varphi'(x) \ge \varphi'(e^{k+1}) = 1-\frac{c}{(k+1)^{k+1}} >0,
\]
from the assumption $c < (k+1)^{k+1}$.
Thus \eqref{Dkcondition} ensures that
 \begin{equation}
  \label{thetaprimeint}
  \theta(x) \ge x \left( 1 - \frac{D_k(n)}{(\log x)^k} \right) \ \text{for all } x \in [p_{n-1}, p_n) .
 \end{equation}
% since $  D_k(n) \le \mathcal{D}_k(a,b) < (k+1)^{k+1}$ by .  
% From our definitions of $p_{n_0}$ and $p_{n_1}$ 
It follows that \eqref{bound:mathcalD} holds since 
$
  [a,b] \subset \bigcup_{n_0 \le n \le n_1} [p_{n-1}, p_n)
 $
and $\mathcal{D}_k(a,b) = \max_{n_0 \le n \le n_1}    D_k(n)$.
% Using these facts along with \eqref{thetaprimeint},  we deduce \eqref{bound:mathcalD}.
\end{proof}
\noindent
{\bf Methodology:}
%\footnote{unclear: to fix!}
We use this lemma to obtain a numerical lower bound for $\theta(x)$ on $[1,7.0 \times 10^{11}]$.
We first subdivide into the intervals $I_n=[(n-1)\cdot 10^{10},n  \cdot 10^{10} ]$
with $n$ ranging from $1$ to $70$. 
We subdivide each $I_n$ further into 100 subintervals, each of length $10^8$.
We apply the Lemma \ref{lem:numerical} to each of these subintervals,
recording the corresponding $\mathcal{D}_k$ value.  
For each $n$, we take for $D_k(n)$ the largest value among all $\mathcal{D}_k$'s arising from the 100 subintervals.  
Values for $\mathcal{D}_k$ for selected ranges in $[1,7.0 \times 10^{11}]$ are recorded in Table \ref{Table-Dk}.
%%%%%%%%%%%
\subsection{The case $k = 0$.}\label{section:k=0}
%Both the upper and lower bounds when $k=0$ are obtained in a different manner than when $k \ge 1 $. 
The upper bound is a direct result of  Theorem \ref{psixmainthm} and the partial verification of $\theta(x) < x$.
\begin{lemma}\label{k0upper-k0implower}
Let $b > 0$.
Assume
\begin{equation}
\theta(x) < x \ \text{ for all }  x \le e^b,
\end{equation}
and that there exists $\varepsilon(b)>0$ such that 
\begin{equation}
|\psi(x)-x| \le \varepsilon(b)x  \ \text{ for all }  x \ge e^b .
\end{equation}
Then we have
%  that $\psi(x) - x < \varepsilon(b)x$ is computable. Then we have
 \begin{equation}
 \label{thetaleconstantx}
  \theta(x) \le (1 + \varepsilon(b))x, \ \text{ for all } x>0.
 \end{equation}
%for all $x > 0$.
%  , where $\varepsilon(b) > 0$ is calculated from Theorem \ref{psixmainthm}.
%\end{lemma}
%
%\begin{proof}
% By assumption, we have that $\theta(x) < x$ for $0 < x \le e^b$ and that $\psi(x) - x < \varepsilon(b)x$ for $x \ge e^b$.
%  Bounding $\psi(x)$ by $\theta(x)$ yields \eqref{thetaleconstantx} for $x \ge e^b$.
%\end{proof}
%
%The lower bound again uses bounds for $\psi(x)$ in conjunction with \cite[Theorem 1]{Cos}, a comparison of $\psi(x) - \theta(x)$.
%%
%\begin{lemma}\label{k0implower}
%Let $b > 0$. Assume there exists $\varepsilon(b) > 0$ such that $|\psi(x) - x| \le \varepsilon(b)x$, for $x \ge e^b$. Then for $x \ge e^b$ we have
%  Let $x \ge e^b$ such that $|\psi(x) - x| < \varepsilon(b)x$ is known or computable. Then for $x \ge e^b$ we have
In addition
\begin{equation}
  \label{thetageconstantx}
  \left(1 - \varepsilon(b) - c_0 (e^{-\frac{b}{2}} + e^{-\frac{2b}{3}} + e^{-\frac{4b}{5}}) \right)x \le \theta(x), \ \text{ for all } x\ge e^b.
 \end{equation}
%  where $\varepsilon_\theta(b) = \varepsilon(b) + c_0 (e^{-\frac{b}{2}} + e^{-\frac{2b}{3}} + e^{-\frac{4b}{5}})$
%where  $c_0= 1.03883$. 
\end{lemma}
\begin{proof}
Inequality \eqref{thetaleconstantx} follows immediately from Theorem \ref{psixmainthm}
 and the fact that $\theta(x)\le \psi(x)$.
Now for every $x > 0$, \cite[Theorem 1]{Cos} asserts that 
%  we have $\psi(x) < 1.03883x$ from \cite[Theorem 12]{RS1} and 
 $\psi(x) - \theta(x) \le \psi(x^{\frac12}) + \psi(x^{\frac13}) + \psi(x^{\frac15})$.
Together with \eqref{defn1:c0}, we have 
%  \footnote{Improvements are possible through bounding 
%  $\psi(x^{\frac12}),\psi(x^{\frac13})$, and $\psi(x^{\frac15})$ by their respective $\varepsilon(b)$ values, but the improvements made 
%  are most often negligibly small and diminish as $b$ increases.}
$
   \psi(x) - c_0(x^{\frac12} + x^{\frac13} + x^{\frac15}) \le \theta(x),  
$
and we conclude by applying Theorem \ref{psixmainthm} to $\psi(x)$. 
% $\psi(x) \ge (1-\varepsilon(b)) x$  for $x \ge e^{b}$ and thus
%  \[
%    \left(1-\varepsilon(b) - c_0\left(x^{-\frac{1}{2}} + x^{-\frac{2}{3}} + x^{-\frac{4}{5}} \right) \right)x \le \theta(x).
%  \]
%  Since the term in the inner brackets decreases with $x$ for $x \ge e^{b}$, we arrive at \eqref{thetageconstantx}.
% %  \[
% %    \left(1- \varepsilon(b) - 1.03883\left(e^{-\frac{b}{2}} + e^{-\frac{2b}{3}} + e^{-\frac{4b}{5}}\right) \right)x \le \theta(x).
% %  \]
\end{proof}
%The remaining subsections investigate the other cases when $k\ge1$.

%%%%%%%%%%%%%%%%%%%%%%%%%%%%%%%%%%%%%%%%%%%%%%%
\subsection{Proof of Theorem \ref{MainResult:theta}}
%%%%%%%%%%%%%%%%%%%%%%%%%%%%%%%%%%%%%%%%%%%%%%%

%We give the proof of Theorem \ref{MainResult:theta}.
%\begin{proof}
%\subsubsection{The case $k=0$.}
\quad \\ If $k=0$, we apply \cref{k0upper-k0implower} with $b = \log X_1$ for the upper bound and $b = \log X_0$ for the lower bound.
We can define:
\[
m_0 = \varepsilon(\log X_0) + c_0(X_0^{-\frac12} + X_0^{-\frac23} + X_0^{-\frac45})
\ \text{and}\ 
M_0 = \varepsilon(\log X_1).
\]
For the rest of this section, we assume $k\in\{1,\ldots, 5\}$.
Throughout this proof we let $b_n$ denote the $n$-th entry of column 1 of Table \ref{Wedpsixvals}. 
Furthermore, let $K=25000$ be the largest $b$ value in \cref{Wedpsixvals}, and let 
 \begin{equation}
   \label{X0X1}
    X_0 = e^{u_0} \text{ and } 
    X_1 = e^{u_1}.
\end{equation}
 The proof consists of four cases.   We shall divide up the intervals $[X_i,\infty)$  for $i=0,1$ into 
 various subintervals.  On each of the subintervals we shall apply a combination of 
 \cref{Cor:Ak},  Corollary \ref{Cor:Bk}, Corollary \ref{Cor:Ck}, and  Lemma \ref{lem:numerical}.
 By combining together our various bounds we shall establish the required bounds \eqref{thetalb} and \eqref{thetaub}.
If $X_0,X_1 \ge e^K$, then only \cref{Cor:Ak} is considered.
\subsubsection{Case 1: $X_0, X_1 \ge e^{K}$.}% (bounds of $\theta(x)$ for large values of $x$).}
By Corollary \ref{Cor:Ak}, we can take 
\[  
m_k = \mathcal{A}_k(u_0)\ \text{and}\ M_k = \mathcal{A}_k(u_1). 
\]
For the second case, we consider $e^J \le X_i \le e^K$. Here we consider \cref{Cor:Ak} and \cref{Cor:Bk}. For all following cases, let $B_{i,k} = B_k(b,b')$ where $b$ is the $i$-th entry in column 1 of \cref{BkiTableWedeniwski}. Validity for all $x \ge X_0$ or $x \ge X_1$ can be established by applying \cref{Cor:Bk} over consecutive intervals, then applying \cref{Cor:Ak} to bound for all $x$ above the end of the intervals. Thus we now have the following case:
\subsubsection{Case 2: $e^{J} \le X_0, X_1 < e^{K}$.}
Combining \cref{Cor:Ak} and \cref{Cor:Bk}, we may take 
\[
                 m_k = \max( \mathcal{A}_k(b_{{\ell}+1}) , \max_{n_0 \le i \le {\ell}}( B_{ i ,k}) )
   \ \text{and}\ M_k = \max( \mathcal{A}_k(b_{{\ell}+1}) , \max_{n_1 \le i \le {\ell}}( B_{ i, k}) )
\]
where $n_0$ and $n_1$ are the greatest natural numbers such that $b_{n_0} \le u_0$ and $b_{n_1} \le u_1$
and $\ell \ge n_0$ or $\ell \ge n_1$ is chosen to minimize the values of $m_k$ and $M_k$.  
\begin{remark}
 Note that whenever possible, the value of ${\ell}$ is chosen such that $\mathcal{A}_k(b_{{\ell}+1}) \le \max_{n_0 \le i \le {\ell}}( B_{ i ,k})$.
 %Near the end of \cref{BkiTableWedeniwski}, increasing the value of $l$ is impossible as the there are no higher rows in \cref{BkiTableWedeniwski}, thus the $M_k$ and $m_k$ is taken from $\mathcal{A}_k(K)$.
\end{remark}

For the last two cases, we denote $j^*$ the row of Table \ref{BkiTableWedeniwski} where $b_{j^*}=J=19 \log 10$.
\subsubsection{Case 3: $e^{J_0} \le X_i \le e^{J}$.}\label{Case3} %Lower bounds of $\theta(x)$ for small values of $x$ 
From Corollary  \ref{Cor:Ak}, Corollary \ref{Cor:Bk}, \eqref{ThetaLex}, and Corollary \ref{Cor:Ck}, we can take
\[
                 m_k = \max( \mathcal{A}_k(b_{\ell+1}) , \max_{j* \le i \le \ell}( B_{ i ,k}) , \mathcal{C}_{ \lfloor u_0 \rfloor,k} ) 
   \ \text{and}\ M_k = \max( \mathcal{A}_k(b_{\ell+1}) , \max_{j* \le i \le \ell}( B_{ i ,k}) )
\]
where $\ell \ge n_0$ or $\ell \ge n_1$ is chosen to minimize the values of $m_k$ and $M_k$. 
\subsubsection{Case 4: $X_i < e^{J_0}$.}\label{Case4}%Lower bounds of $\theta(x)$ for small values of $x$:  
By applying Lemma \ref{lem:numerical} we obtain 
\[
   x \left( 1- \frac{\mathcal{D}_k(X_0,e^{J_0})}{(\log x)^k} \right) \le \theta(x)
   \text{ when }  X_0 \le x \le e^{J_0}.  
\]
% We subdivide 
%\todo{subdivide?} 
% $[X_0,e^{J_0})$ into (as small as possible) sub-intervals $[a,b)$, and we calculate each $\mathcal{D}_k(a,b)$ (sample of values are listed in Table \ref{Table-Dk}).
% The value for $\mathcal{D}_k(X_0,e^{J_0})$ is then determined as the maximum over all these $\mathcal{D}_k(a,b)$. 
% 
% \ref{Table-Dk}we must take the maximum of the values of the $kth$ column
% between the lines of the desired $X_0$ and $e^{J_0}$.   
We combine this with \cref{Cor:Ak}, \cref{lemma2}, \cref{Cor:Ck}, and \eqref{ThetaLex}: 
\[
                  m_k = \max( \mathcal{A}_k(b_{\ell+1}) , \max_{j* \le i \le \ell}( B_{ i ,k}) , \mathcal{C}_{ J_0,k}, \mathcal{D}_k(X_0,e^{J_0}) ),
    \ \text{and}\ M_k = \max( \mathcal{A}_k(b_{\ell+1}) , \max_{j* \le i \le \ell}( B_{ i ,k}) )
\]
where $\ell \ge n_0$ or $\ell \ge n_1$ is chosen to minimize the values of $m_k$ and $M_k$. 

\subsection{Computational examples.}
We now give examples of how to apply Theorem \ref{MainResult:theta} in specific cases.  First, we   are able 
 to improve the second part of  \cite[Theorem 1]{Axler17_2}.
\begin{corollary}
% \cite[Theorem 1]{Axler17_2}
 For every $x \ge 19\, 035\, 709\, 163$, we have
 \begin{equation}
  \label{AxlerLowerBound}
  x\left( 1 - \frac{0.15}{(\log x)^3} \right) < \theta(x),
 \end{equation}
 and for every $x > 1$ we have
 \begin{equation}
 \label{AxlerUpperBound}
  \theta(x) < x\left( 1 + \frac{0.024334}{(\log x)^3} \right).
 \end{equation}
\end{corollary}
Observe that \eqref{AxlerLowerBound} is the same as equation (1.4) in \cite{Axler17_2}
and  \eqref{AxlerUpperBound} improves the constant in (1.5) of \cite{Axler17_2} from $0.15$ to $0.024334$.
\begin{proof}
%To establish \eqref{AxlerLowerBound} and \eqref{AxlerUpperBound}, 
We apply \cref{MainResult:theta} with $k = 3$, $X_0 = 19\, 035\, 709\, 163$, and $X_1 = 1$. 
Since $X_0 \in [1,e^{J_0}=e^{27}]$, we proceed as described in Section \ref{Case4} and obtain %which by the proof of \cref{MainResult:theta}, places us in Case 4. Thus by \eqref{mk4} we have 
\[
 m_3 = \max(\mathcal{A}_3(3\,400), \max_{27\le b_i \le 3400}( B_{i*,3} ),\mathcal{C}_{27,3},\mathcal{D}_3(19\, 035\, 709\, 163,e^{27})).
\]
%To prove \eqref{AxlerUpperBound}, we apply \cref{MainResult:theta} with $k = 3$ and $X_1 = 1$. 
Since $X_1 \le e^J = 10^{19}$, we proceed as in Section \ref{Case3} and obtain %which by the proof of \cref{MainResult:theta} places us in Case 3. Thus by \eqref{Mk3}, we have
\[
 M_3 = \max(\mathcal{A}_3(3\,400),\max_{27 \le b_i \le 3400}( B_{b_i,3} ) ).
\]
Since
$\mathcal{A}_3(3\,400) = 2.1719 \cdot 10^{-2}$, 
$\max_{27\le b_i \le 3400}( B_{b_i,3}) = B_{J,3} = 2.4334 \cdot 10^{-2}$, 
$\mathcal{C}_{27,3} = 5.0536 \cdot10^{-2}$, and  
$\mathcal{D}_3(19\, 035\, 709\, 163,e^{27})=0.15$, then 
\[m_3 = 0.15\ \text{and } \ M_3 = 2.4334 \cdot 10^{-2}.\]
\end{proof}

In our next examples we provide valid proofs of Theorem 4.2 of \cite{Dus4}. 
The proof of the main theorem in \cite{Dus4} is incorrect as it applies an incorrect result of Ramar\'{e} \cite{Ram}. 
The stated leading constant in Theorem 1.1 of \cite{Ram} is off by a factor of over 100. The results given here are optimal, verifiable through computation.
\begin{corollary}\cite[Theorem 4.2]{Dus4} We have
 \begin{equation}
  |\vartheta(x) - x| < M_k \frac{x}{\log^k x} \quad \text{for }x \ge X_0
 \end{equation}
with 
\begin{center}
\begin{tabular}{|l|lllllll|}
 \Xhline{4\arrayrulewidth}
 $k$ & $0$ & $1$ & $1$ & $2$ & $2$ & $2$ & $2$ \\
 $M_k$ & $1$ & $1.2323$ & $0.001$ & $3.965$ & $0.2$ & $0.05$ & $0.01$ \\
 $X_0$ & $1$ & $2$ & $908\ 994\ 923$ & $2$ & $3\ 594 \ 641$ & $112\ 568\ 683$ & $7\ 713\ 113\ 853$ \\
 \hline
\end{tabular}
\end{center}
and
\begin{center}
\begin{tabular}{|l|llllll|}
 \Xhline{4\arrayrulewidth}
 $k$ & $3$ & $3$ & $3$ & $3$ & $3$ & $4$ \\
 $M_k$ & $20.83$ & $10$ & $1$ & $0.78$ & $0.5$ & $151.3$ \\
 $X_0$ & $2$ & $32\ 321$ & $89\ 967\ 803$ & $158\ 822\ 621$ & $767\ 135\ 587$ & $2$ \\
 \hline
\end{tabular}
\end{center}
\end{corollary}
\begin{proof} 
\ \\
%First, observe that the  results with $x_k \le 2$ immediately follow from  Tables \ref{Aktable}, \ref{BkMaxTableWedeniwski}, \ref{CkValues}, and \ref{Table-Dk} below.
%\todo[inline,color=red!30]{Perhaps we could condense many of these, as the methodology is the same with different numerical confirmation, which can't be fully listed anyways.}
\begin{itemize}
\item Observe that the  results with $x_k \le 2$ immediately follow from  Tables \ref{Aktable}, \ref{BkMaxTableWedeniwski}, \ref{CkValues}, and \ref{Table-Dk} below.

%\item For $m_2,M_2=0.05$, the bound fails at $x=122\ 467\ 591$. The result becomes valid again with either $m_2,M_2=0.05062$, or $x_2=122\ 568\ 683$. \\
%% Loading GeneralAxler.pari and running Gbackwards(2,0.05,2*10^8) reconfirms this result.
%
%\item For $m_2,M_2=0.01$, the bound fails at $x=7\ 713\ 117\ 761$. The result becomes valid again with either $m_2,M_2=0.010039$, or $x_2=7\ 713\ 133\ 853$. \\
%% Loading GeneralAxler.pari and running Gbackwards(2,0.01,8*10^9) reconfirms this result.
\item For $m_1,M_1=0.001$, numerical confirmation gives validity for $x \in [908\ 994\ 923, 10^{10})$. \cref{Table-Dk} then gives validity for $x \in [10^{10},7\cdot 10^{11})$, which \cref{CkValues} extends to $10^{19}$. From here, combining \cref{BkiTableWedeniwski} and \cref{Aktable} gives validity for all $x \ge 10^{19}$. 

\item For $m_2,M_2=0.2$, numerical confirmation gives validity for $x \in [3\ 594\ 641, 10^{8})$. \cref{Table-Dk} then gives validity for $x \in [10^{8},7\cdot 10^{11})$, which \cref{CkValues} extends to $10^{19}$. From here, combining \cref{BkiTableWedeniwski} and \cref{Aktable} gives validity for all $x \ge 10^{19}$. 

\item For $m_3,M_3=10$, numerical confirmation gives validity for $x \in [32\ 321,10^{8})$. \cref{Table-Dk} then gives validity for $x \in [10^{8},7\cdot 10^{11})$, which \cref{CkValues} extends to $10^{19}$. From here, combining \cref{BkiTableWedeniwski} and \cref{Aktable} gives validity for all $x \ge 10^{19}$. 

\item For $m_3,M_3=1$, numerical confirmation gives validity for $x \in [89\ 967\ 803,10^{8})$. \cref{Table-Dk} then gives validity for $x \in [10^{8},7\cdot 10^{11})$, which \cref{CkValues} extends to $10^{19}$. From here, combining \cref{BkiTableWedeniwski} and \cref{Aktable} gives validity for all $x \ge 10^{19}$. 

\item For $m_3,M_3=0.78$, numerical confirmation gives validity for $x \in [158\ 882\ 621,10^{10})$. \cref{Table-Dk} then gives validity for $x \in [10^{10},7\cdot 10^{11})$, which \cref{CkValues} extends to $10^{19}$. From here, combining \cref{BkiTableWedeniwski} and \cref{Aktable} gives validity for all $x \ge 10^{19}$. 

\item For $m_3,M_3=0.5$, numerical confirmation gives validity for $x \in [767\ 135\ 587,10^{10})$. \cref{Table-Dk} then gives validity for $x \in [10^{10},7\cdot 10^{11})$, which \cref{CkValues} extends to $10^{19}$. From here, combining \cref{BkiTableWedeniwski} and \cref{Aktable} gives validity for all $x \ge 10^{19}$. 
\end{itemize}
\end{proof}
\noindent {\it Acknowledgements}.   
This research was supported by NSERC Discovery grants of authors H.K. and N.N..
 S.B. was supported by NSERC USRA grants in Summers 2017, 2018, and 2020,   
K.W. was supported by an NSERC USRA grant in Summer 2017, and A.L. was supported by a York University graduate fellowship. Thank-you to Alia Hamieh and Peng-Jie Wong for helping with the supervision of S.B. and K.W. during part of this project. 
Thank-you to Andrew Fiori and Josh Swidinsky for their independent verification of the calculations in Appendix B. 

\ 
\pagebreak
\appendix
\section{Sharper bounds for $\psi(x)$: $|\psi(x) - x| < \varepsilon x$ where $\varepsilon$ is computable for $x \ge x_0$ fixed}
\label{psixboundssection}
%\\
%He also proved that 
%$$
%  ax \le \theta(x) \le bx \text{ for all } x > 0. 
%$$
%The current best published result for an explicit bound on $\theta(x)$ for $x > 0$ is due to Rosser and Schoenfeld \cite{Sch}:
%\begin{equation}
% 0 < \theta(x) < 1.001093x
%\end{equation}
In this section we explain how to bound  $E(x) :=(\psi(x)-x)/x$. 
%\[
%% \label{E(x)defn}
%E(x) = \frac{|\psi(x)-x|}{x}.  
%\]
A classic explicit formula that relates prime numbers
  to the  non-trivial zeros of $\zeta$ is given by \cite[$\S 17$, $(1)$]{Dav}:
\begin{equation}
%  \label{explicit}
  \label{classical} 
 \psi(x)=x-\sum_{\rho}\frac{x^{\rho}}{\rho}-\log2\pi-\frac12\log(1-x^{-2}), 
\end{equation}
when $x$ is not a prime power. 
The sum over the zeros is not absolutely convergent, hence it is difficult to directly use this formula to bound $E(x)$.
Rosser \cite{Ros}, \cite{Ros2} introduced an averaging technique to gives bounds for $\psi(x)$.  The averaging produces a formula 
like \eqref{classical}, but with an absolutely convergent sum over zeros of $\zeta$. He later joined forces with Schoenfeld \cite{RS2} and they streamlined and improved his arguments.  
%Some of the main tools used in the work of Rosser and Schoenfeld
%were: a numerical verification of the Riemann hypothesis (RH), an explicit zero-free region for the Riemann zeta function, 
%an explicit version of the number of zeros $N(T)$ of the Riemann zeta function, explicit evaluation of sums of zeros,
%and a smoothing argument for prime sums. 
Since their last article in 1975 improvements to their work have been based on the following:
\begin{enumerate}
\item[1.] An improved verification of the partial Riemann hypothesis $\text{RH}(H)$: 
namely all the non-trivial zeros $\varrho = \beta+i\gamma$ of the Riemann zeta function which have imaginary part $|\gamma|\le H$ lie on the 1/2-line.
\item[2.] An improved zero-free region for $\zeta(s)$: 
the zeta function has no zeros in the complex region 
$\{ \sigma + it \in \C : \sigma \ge 1- \frac{1}{R\log|t|}, |t| \ge 2 \}$.
\item[3.] A better smooth weight: 
$|\psi(x) -\psi_{\varphi}(x)|$ is small, where $\psi_{\varphi}(x)$ is a smooth variant of $\psi(x)$.
\item[4.] An explicit zero-density result: 
for $\sigma$ and $T$ fixed, we define 
\begin{equation}
  \label{NsT}
  N(\s,T) = \# \{ \varrho=\beta+i\gamma \ | \ \beta \ge \sigma \text{ and } 
  0 \le \gamma \le T \},
\end{equation}
the number of non-trivial zeros $\varrho = \beta+i \gamma$ of zeta with $\beta \ge \sigma$
and imaginary part $\gamma $ between $0$ and $T$.  
We have an explicit bound of the form 
$ N(\sigma,T) \le c(\sigma) T^{a(\sigma)}(\log T)^{b(\sigma)}$,  
where $a, b$ and $c$ are certain functions of $\sigma$.
\end{enumerate}
%Dusart \cite{Dus} obtained improvements by inserting new partial verifications of the Riemann hypothesis into their arguments and a better zero-free region constant.  
%For instance, in 1962, Rosser and Schoenfeld \cite{RS1} used the verification of RH to height $H_0=e^{9.99}=21807.299\ldots$ and zero-free region constant $R=17.51\ldots$,
%while in 1975 Rosser and Schoenfeld \cite{RS2} used the values $H_0=1894438.51\ldots$ and $R=9.6459 \ldots$ whereas in 2016
%Dusart used $H_0 =2\,445\,999\,556\,030.34$ and $R=5.573412$.  
In recent years there has been a lot of activity on bounding $E(x)$
(see \cite{But}, \cite{But2}, \cite{Dus5}, \cite{Dus3}, \cite{Dus4}, \cite{FaKa}, \cite{PT2019}). 
In 2015 the first theoretical improvements to Rosser and Schoenfeld's method were provided by  Faber and Kadiri  \cite{FaKa}.  They introduced the idea of smoothing  to the problem and it was demonstrated that the averaging technique used in \cite{Ros} and \cite{RS2} could be interpreted as a  particular case of smoothing.
In addition, the use of an explicit zero-density result was first applied in \cite{FaKa}. 
Faber and Kadiri's smooth functions gave better results for $x\ge x_0$ for any $x_0 \le e^{4\,000}$ (as compared with the Rosser and Schoenfeld method used in \cite{Dus3}). 
In $2016$, B\"{u}the \cite{But} used a different smoothing: he introduced the Logan function which puts more weight  on the first zeros for which RH has been verified. This method did not appeal to a zero-free region or zero-density and it worked better than \cite{FaKa} for $e^{50} \le x_0 \le e^{3\,000}$. 
In addition, for $x \le 10^{19}$ B\"{u}the \cite[Theorem 2]{But2}  gave an explicit numerical bound for 
$(\psi(x)-x)/\sqrt{x}$.
In 2018 Dusart \cite{Dus4} used the Faber-Kadiri method along with a recent zero-density result of Ramar\'{e} \cite{Ram}. 
It provided new bounds for $\psi(x)$, $\theta(x)$, and other prime counting sums \footnote{Unfortunately the main theorem of \cite{Ram} 
is incorrect and thus bounds claimed in \cite{Dus4} are likely affected, in particular  
Theorems 3.2, 3.3 and Table 1 for bounds for $\psi(x)$, and consequently Theorem 4.2 and Table 2 for $\theta(x)$.
In addition this unfortunately affects the main theorem \cite[Theorem 1.1]{Axler17_2}.}.
Recently,  Platt and Trudgian employed Perron's formula along with the zero-density result
 of \cite{KLN} to give the best results in the range $x \ge X_0 :=e^{2314}$.    
 Currently, the results of \cite{But2}, \cite{But}, and \cite{PT2019} provide the best explicit bounds for $E(x)$ in the ranges $[1,10^{19}]$, $[10^{19}, X_0]$, and $[X_0,\infty)$ respectively. 
 In this section, we apply the techniques of these articles to bound $E(x)$.  
 
 The table below displays some of the historical improvements concerning the zeros of zeta that have been applied in obtaining sharper bounds for $\psi(x), \theta(x)$, and $\pi(x)$.
\begin{center}
\begin{table}[h] \caption{}
\label{hist}
\begin{tabular}{|l|l|l|l|l|} 
\hline
 Author & $H$ & $R$ %& $[a(\s),b(\s)]$ 
 & $\varepsilon(100)$ & $\varepsilon(6\,000)$
 \\ \hline
 Rosser \& Schoenfeld (1962) \cite{RS1} & $e^{9.99}\simeq21\,807$ \cite{Lehm} & $17.52$ \cite{RS1} %& $[1,1]$
 & $ 9.97\cdot10^{-4}$ & $4.92\cdot10^{-6}$\\
 Rosser \& Schoenfeld (1975) \cite{RS2} & $1\,894\,438$ \cite{RS2} & $9.65$ \cite{Ste} \cite{RS2}  %& $[1,1]$ 
 &  $ 1.70\cdot10^{-5}$ & $3.67\cdot10^{-9}$\\
 Dusart (1999) \cite{Dus} & $545\,439\,823$ \cite{vdL} & $9.65 $ \cite{Ste} \cite{RS2} %& $[1,1]$ 
 &  $ 9.00\cdot10^{-8}$ & $2.41\cdot10^{-9}$\\ 
 Faber \& Kadiri (2015) \cite{FaKa} \cite{FaKaCorr} & $2\,445\,999\,556\,030$ \cite{Wed} & $5.70 $ \cite{Kad} & %$[1,0]$ & 
 $ 2.42\cdot10^{-11}$ & $9.68\cdot10^{-14}$\\ 
Dusart (2016) \cite{Dus3} & $2\,445\,999\,556\,030$ \cite{Wed} & $5.58 $ \cite{TrudMoss}  & %$[1,1]$ & 
 & $6.77\cdot10^{-14}$\\  
B\"uthe (2016) \cite{But} & - & - %& - 
& $ 2.46\cdot10^{-12}$ & \\ 
 \hline
\end{tabular}
\end{table}
\end{center}
%In  \cite{KLN} some of the authors (Kadiri, Lumley, and Ng) improved upon the (corrected) zero-density from \cite{Ram}. Combining these with \cite{FaKa}'s smooth weight, we use it to establish new bounds for $\psi(x)$and this is the first main result of this article.
%%
%\begin{theorem} \label{psitheorem}
%Let $20 \le b \le 13\,900$ be a fixed positive constant. Let $x \ge e^b$. Then there exists $\varepsilon(b) >0$ such that 
%\[
%  | \psi(x)-x| \le \varepsilon(b)  x.
%\] 
%Values for $b$ and $\varepsilon(b)$ are given in Table %s \ref{Plapsixvals1}and 
% \ref{Wedpsixvals} of Appendix \ref{Section:Tables}.
%\end{theorem}
%Theorem \ref{psixmainthm} gives the definition of $\varepsilon(b)$, with the proof available in Section \ref{psixboundssection}. % {\color{red}and \cite{FaKa,FaKaCorr}?}.
\begin{remark}
In Table \ref{Wedpsixvals} we recorded the best values for $\varepsilon(b)$ that we computed using all various methods known up to today.  
%In comparison to the results listed in Table \ref{hist}, 
We found that \cite{But} works best for all $b< 2400$ at which point \cite{PT2019}  works best. 
For instance, it gives $\varepsilon(6\,000) = 6.45\cdot10^{-16}$, while using B\"uthe's theorem gives $1.91\cdot10^{-12}$.
%\[\varepsilon_{100} = 2.46\cdot10^{-12} \text{ and }\varepsilon_{6\,000} = 6.45\cdot10^{-16}.\]
\end{remark}
%we have the following Corollary. 
%\begin{corollary}
%
%%For all  $x  \ge  e^{20}$, 
%%$$
%%   |\psi(x)-x| \le 5.36884 \cdot 10^{-4} x.
%%$$ 
%\end{corollary}
% This slightly improves $5.3688 \times 10^{-4}$ of \cite{FaKa}. 
%This improves the results of \cite{FaKa} \cite{FaKaCorr} \cite{Dus3} for any value $x$ and of \cite{But} for larger values of $x$. 
%Table \ref{Wedpsixvals} provides best values for $\psi(x)$ known as of today.  
%\cite{FaKaCorr}, \cite{But}, \cite{Dus3} and \cite{Dus4}. \footnote{{\color{red} Discussion about improvements here.} To do by HK}
-------------------------------------------\\

\subsection{Zeros of the Riemann zeta function}
We list here the effective results currently available for the zeros of the Riemann zeta function and which we use to obtain  Theorem \ref{psixmainthm}.
%\subsubsection{Calculation of $s_0=\sum_{0<\gamma<T_0} \frac1{\gamma}$}
%We take the values $ T_0=1\,132\,491$, from \cite{FaKa} which gives 
%$s_0 =  11.637732$.
%
\subsubsection{Partial verification of RH}
%In Extensive calculations have been undertaken to test RH to ever increasing heights, with Gourdon having checked 1013 zeros [6] using an algorithm first described by Odlyzko and Scho?nhage [14]
% We use the latest computations of Platt \cite{Pla0} \cite{Pla} concerning the verification of RH
%:\begin{theorem}\label{RH}
% and thus take $H_0=3.061\cdot10^{10}$.
%If $\zeta(s)=0$ at $0\le\Re(s)\le1$ and $0\le \Im(s) \le H$, then $\Re(s)=\frac12$. \end{theorem}
% Table  \ref{Plapsixvals1}
%  presents values of $\varepsilon(b_0)$ computed for this value of $H_0 $.
% Prior to the work of Platt, Gourdon \cite{Gou} and Wedeniwski \cite{Wed} announced a verification up to $H_0 =2\,445\,999\,556\,030$.
In this article, we use the verification up to 
\begin{equation}
  \label{H}
H =2\,445\,999\,556\,030
\end{equation} announced by Gourdon \cite{Gou} and Wedeniwski \cite{Wed}.  In \cite{Pla-zeta}
the value of $H=3.061\cdot10^{10}$ was established and this was considered the most rigorous verification of RH
and a number of articles preferred to use this value instead of the value in \eqref{H}.
After this paper was submitted, an article of Platt and Trudgian \cite{PT2020}  appeared which now verifies that 
$H=3000175332800$ is valid.   This work makes use of the more rigorous techniques as 
in \cite{Pla-zeta} and thus confirms the earlier calculations in  \cite{Gou} and \cite{Wed}.
However, the focus of this paper is about the method of computation, and as 
Theorem \ref{psixmainthm} was calibrated so that the verification height for RH is a parameter,  we can obtain results regardless of which height is used.  Work in progress, will establish new bounds for $\psi(x)$ making use of the new values of $H$ from \cite{PT2020}.  Moreover, Andrew Fiori and Habiba Kadiri have developed a webpage which provides bounds
for $\psi(x)$ and $\theta(x)$ given any set of parameters $H$, $R$. 
 % we also give a version of Theorem \ref{psixmainthm} based on this value (see Table \ref{Wedpsixvals}).
 % Theorem \ref{psixmainthm} was calibrated so that the verification height for RH is a parameter, thus we can obtain results %regardless of which height is used.

\subsubsection{Explicit zero-free region}
We use the following type of  zero-free region for $\zeta(s)$: \\
Assume there exists $R \ge 1$ such that $\zeta(\s+it)$ does not vanish when
 \begin{equation}
  \label{zfr2} \s \ge 1-\frac1{R   \log |t|},\ \text{ for every }\  |t|\ge2.
 \end{equation}
We use this with $R =5.573412$ as established in \cite[Theorem 1]{TrudMoss}.
%
%\subsubsection{Explicit bound for the number of zeros}
%For the estimate of $N(T)$, we use 
%%$a_1=0.111$, $a_2=0.275$ and $a_3=2.450$ due to 
%%Trudgian \cite{Trudgian2}.
%due to Rosser \cite{Ros} $a_1=0.137,a_2=0.443,a_3=1.588$ from \cite{Ros}.
%Trudgian \cite{Trudgian} improves these $a_i$: $a_1=0.111$, $a_2=0.275$ and $a_3=2.450$. The numerical results displayed in this paper are valid with either one.
%were tabulated twice, once with Rosser's values and again with Trudgian's. The displayed $\varepsilon(b_0)$'s are valid with either.
%%
%From \cite[Theorem 1]{TrudMoss} we have the zero-free region:
%\begin{theorem} \label{ZFR}
%Let $R_0 =5.573412$. 
%Then there are no zeros of $\zeta(s)$ in the region
%\[\Re s \ge 1-\frac1{R_0 \log|\Im s|}\text{ and } |\Im s|\ge2.\]
%\end{theorem}
%%
%Let $T\ge 2$ and $N(T)$ be the number of non-trivial zeros $\varrho =\beta+i\gamma$ in the region $0 \le \gamma \le T$ and $0\le\beta\le1$.
%%
%In $1941$, Rosser \cite[Theorem 19]{R_0 } proved  
%\begin{theorem}\label{Rosser41}
%Let $T\ge 2$, 
%\[
%P(T)  = \frac T{2\pi}\log \frac T{2\pi} - \frac T{2\pi} + \frac{7}{8} ,\ 
%R( T) = a_1 \log T + a_2 \log \log T + a_3,
%\]  
%% 
%and $\displaystyle{a_1 = 0.137}$, $\displaystyle{a_2 =0.443}$, $\displaystyle{a_3 = 1.588}$.
%Then
%\[ |N(T) - P(T) | \le  R( T).\]
%\end{theorem}
%%
%\end{itemize}
\subsubsection{Explicit zero-density for zeta }
We recall that $N(\s,T)$ is the number of non-trivial zeros in the region $\s \le \Re(s) \le 1$ and $0\le \Im(s) \le T$.
In \cite[Theorem 1.1]{KLN} the following explicit upper bounds for $N(\s,T)$ were established.
%\begin{theorem}\cite[Theorem 1.1]{KLN} \label{thm-NsigT}
%Let $H_0$ satisfy \eqref{hypothese-RHver}. Let  $k \ge \frac{10^9}{H_0 },d>0, H\in [1002,H_0 )$ be fixed. Let $\s > \frac12 +\frac{d}{\log H_0 }$.\\
%Then there exist $\mathcal{C}_1 ,\mathcal{C}_2 >0$ such that, for any $T\ge H_0 $,
%\begin{align*}
%& %\label{KLN-NsigT1}
%N(\s,T)
%\le \frac{(T-H)(\log T)}{2\pi d}
%\log
%\Big( 1 
%+\frac{  \mathcal{C}_1
%(\log (kT))^{2\s}  (\log T)^{4(1-\s)} T^{\frac83(1-\s)}
% }{T-H}  \Big)
%+ \frac{ \mathcal{C}_2  }{2\pi d} (\log T)^2 ,
% \\%\text{ and}\ 
%& %\label{KLN-NsigT2}
%N(\s,T)
%\le 
% \frac{ \mathcal{C}_1}{2\pi d}
%(\log (kT))^{2\s}  (\log T)^{5-4\s } T^{\frac83(1-\s)}
%+ \frac{ \mathcal{C}_2  }{2\pi d} (\log T)^2 ,
% \end{align*}
%where $\mathcal{C}_1$ and $\mathcal{C}_2$ are defined in \cite[Lemma 4.14]{KLN}. 
% %Tables \ref{g2asymptotic} and \ref{g1optatH0}.
%\end{theorem}
\begin{theorem} \cite[Theorem 1.1]{KLN}  \label{thm-NsigT}
Let  $\frac{10^9}{H} \le k \le 1,d>0, Y\in [1002,H)$, $\a>0$, $\del\ge1$, $\eta_0=0.23622 \ldots$, $1+\eta_0\le \mu \le 1+\eta$,
and $\eta\in(\eta_0,\tfrac12)$ be fixed. Let $\s > \frac12 +\frac{d}{\log H}$.\\
Then there exist $\mathcal{C}_1 ,\mathcal{C}_2 >0$ such that, for any $T\ge H$,
\begin{equation}
\begin{split}
 \label{KLN-NsigT1}
N(\s,T)
& \le \frac{(T-Y)(\log T)}{2\pi d}
\log
\Big( 1 
+\frac{  \mathcal{C}_1
(\log (kT))^{2\s}  (\log T)^{4(1-\s)} T^{\frac83(1-\s)}
 }{T-Y}  \Big)
+ \frac{ \mathcal{C}_2  }{2\pi d} (\log T)^2 , \\
N(\s,T)
& \le 
 \frac{ \mathcal{C}_1}{2\pi d}
(\log (kT))^{2\s}  (\log T)^{5-4\s } T^{\frac83(1-\s)}
+ \frac{ \mathcal{C}_2  }{2\pi d} (\log T)^2
 %\text{ and}\ 
\end{split}
\end{equation}
where $\mathcal{C}_1= \mathcal{C}_1(\alpha, d,\del, k, Y, \s)$ and $\mathcal{C}_2=\mathcal{C}_2 (d, \eta, k, Y, \mu,\s)$ are defined in  \cite[Lemma 4.14]{KLN}. 
%Since $\log(1+x) \le x$ for $x \ge 0$, \eqref{KLN-NsigT1} implies
%\begin{equation}
%\label{KLN-NsigT2}
%N(\s,T)
%\le 
% \frac{ \mathcal{C}_1}{2\pi d}
%(\log (kT))^{2\s}  (\log T)^{5-4\s } T^{\frac83(1-\s)}
%+ \frac{ \mathcal{C}_2  }{2\pi d} (\log T)^2.
% \end{equation}
\end{theorem}
From this we have the following corollary.
\begin{corollary}
Let $\sigma \in [0.75,1)$.  Then there exist $c_1(\s), c_2(\s) >0$ such that 
\begin{equation}\label{ZD-KLN}
N(\sigma, T) \leq c_1(\s)  T^{8(1-\sigma) / 3} (\log T)^{5-2 \sigma}+c_2(\s)  (\log  T)^2 
\end{equation}
where $c_1,c_2$ are given in the table below:
\end{corollary}
\begin{small}
\begin{table}[h]
 \centering
\caption{The bound $N(\s,T)
\le c_1 T^{\frac83(1-\s)} (\log T)^{5-2\s} 
+ c_2 (\log T)^2$  \label{bnd-Nasymp}.}
\label{g2asymptotick}
\begin{tabular}{|c|c|c|c|c|c|c|}
\hline
$\s_0$ & $\mu$ &$\alpha$ &  ${\del} $ & $d$ &  $c_1=\frac{\mathcal{C}_1}{2\pi d}$ & $c_2=\frac{\mathcal{C}_2}{2\pi d}$ \\
 \hline
%$ 0.60 $&$ 0.5 $&$ 1.251 $&$ 0.288 $&$ 0.3140 $&$ 0.341$&$ 2.177 $&$ 5.663 $\\
%$ 0.65 $&$ 0.6 $&$ 1.249 $&$ 0.256 $&$ 0.3070 $&$ 0.340$&$ 2.963 $&$ 5.249 $\\
%$ 0.70 $&$ 0.8 $&$ 1.247 $&$ 0.222 $&$ 0.3040 $&$ 0.339$&$ 3.983 $&$ 4.824 $\\
$ 0.75 $&$ 1.245 $&$ 0.189 $&$ 0.3030 $&$ 0.338$&$ 5.277 $&$ 4.403 $\\
$ 0.80 $&$ 1.245 $&$ 0.160 $&$ 0.3030 $&$ 0.337$&$ 6.918 $&$ 3.997 $\\
$ 0.85 $&$ 1.245 $&$ 0.133 $&$ 0.3030 $&$ 0.336$&$ 8.975 $&$ 3.588 $\\
$ 0.86 $&$ 1.245 $&$ 0.127 $&$ 0.3030 $&$ 0.335$&$ 9.441 $&$ 3.514 $\\
$ 0.87 $&$ 1.245 $&$ 0.122 $&$ 0.3030 $&$ 0.335$&$ 9.926 $&$ 3.430 $\\
$ 0.88 $&$ 1.245 $&$ 0.116 $&$ 0.3030 $&$ 0.335$&$ 10.431 $&$ 3.346 $\\
$ 0.89 $&$ 1.245 $&$ 0.111 $&$ 0.3030 $&$ 0.335$&$ 10.955 $&$ 3.262 $\\
$ 0.90 $&$ 1.245 $&$ 0.105 $&$ 0.3030 $&$ 0.334$&$ 11.499 $&$ 3.186 $\\
$ 0.91 $&$ 1.245 $&$ 0.100 $&$ 0.3030 $&$ 0.334$&$ 12.063 $&$ 3.102 $\\
$ 0.92 $&$ 1.245 $&$ 0.095 $&$ 0.3030 $&$ 0.334$&$ 12.646 $&$ 3.017 $\\
$ 0.93 $&$ 1.245 $&$ 0.089 $&$ 0.3030 $&$ 0.333$&$ 13.250 $&$ 2.941 $\\
$ 0.94 $&$ 1.245 $&$ 0.084 $&$ 0.3030 $&$ 0.333$&$ 13.872 $&$ 2.856 $\\
$ 0.95 $&$ 1.245 $&$ 0.079 $&$ 0.3030 $&$ 0.333$&$ 14.513 $&$ 2.772 $\\
$ 0.96 $&$ 1.245 $&$ 0.074 $&$ 0.3030 $&$ 0.332$&$ 15.173 $&$ 2.694 $\\
$ 0.97 $&$ 1.245 $&$ 0.069 $&$ 0.3030 $&$ 0.332$&$ 15.850 $&$ 2.609 $\\
$ 0.98 $&$ 1.245 $&$ 0.064 $&$ 0.3030 $&$ 0.331$&$ 16.544 $&$ 2.532 $\\
$ 0.99 $&$ 1.245 $&$ 0.060 $&$ 0.3030 $&$ 0.331$&$ 17.253 $&$ 2.446 $\\
\hline
\end{tabular}
\end{table}
\end{small}

\subsection{Platt and Trudgian's bounds for $\psi(x)$:}
\subsubsection{The Prime Number Theorem  with a small constant error term}
In \cite{PT2019}, Platt and Trudgian use an explicit version of Perron's formula proven by Dudek \cite[Theorem 1.3]{Dud}: Let $x \ge e^{1000}$ and $T$ satisfies $50 < T \le x$.  Then  
\begin{equation}\label{eq0}
\frac{\psi(x)-x}{x} = \sum_{|\gamma|<T} \frac{x^{\rho-1}}{\rho} + \mathcal{O}^{*} \left(\frac{2 (\log  x)^2}{T}\right)
\end{equation}
where $A= \mathcal{O}^{*}(B)$ means $|A| \le B$.  
%{\color{red} Not sure correct.  Has to be checked. Dudek claims $e^{60}$ works.  I think $e^{1000}$ may work. }
Writing $b=\log x$, we denote 
\begin{equation}\label{def-s0}
 s_0(b, T)= \frac{2 b^2 }{T}.
\end{equation}
The sum over the zeros is then split vertically 
%\todo{we are splitting horizontally, no? It's a vertical split.NN} 
at a fixed value $1-\delta$ with $0.001\le \d \le 0.025$.
\begin{equation}\label{eq1}
 \sum_{|\gamma|<T} \frac{x^{\rho-1}}{\rho}  = \Sigma_1 + \Sigma_2, \ \text{with}\ 
\Sigma_1 = \sum_{|\gamma| \leq T \atop \beta<1-\delta} \frac{x^{\rho-1}}{\rho}, 
\Sigma_2 =\sum_{|\gamma| \leq T \atop \beta\ge 1-\delta} \frac{x^{\rho-1}}{\rho}.
\end{equation}
The first sum $\Sigma_1$ is evaluated in  \cite[Lemma 2.10]{DST} by 
\begin{equation}\label{eq2}
 \left| \Sigma_1 \right| \leq x^{-\delta} \Big(\frac{1}{2 \pi} (\log  (T/2\pi))^{2}+1.8642 \Big).
\end{equation}
We denote 
\begin{equation}\label{def-s1}
 s_1(b,\d,T) = e^{-\delta b}\Big(\frac{1}{2 \pi} (\log  (T/2\pi))^{2}+1.8642 \Big).
\end{equation}
To estimate $\Sigma_2$, an argument of  Pintz  \cite{Pintz} is employed.  The interval $[0,T]$
is split into subintervals $\Big[\frac{T}{\lambda^{k+1}}, \frac{T}{\lambda^k}\Big]$
where $\lambda >1$, $0 \le k \le K-1$, and $K=\Big\lfloor \frac{\log\frac{T}{H}}{\log \lambda}\Big\rfloor+1$.
Using the zero-free region \eqref{zfr2} to bound $\Re(\rho)$ we find 
 \begin{equation}
\left| \Sigma_2 \right|
\leq 2 \sum_{k=0}^{K-1} \frac{\lambda^{k+1}
x^{-\frac1{R\log \left(T/\lambda^{k}\right)} }
 }{T} N\left(1-\delta, \frac{T}{\lambda^{k}}\right).
\end{equation}
Inserting  \eqref{ZD-KLN} we obtain the following:
\begin{equation}\label{eq3}
 \left| \Sigma_2 \right|
\leq 2 \frac{\lambda}{T}  \sum_{k=0}^{K-1} \lambda^{k} x^{-\frac1{R\log \left(T/\lambda^{k}\right)} }
\left( c_1  \left( \frac{T}{\lambda^k} \right)^{\frac{8\delta}3} (\log(T/\lambda^{k}))^{3+2 \delta}  +c_2  (\log(T/\lambda^{k}))^{2}\right).
\end{equation}
We denote 
\begin{multline}\label{def-s2}
 s_2(b,\lambda,K,T) \\= 2 \frac{\lambda}{T}  \sum_{k=0}^{K-1} 
 \exp\left(k\log \lambda -\frac{b}{R(\log T - k\log \lambda)}\right)
\left( c_1  \left( \frac{T}{\lambda^k} \right)^{\frac{8\delta}3} (\log(T/\lambda^{k}))^{3+2 \delta}  +c_2  (\log(T/\lambda^{k}))^{2}\right).
\end{multline}
Finally, putting together \eqref{eq0}, \eqref{eq1}, \eqref{eq2}, and \eqref{eq3} gives the following result. 
\begin{theorem}\label{PlattTrudgian19+}
Let $b_1, b_2$ satisfy $1000 \le b_1 < b_2$.   Let $0.001\le \delta \le 0.025$, $\lambda>1$, $H <T < e^{b_1}$,  and $K=\Big\lfloor \frac{\log\frac{T}{H}}{\log \lambda}\Big\rfloor+1$. 
Then for all $x\in [e^{b_1},e^{b_2}]$
\begin{equation}\label{finalPT}
\left|\frac{\psi(x)-x}{x} \right| \le s_0(b_2,T) +s_1(b_1,\d,T) + s_2(b_1,\d,\lambda,K,T) ,
 \end{equation}
 where $s_0,s_1,s_2$ are respectively defined in \eqref{def-s0}, \eqref{def-s1}, and \eqref{def-s2}.
\end{theorem}
%%%
\subsubsection{The Prime Number Theorem  with an error term of the form 
\eqref{equ:1}.}
Let $x_0\ge 1000$ be fixed, and let $R$ be a constant such that Riemann zeta function does not vanish 
in the region \eqref{zfr2}. 
We define
%\begin{equation} 
%  \label{X0}
% X_0 = \sqrt{\frac{\log x_0}{R}},
%\end{equation}
\begin{align}
  k(\s,x_0) & = ( \exp( ( \tfrac{10-16 \s}{3}) ( \tfrac{\log x_0}{R} )^{\frac{1}{2}}  )   (\sqrt{ \tfrac{\log x_0}{R} })^{5-2\s}    )^{-1} , \\
 c_3(\s,x_0) & = 2 \exp ( -2 ( \tfrac{\log x_0}{R})^{\frac{1}{2}}) (\log x_0)^2  k(\s,x_0), \\
  c_4(\s,x_0) & = x_{0}^{\s-1} ( \tfrac{2 \log x_0}{\pi R} + 1.8642 ) k(\s,x_0) , \\
   c_5(\s,x_0) & =8.01\cdot c_2(\s)  \exp ( - 2 ( \tfrac{\log x_0}{R} )^{\frac{1}{2}}  )  \tfrac{\log x_0}{R} k(\s,x_0).
\end{align}
With the constants defined, we now set 
\begin{equation}
  \label{newA}
  A(\s,x_0) = 2.0025 \cdot 2^{5-2 \s} \cdot c_1(\s) + c_3(\s,x_0) + c_4(\s,x_0) + c_5(\s,x_0), 
\end{equation}
\begin{equation}
 \label{newBC}
 B = \tfrac{5}{2}-\s,  \text{ and } 
 C = \tfrac{16 \s}{3}-\tfrac{10}{3}.
\end{equation}
%\begin{equation}\label{def-A0}
%A_0 := %2 R^2 \exp\left(-\frac{4 \log X_1}{RX_1} - \left(2-2\sqrt{1-\frac{16 }{3RX_1}} \right)X_1\right) \le 
%2 R^2 \exp\left( - \frac{16}{3R} \right) ,
%\end{equation}
%%
%\begin{equation}\label{def-A1}
%A_1 :=  \left(\frac2{ \pi}  + \frac{1.8642}{X_0}\right) \exp \left( -2 \log X_0 -  \frac{16}{3R} \right)  ,
%\end{equation}
%%
%\begin{multline}\label{def-A2}
%  A_2:=  2^{4+\frac4{RX_0}}  e  c_1 X_0^{-1} \exp\left( -\frac12X_0 + \frac{16}{3R} \right) 
% + 2^{5+\frac4{RX_0}} e  c_1 
% \\+ 2^{3} e c_2 X_0^{-2} \exp\left( -\frac12X_0 - \frac{16}{3R} \right)
% + 2^{4} e c_2 X_0^{-1} \exp\left(  - \frac{16}{3R}\right) ,
%\end{multline}
%%
%\begin{equation}
%  \label{def-A}
%  A:= A_0+A_1+A_2, 
%\end{equation}
%%
%\begin{equation}
% \label{def-B}
% B:= 2+\frac{2}{RX_0},
%\end{equation}
%and
%\begin{equation}
% \label{def-C}
% C:= 2\sqrt{1-\frac{16}{3RX_0}}.
%\end{equation}
Then  we have the following result:
\begin{theorem}
Let $x_0 \ge 1000$ and let $\s \in [0.75, 1)$.  
For all $x \ge e^{x_0}$, 
\begin{equation}
 \left| \frac{\psi(x)-x}x \right| \leq A ( \tfrac{\log x}{R} )^{B} \exp ( -C ( \tfrac{\log x}{R} )^{\frac{1}{2}} )
\end{equation}
where $A,B,$ and $C$ are defined in \eqref{newA} and \eqref{newBC}. 
\end{theorem}  
From this and Corollary \ref{psi-theta:ExplicitCor} we deduce 
% \todo[inline,color=red!30]{Notation 1}
\begin{corollary}  \label{thetaxlargeABC}
Let $x_0 \ge 1000$.  
For all $x \ge e^{x_0}$, 
\begin{equation}
  \label{thetaxbound}
 \left| \frac{\theta(x)-x}x \right| \leq A' ( \tfrac{\log x}{R} )^{B}  \exp ( -C ( \tfrac{\log x}{R} )^{\frac{1}{2}} )
\end{equation}
where  $B$ and $C$ are defined in \eqref{newBC} and 
\begin{equation}
  \label{Aprime}
  A' = A \Big( 1 + 
  \tfrac{1}{ A} ( \tfrac{R}{x_0} )^{B}  \exp \Big(C \sqrt{\tfrac{x_0}{R}} \Big)
  \Big( 
  a_1(x_0) \exp(\tfrac{-x_0}{2}) + a_2(x_0) 
  \exp(\tfrac{-2x_0}{3})\Big)
   \Big),
\end{equation}
where $a_1$ and $a_2$ are defined in Corollary 
\ref{psi-theta:ExplicitCor}.
\end{corollary}
\begin{proof}
Let $x \ge e^{x_0}$.  By writing $\theta(x)-x = \psi(x)-x + \theta(x)-\psi(x)$, applying the triangle inequality, 
and invoking Corollary 
\ref{psi-theta:ExplicitCor}, it follows that 
\begin{equation}
\begin{split}
  \left| \frac{\theta(x)-x}x \right| & \le A ( \tfrac{\log x}{R} )^{B} \exp ( -C ( \tfrac{\log x}{R} )^{\frac{1}{2}} )
  + a_1(x_0) x^{-\frac{1}{2}} +a_2(x_0) x^{-\frac{2}{3}} \\
  & \le A ( \tfrac{\log x}{R} )^{B} \exp ( -C ( \tfrac{\log x}{R} )^{\frac{1}{2}} ) \left( 1 + 
  \frac{a_1(x_0) \exp( C  \sqrt{\frac{\log x}{R}})
 }{ A \sqrt{x} \Big( \frac{\log x}{R } \Big)^{B}}
 +  \frac{a_2(x_0) \exp( C  \sqrt{\frac{\log x}{R}})
 }{ A x^{\frac{2}{3}} \Big( \frac{\log x}{R } \Big)^{B}}
   \right).  
\end{split}
\end{equation}
It may be checked the function in brackets decreases for $x \ge e^{x_0}$ with $x_0 \ge 1000$
and thus we obtain \eqref{thetaxbound} with $A'$ given by \eqref{Aprime}.
%Change here above. 
\end{proof}

\subsection{Bounding $\psi(x)$ using B\"{u}the's methods}
For the range $[0,\ e^{2313}]$, the best bounds for $|\psi(x)-x|$ are based on two arguments of B\"{u}the \cite{But},
\cite{But2}.  First, for $100 \le x \le 10^{19}$ B\"{u}the \cite[Theorem 2]{But2} developed an analytic algorithm to compute
$\psi(x)$.  Using this, he showed that 
\begin{equation}
  \label{Butbd}
  -0.94 \le \frac{x-\psi(x)}{\sqrt{x}} \le 0.94. 
\end{equation}
This yields sharp bounds for $|\psi(x)-x)|$ in the range $x \le 10^{19}$.  In \cite{But2}, B\"{u}the used a
smoothing argument similar to \cite{FaKa}. However, instead he used Logan's function, which works 
extremely well in the range  $[10^{19}, e^{2314})$.

%For lower values of $x$, a computational result of B\"{u}the gives improvements for $e^b \le x \le 10^{19}$, namely \cite[(1.5)]{But2}.
First, we give a general statement for an application of  general bounds of the type \eqref{Butbd}.
\begin{lemma}\label{PsiCompAdjusted}
 Let $B_0$, $B$, and $c$ be positive constants such that 
 \begin{equation}
  \label{GeneralXMinusPsi}
  \left| \frac{x - \psi(x)}{\sqrt{x}} \right| \le c \qquad \text{for all } B_0 < x \le B 
 \end{equation}
 is known.  Furthermore,  assume  
for every $b_0 >0$ there exists  $\varepsilon(b_0)>0$ such that 
\begin{equation}
  \label{Psibdassumption}
  |\psi(x) - x| \le \varepsilon(b_0)x \text{ for all } x \ge e^{b_0}.
\end{equation}
Let $b$ be positive such that $e^{b}  \in (B_0,B]$.  
Then, for all $x \ge e^b$ we have
 \begin{equation}
   \label{PsiButhesmallrange0}
  \left| \frac{\psi(x) - x}{x} \right| \le \max \left\{ \frac{c}{e^{\frac{b}{2}}}, \varepsilon(\log B) \right\}.
 \end{equation}
\end{lemma}
\begin{proof}
 Multiplying both sides of \eqref{GeneralXMinusPsi} by $\frac{1}{\sqrt{x}}$ gives
 \[
  \left| \frac{\psi(x) - x}{x} \right| \le \frac{c}{e^\frac{b}{2}} \qquad \text{for all }e^b \le x \le B
 \]
 as $\frac{1}{\sqrt{x}} \le \frac{1}{e^{\frac{b}{2}}}$.
 Then, for $x \ge B$ we apply \eqref{Psibdassumption} with $b_0 = \log B$. 
 Combining these bounds, we derive \eqref{PsiButhesmallrange0}.
%  \[
%   \left| \frac{\psi(x) - x}{x} \right| \le \frac{c}{\sqrt{x}}.
%  \]
%  Then, as $\frac{1}{\sqrt{x}} \le \frac{1}{e^{-\frac{b}{2}}}$, we have
\end{proof}
Using \cite[(1.5)]{But2}, we have
\begin{corollary}
 Let $b$ be a positive constant such that $\log 11 < b \le 19 \log(10)$. Then we have
 \begin{equation}
   \label{PsiButhesmallrange}
  \left| \frac{\psi(x) - x}{x} \right| \le \max \left\{ \frac{0.94}{e^{\frac{b}{2}}}, \varepsilon(19\log 10) \right\} \qquad \text{for all }x \ge e^b.
 \end{equation}
Note that by Table \ref{Wedpsixvals}, we have $ \varepsilon(19\log 10)=
1.93378 \cdot 10^{-8}$. 
\end{corollary}
\begin{proof}
By B\"{u}the  \cite[(1.5)]{But2}, \eqref{GeneralXMinusPsi} holds with 
$B_0=11$, $B=10^{19}$, and $c=0.94$.  Thus we may apply 
Lemma \ref{PsiCompAdjusted} with $B_0=11$, $B = 10^{19}$, and $c=0.94$ from \cite[(1.5)]{But2}
to obtain \eqref{PsiButhesmallrange}.
\end{proof}
% For an intermediate range, the best bounds for $|\psi(x)-x|$ are based on an argument of B\"{u}the \cite{But}. Namely $[e^{49},\ e^{1875}]$ for Platt's value 
% for $H_0$ and $[e^{58},\ e^{2300}]$ for Wedeniwski's value for $H_0$.
% In the range $x\in[e^{b_1},e^{b_2}]$ (where $b_1=49$, $b_2=1475$ when we use Platt's value for $H_0$
% and $b_1=58$, $b_2=1800$ when we use Wedeniwski's value for $H_0$)
% the best bounds for $|\psi(x)-x|$ are based on an argument of B\"{u}the \cite{But}.
We now describe the main theorem in \cite{But}.  
Like \cite{FaKa},  this a smoothing argument.  B\"{u}the considers the  Fourier transform of Logan's function which is a sharp cut-off filter kernel described in \cite{Log88}:
\[\ell_{c,\varepsilon}(\xi)=\frac{c}{\sinh c}\frac{\sin(\sqrt{(\xi \varepsilon)^2-c^2})}{\sqrt{(\xi \varepsilon)^2-c^2}}.\]
Our computations require more values than those provided in \cite{But}, so we use his method to compute more values in these ranges.
% In \cite{But} the author only computes a few values so we fill in the gaps of his table to suit our computational needs. 
\begin{theorem}\cite[Theorem 1]{But}
 Let $0<\varepsilon <10^{-3}$, $c\ge 3$, $x_0\ge 100$ and $\a\in[0,1)$ such that the inequality
 \[B_0:=\frac{\varepsilon e^{-\varepsilon}x_0|\nu_c(\a)|}{2(\mu_{c})_{+}(\a)}>1\]
 holds. We denote the zeros of the Riemann zeta function by $\rho=\beta+i\gamma$ with $\beta,\gamma\in\R$. Then, if $\beta=\tfrac12$ holds for $0<\gamma\le\tfrac{c}{\varepsilon}$, the inequality 
 \[|\psi(x)-x|\le xe^{\varepsilon\a}(\mathcal{E}_1+\mathcal{E}_2+\mathcal{E}_3)\]
 holds for all $x\ge e^{\varepsilon \a}x_0$, where 
 \begin{align*}
 \mathcal{E}_1 &=e^{2\varepsilon}\log(e^{\varepsilon}x_0)\left[\frac{2\varepsilon|\nu_c(\a)|}{\log B_0}+\frac{2.01\varepsilon}{\sqrt{x_0}}+\frac{\log\log(2x_0^2)}{2x_0}\right]+e^{\varepsilon\a}-1,  \footnotemark \\
 \mathcal{E}_2&=0.16\frac{1+x_0^{-1}}{\sinh c}e^{0.71\sqrt{c\varepsilon}}\log(\tfrac{c}{\varepsilon}), \text{ and}\\
 \mathcal{E}_3&=\frac2{\sqrt{x_0}}\sum_{0<\gamma<\tfrac{c}{\varepsilon}}\frac{\ell_{c,\varepsilon}(\gamma)}{\gamma}+\frac2{x_0}.
 \end{align*}
 %\addtocounter{footnote}{1}
 \footnotetext{This term is written without the $e^{\varepsilon\a}-1$ in \cite{But}, during personal communication with the author A.L. discovered this error and are updating the theorem statement to reflect this.}
\end{theorem}
The $\nu_c(\a)=\nu_{c,1}(\a)$ and $\mu_{c}(\a)=\mu_{c,1}(\a)$ 
%Change here. 
%\todo{I suspect all of these should be $\a$} 
where $\nu_{c,\varepsilon}(\alpha)$ and $\mu_{c,\varepsilon}(\a)$ are defined by \cite[p. 2490]{But}.
\newpage
\section{Useful Tables}\label{Section:Tables}
%%
%%%%%%%%%%%%
The tables in this appendix are subsets of longer tables that can be found in the accompanying document \cite{BKLNW}.  Several results in this article make use of
the longer tables. All the values are tabulated using gp-pari and the underlying c libraries. The reported values here are displayed using a function printf which has automatically chosen to round the values. This rounding is done differently depending on the compiler used but it effects only the last digit. 
\subsection{Table for $\psi(x)$: $|\psi(x) - x| < \varepsilon(b,b') x$ for every $e^{b} \le x \le e^{b'}$.}
Here $b'$ is the entry following $b$ in the table below. The result in the last row is valid in the interval $[e^{25000}, e^{26000}]$.
The calculations conducted here are using Wedeniwski's partial verification of the Riemann Hypothesis \cite{Wed}: $H_0=2\,445\,999\,556\,030$. 
Values for $b \in \{20\ldots 2000 \}$ are computed using the method of B\"uthe \cite{But}, and values for $b \in \{ 2500 \ldots  25000 \}$ are computed as in Theorem \ref{PlattTrudgian19+} using the method of Platt-Trudgian \cite{PT2019} .

\begin{table}[h!]
\caption{$|\psi(x)-x|<\varepsilon(b,b') x$ for every $ e^{b} \le x \le e^{b'}$.}
\label{Wedpsixvals}
\begin{minipage}{0.3\textwidth}
\begin{tabular}{|rrrrrr}
\toprule
\multicolumn{1}{c}{\phantom{a}} &
\multicolumn{1}{c}{$b,b'$ } &
\multicolumn{1}{c}{\phantom{a}} &
\multicolumn{1}{c}{\phantom{a}} &
\multicolumn{1}{c}{$\varepsilon(b,b')$ } &
\multicolumn{1}{c}{\phantom{}} \\
\midrule
\multicolumn{6}{|c|}{Computed as in \cite[Theorem 1]{But}}\\
 \midrule
&&&&&\\[-1em]
%%%%%%%%%%%%%%%%%%%%%%%%%%%%%%%%%%%%%%%%%%%%%%%%%%%%%%%%%%%%%%%%%%%%
& $ 20 $ &&& $ 4.26760\cdot 10^{-5 } $ & \\
& $ 21 $ &&& $ 2.58843\cdot 10^{-5 } $ & \\
& $ 22 $ &&& $ 1.56996\cdot 10^{-5 } $ & \\
& $ 23 $ &&& $ 9.52229\cdot 10^{-6 } $ & \\
& $ 24 $ &&& $ 5.77556\cdot 10^{-6 } $ & \\
& $ 25 $ &&& $ 3.50306\cdot 10^{-6 } $ & \\
& $ 30 $ &&& $ 2.87549\cdot 10^{-7 } $ & \\
& $ 35 $ &&& $ 2.36034\cdot 10^{-8 } $ & \\
& $ 40 $ &&& $ 1.93378\cdot 10^{-8 } $ & \\
& $ 45 $ &&& $ 1.09073\cdot 10^{-8 } $ & \\
& $ 50 $ &&& $ 1.11990\cdot 10^{-9 } $ & \\
& $ 100 $ &&& $ 2.45299\cdot 10^{-12 } $ & \\
& $ 200 $ &&& $ 2.18154\cdot 10^{-12 } $ & \\
& $ 300 $ &&& $ 2.09022\cdot 10^{-12 } $ & \\
& $ 400 $ &&& $ 2.03981\cdot 10^{-12 } $ & \\
& $ 500 $ &&& $ 1.99986\cdot 10^{-12 } $ & \\
& $ 600 $ &&& $ 1.98894\cdot 10^{-12 } $ & \\
& $ 700 $ &&& $ 1.97643\cdot 10^{-12 } $ & \\
& $ 800 $ &&& $ 1.96710\cdot 10^{-12 } $ & \\
& $ 900 $ &&& $ 1.95987\cdot 10^{-12 } $ & \\
& $ 1000 $ &&& $ 1.94751\cdot 10^{-12 } $ & \\
& $ 1500 $ &&& $ 1.93677\cdot 10^{-12 } $ & \\
& $ 2000 $ &&& $ 1.92279\cdot 10^{-12 } $ & \\
\bottomrule
\end{tabular}

\end{minipage}\hfill
\begin{minipage}{0.3\textwidth}
\begin{tabular}{|llllll|}
\toprule
\multicolumn{1}{c}{\phantom{a}} &
\multicolumn{1}{c}{$b,b'$ } &
\multicolumn{1}{c}{\phantom{a}} &
\multicolumn{1}{c}{\phantom{a}} &
\multicolumn{1}{c}{$\varepsilon(b,b')$ } &
\multicolumn{1}{c}{\phantom{}} \\
\midrule
\multicolumn{6}{c|}{Computed as in Theorem \ref{PlattTrudgian19+}}\\
 \midrule
&&&&&\\[-1em]
%&\vdots &&& \vdots &\\
& $ 2500 $ &&& $ 9.06304\cdot 10^{-13 } $ & \\
& $ 3000 $ &&& $ 4.59972\cdot 10^{-14 } $ & \\
& $ 3500 $ &&& $ 2.48641\cdot 10^{-15 } $ & \\
& $ 4000 $ &&& $ 1.42633\cdot 10^{-16 } $ & \\
& $ 4500 $ &&& $ 8.68295\cdot 10^{-18 } $ & \\
& $ 5000 $ &&& $ 5.63030\cdot 10^{-19 } $ & \\
& $ 5500 $ &&& $ 3.91348\cdot 10^{-20 } $ & \\
& $ 6000 $ &&& $ 2.94288\cdot 10^{-21 } $ & \\
& $ 6500 $ &&& $ 2.38493\cdot 10^{-22 } $ & \\
& $ 7000 $ &&& $ 2.07655\cdot 10^{-23 } $ & \\
& $ 7500 $ &&& $ 1.96150\cdot 10^{-24 } $ & \\
& $ 8000 $ &&& $ 1.97611\cdot 10^{-25 } $ & \\
& $ 8500 $ &&& $ 2.12970\cdot 10^{-26 } $ & \\
& $ 9000 $ &&& $ 2.44532\cdot 10^{-27 } $ & \\
& $ 9500 $ &&& $ 2.97001\cdot 10^{-28 } $ & \\
& $ 10000 $ &&& $ 3.78493\cdot 10^{-29 } $ & \\
& $ 10500 $ &&& $ 5.10153\cdot 10^{-30 } $ & \\
& $ 11000 $ &&& $ 7.14264\cdot 10^{-31 } $ & \\
& $ 11500 $ &&& $ 1.04329\cdot 10^{-31 } $ & \\
& $ 12000 $ &&& $ 1.59755\cdot 10^{-32 } $ & \\
& $ 12500 $ &&& $ 2.53362\cdot 10^{-33 } $ & \\
 & $ 13000$ &&& $ 4.13554\cdot 10^{-34 } $ & \\
 & $ 13500$ &&& $ 7.21538\cdot 10^{-35 } $ & \\

\bottomrule
\end{tabular}

\end{minipage}\hfill
\begin{minipage}{0.3\textwidth}
\begin{tabular}[t]{llllll|}
\toprule
\multicolumn{1}{c}{\phantom{a}} &
\multicolumn{1}{c}{$b,b'$ } &
\multicolumn{1}{c}{\phantom{a}} &
\multicolumn{1}{c}{\phantom{a}} &
\multicolumn{1}{c}{$\varepsilon(b,b')$ } &
\multicolumn{1}{c}{\phantom{}} \\
\midrule
\multicolumn{6}{c|}{Computed as in Theorem \ref{PlattTrudgian19+}}\\
 \midrule
&&&&&\\[-1em]
%&\vdots &&& \vdots &\\
% & $3*Y_0$ &&& $ 2.54230\cdot 10^{-35 } $ & \\
& $  14000$ &&& $ 1.22655\cdot 10^{-35 } $ & \\
  & $15000$ &&& $ 4.10696\cdot 10^{-37 } $ & \\
  & $16000$ &&& $ 1.51402\cdot 10^{-38 } $ & \\
 & $ 17000$ &&& $ 6.20397\cdot 10^{-40 } $ & \\
 & $ 18000$ &&& $ 2.82833\cdot 10^{-41 } $ & \\
 & $ 19000$ &&& $ 1.36785\cdot 10^{-42 } $ & \\
& $  20000$ &&& $ 7.16209\cdot 10^{-44 } $ & \\
 & $ 21000$ &&& $ 4.11842\cdot 10^{-45 } $ & \\
 & $ 22000$ &&& $ 2.43916\cdot 10^{-46 } $ & \\
 & $ 23000$ &&& $ 1.56474\cdot 10^{-47 } $ & \\
  & $24000$ &&& $ 1.07022\cdot 10^{-48 } $ & \\
 & $ 25000$ &&& $ 7.57240\cdot 10^{-50} $ & \\
 &&&&& \\&&&&& \\&&&&& \\&&&&& \\&&&&& \\&&&&& \\&&&&& \\&&&&& \\&&&&& \\&&&&& \\&&&&& \\
%\phantom{}\\
%\phantom{}\\
%\phantom{}\\
%\phantom{}\\
%\phantom{}\\
%\phantom{}\\
%\phantom{}\\
%\phantom{}\\
%\phantom{}\\
%\phantom{}\\
%\phantom{}\\
\bottomrule
%\phantom{}\\
%\phantom{}\\
%\phantom{}\\
%\phantom{}\\
%\phantom{}\\
%\phantom{}\\
%\phantom{}\\
%\phantom{}\\
%\phantom{}\\
\end{tabular}

\end{minipage}
\end{table}

\pagebreak

\subsection{Sharper bounds for $\theta(x)$: $x$ in a large range ($x \ge e^{2000}$)} \ \

{
\footnotesize
%{\color{blue} Table updated August 11, 2020}.
%{\color{blue} Uses method from paper of Platt-Trudgian (new method).}
\begin{longtable}{*{13}{c}}
% \caption{Values of $\mathcal{A}_k(b)$ as defined in \cref{Cor:Ak}: $|\theta(x) - x| < \frac{\mathcal{A}_k(b) x}{(\log x)^k }$ for $x \ge e^{b}$.}
\caption{$|\theta(x) - x| < \frac{\mathcal{A}_k(b) x}{(\log x)^k }$, for all $x \ge e^{b}$, where $\mathcal{A}_k$ is defined in Corollary \ref{Cor:Ak}.}
\label{Aktable}
\\
\hline
& $b$ && $\mathcal{A}_1(b)$ && $\mathcal{A}_2(b)$ && $\mathcal{A}_3(b)$ && $\mathcal{A}_4(b)$ && $\mathcal{A}_5(b)$ & 
\\ \hline
\endfirsthead
\multicolumn{13}{c}%
{\tablename\ \thetable{} -- continued from previous page} \\
\hline
& $b$ && $\mathcal{A}_1(b)$ && $\mathcal{A}_2(b)$ && $\mathcal{A}_3(b)$ && $\mathcal{A}_4(b)$ && $\mathcal{A}_5(b)$ & 
\\ \hline
\endhead
\hline \multicolumn{13}{r}{Continued on next page} \\ \hline
\endfoot
\hline %\hline
\endlastfoot
%%%%%%%%%%%%%%%%%%%%%%%%%%%%%%%%%%%%%%%%%%%%%%%
&&&&&&&&&&&&\\[-1em]
& $ 1000 $ && $ 1.1919 \cdot10^{-2 } $ && $ 1.1919 \cdot10^{1 } $ && $ 1.1919 \cdot10^{4 } $ && $ 1.1919 \cdot10^{7 } $ && $ 1.1919 \cdot10^{10 } $ & \\
& $ 2000 $ && $ 1.6685 \cdot10^{-6 } $ && $ 3.3369 \cdot10^{-3 } $ && $ 6.6738 \cdot10^{0 } $ && $ 1.3348 \cdot10^{4 } $ && $ 2.6696 \cdot10^{7 } $ & \\
& $ 3000 $ && $ 1.3504 \cdot10^{-9 } $ && $ 4.0512 \cdot10^{-6 } $ && $ 1.2154 \cdot10^{-2 } $ && $ 3.6460 \cdot10^{1 } $ && $ 1.0938 \cdot10^{5 } $ & \\
& $ 4000 $ && $ 2.9283 \cdot10^{-12 } $ && $ 1.1713 \cdot10^{-8 } $ && $ 4.6852 \cdot10^{-5 } $ && $ 1.8741 \cdot10^{-1 } $ && $ 7.4963 \cdot10^{2 } $ & \\
& $ 5000 $ && $ 4.8831 \cdot10^{-15 } $ && $ 2.4416 \cdot10^{-11 } $ && $ 1.2208 \cdot10^{-7 } $ && $ 6.1039 \cdot10^{-4 } $ && $ 3.0520 \cdot10^{0 } $ & \\
% & $ 5100 $ && $ 2.7505 \cdot10^{-15 } $ && $ 1.4028 \cdot10^{-11 } $ && $ 7.1540 \cdot10^{-8 } $ && $ 3.6486 \cdot10^{-4 } $ && $ 1.8608 \cdot10^{0 } $ & \\
% & $ 5200 $ && $ 1.5765 \cdot10^{-15 } $ && $ 8.1977 \cdot10^{-12 } $ && $ 4.2628 \cdot10^{-8 } $ && $ 2.2167 \cdot10^{-4 } $ && $ 1.1527 \cdot10^{0 } $ & \\
% & $ 5300 $ && $ 9.1594 \cdot10^{-16 } $ && $ 4.8545 \cdot10^{-12 } $ && $ 2.5729 \cdot10^{-8 } $ && $ 1.3637 \cdot10^{-4 } $ && $ 7.2272 \cdot10^{-1 } $ & \\
% & $ 5400 $ && $ 5.3785 \cdot10^{-16 } $ && $ 2.9044 \cdot10^{-12 } $ && $ 1.5684 \cdot10^{-8 } $ && $ 8.4692 \cdot10^{-5 } $ && $ 4.5734 \cdot10^{-1 } $ & \\
% & $ 5500 $ && $ 3.1854 \cdot10^{-16 } $ && $ 1.7520 \cdot10^{-12 } $ && $ 9.6357 \cdot10^{-9 } $ && $ 5.2997 \cdot10^{-5 } $ && $ 2.9148 \cdot10^{-1 } $ & \\
% & $ 5600 $ && $ 1.8998 \cdot10^{-16 } $ && $ 1.0639 \cdot10^{-12 } $ && $ 5.9577 \cdot10^{-9 } $ && $ 3.3363 \cdot10^{-5 } $ && $ 1.8684 \cdot10^{-1 } $ & \\
% & $ 5700 $ && $ 1.1399 \cdot10^{-16 } $ && $ 6.4969 \cdot10^{-13 } $ && $ 3.7033 \cdot10^{-9 } $ && $ 2.1109 \cdot10^{-5 } $ && $ 1.2032 \cdot10^{-1 } $ & \\
% & $ 5800 $ && $ 6.8742 \cdot10^{-17 } $ && $ 3.9871 \cdot10^{-13 } $ && $ 2.3125 \cdot10^{-9 } $ && $ 1.3413 \cdot10^{-5 } $ && $ 7.7792 \cdot10^{-2 } $ & \\
% & $ 5900 $ && $ 4.1654 \cdot10^{-17 } $ && $ 2.4576 \cdot10^{-13 } $ && $ 1.4500 \cdot10^{-9 } $ && $ 8.5548 \cdot10^{-6 } $ && $ 5.0474 \cdot10^{-2 } $ & \\
& $ 6000 $ && $ 2.5350 \cdot10^{-17 } $ && $ 1.5210 \cdot10^{-13 } $ && $ 9.1257 \cdot10^{-10 } $ && $ 5.4755 \cdot10^{-6 } $ && $ 3.2853 \cdot10^{-2 } $ & \\
% & $ 6100 $ && $ 1.5490 \cdot10^{-17 } $ && $ 9.4486 \cdot10^{-14 } $ && $ 5.7637 \cdot10^{-10 } $ && $ 3.5159 \cdot10^{-6 } $ && $ 2.1447 \cdot10^{-2 } $ & \\
% & $ 6200 $ && $ 9.5014 \cdot10^{-18 } $ && $ 5.8909 \cdot10^{-14 } $ && $ 3.6524 \cdot10^{-10 } $ && $ 2.2645 \cdot10^{-6 } $ && $ 1.4040 \cdot10^{-2 } $ & \\
% & $ 6300 $ && $ 5.8498 \cdot10^{-18 } $ && $ 3.6854 \cdot10^{-14 } $ && $ 2.3218 \cdot10^{-10 } $ && $ 1.4628 \cdot10^{-6 } $ && $ 9.2152 \cdot10^{-3 } $ & \\
% & $ 6400 $ && $ 3.6145 \cdot10^{-18 } $ && $ 2.3133 \cdot10^{-14 } $ && $ 1.4805 \cdot10^{-10 } $ && $ 9.4752 \cdot10^{-7 } $ && $ 6.0642 \cdot10^{-3 } $ & \\
% & $ 6500 $ && $ 2.2412 \cdot10^{-18 } $ && $ 1.4568 \cdot10^{-14 } $ && $ 9.4687 \cdot10^{-11 } $ && $ 6.1547 \cdot10^{-7 } $ && $ 4.0006 \cdot10^{-3 } $ & \\
% & $ 6600 $ && $ 1.3943 \cdot10^{-18 } $ && $ 9.2021 \cdot10^{-15 } $ && $ 6.0734 \cdot10^{-11 } $ && $ 4.0085 \cdot10^{-7 } $ && $ 2.6456 \cdot10^{-3 } $ & \\
% & $ 6700 $ && $ 8.7028 \cdot10^{-19 } $ && $ 5.8309 \cdot10^{-15 } $ && $ 3.9067 \cdot10^{-11 } $ && $ 2.6175 \cdot10^{-7 } $ && $ 1.7537 \cdot10^{-3 } $ & \\
% & $ 6800 $ && $ 5.4496 \cdot10^{-19 } $ && $ 3.7058 \cdot10^{-15 } $ && $ 2.5199 \cdot10^{-11 } $ && $ 1.7136 \cdot10^{-7 } $ && $ 1.1652 \cdot10^{-3 } $ & \\
% & $ 6900 $ && $ 3.4233 \cdot10^{-19 } $ && $ 2.3621 \cdot10^{-15 } $ && $ 1.6299 \cdot10^{-11 } $ && $ 1.1246 \cdot10^{-7 } $ && $ 7.7596 \cdot10^{-4 } $ & \\
& $ 7000 $ && $ 2.1571 \cdot10^{-19 } $ && $ 1.5100 \cdot10^{-15 } $ && $ 1.0570 \cdot10^{-11 } $ && $ 7.3987 \cdot10^{-8 } $ && $ 5.1791 \cdot10^{-4 } $ & \\
% & $ 7100 $ && $ 1.3633 \cdot10^{-19 } $ && $ 9.6793 \cdot10^{-16 } $ && $ 6.8723 \cdot10^{-12 } $ && $ 4.8794 \cdot10^{-8 } $ && $ 3.4644 \cdot10^{-4 } $ & \\
% & $ 7200 $ && $ 8.6417 \cdot10^{-20 } $ && $ 6.2220 \cdot10^{-16 } $ && $ 4.4799 \cdot10^{-12 } $ && $ 3.2255 \cdot10^{-8 } $ && $ 2.3224 \cdot10^{-4 } $ & \\
% & $ 7300 $ && $ 5.4938 \cdot10^{-20 } $ && $ 4.0105 \cdot10^{-16 } $ && $ 2.9277 \cdot10^{-12 } $ && $ 2.1372 \cdot10^{-8 } $ && $ 1.5602 \cdot10^{-4 } $ & \\
% & $ 7400 $ && $ 3.5025 \cdot10^{-20 } $ && $ 2.5919 \cdot10^{-16 } $ && $ 1.9180 \cdot10^{-12 } $ && $ 1.4193 \cdot10^{-8 } $ && $ 1.0503 \cdot10^{-4 } $ & \\
% & $ 7500 $ && $ 2.2393 \cdot10^{-20 } $ && $ 1.6795 \cdot10^{-16 } $ && $ 1.2596 \cdot10^{-12 } $ && $ 9.4468 \cdot10^{-9 } $ && $ 7.0851 \cdot10^{-5 } $ & \\
% & $ 7600 $ && $ 1.4356 \cdot10^{-20 } $ && $ 1.0910 \cdot10^{-16 } $ && $ 8.2915 \cdot10^{-13 } $ && $ 6.3016 \cdot10^{-9 } $ && $ 4.7892 \cdot10^{-5 } $ & \\
% & $ 7700 $ && $ 9.2275 \cdot10^{-21 } $ && $ 7.1052 \cdot10^{-17 } $ && $ 5.4710 \cdot10^{-13 } $ && $ 4.2127 \cdot10^{-9 } $ && $ 3.2438 \cdot10^{-5 } $ & \\
% & $ 7800 $ && $ 5.9471 \cdot10^{-21 } $ && $ 4.6388 \cdot10^{-17 } $ && $ 3.6183 \cdot10^{-13 } $ && $ 2.8222 \cdot10^{-9 } $ && $ 2.2014 \cdot10^{-5 } $ & \\
% & $ 7900 $ && $ 3.8429 \cdot10^{-21 } $ && $ 3.0359 \cdot10^{-17 } $ && $ 2.3984 \cdot10^{-13 } $ && $ 1.8947 \cdot10^{-9 } $ && $ 1.4968 \cdot10^{-5 } $ & \\
& $ 8000 $ && $ 2.4895 \cdot10^{-21 } $ && $ 1.9916 \cdot10^{-17 } $ && $ 1.5933 \cdot10^{-13 } $ && $ 1.2747 \cdot10^{-9 } $ && $ 1.0197 \cdot10^{-5 } $ & \\
% & $ 8100 $ && $ 1.6168 \cdot10^{-21 } $ && $ 1.3096 \cdot10^{-17 } $ && $ 1.0608 \cdot10^{-13 } $ && $ 8.5921 \cdot10^{-10 } $ && $ 6.9596 \cdot10^{-6 } $ & \\
% & $ 8200 $ && $ 1.0526 \cdot10^{-21 } $ && $ 8.6311 \cdot10^{-18 } $ && $ 7.0775 \cdot10^{-14 } $ && $ 5.8036 \cdot10^{-10 } $ && $ 4.7589 \cdot10^{-6 } $ & \\
% & $ 8300 $ && $ 6.8692 \cdot10^{-22 } $ && $ 5.7015 \cdot10^{-18 } $ && $ 4.7322 \cdot10^{-14 } $ && $ 3.9278 \cdot10^{-10 } $ && $ 3.2600 \cdot10^{-6 } $ & \\
% & $ 8400 $ && $ 4.4936 \cdot10^{-22 } $ && $ 3.7746 \cdot10^{-18 } $ && $ 3.1707 \cdot10^{-14 } $ && $ 2.6634 \cdot10^{-10 } $ && $ 2.2373 \cdot10^{-6 } $ & \\
% & $ 8500 $ && $ 2.9465 \cdot10^{-22 } $ && $ 2.5045 \cdot10^{-18 } $ && $ 2.1288 \cdot10^{-14 } $ && $ 1.8095 \cdot10^{-10 } $ && $ 1.5381 \cdot10^{-6 } $ & \\
% & $ 8600 $ && $ 1.9364 \cdot10^{-22 } $ && $ 1.6653 \cdot10^{-18 } $ && $ 1.4322 \cdot10^{-14 } $ && $ 1.2317 \cdot10^{-10 } $ && $ 1.0593 \cdot10^{-6 } $ & \\
% & $ 8700 $ && $ 1.2755 \cdot10^{-22 } $ && $ 1.1097 \cdot10^{-18 } $ && $ 9.6540 \cdot10^{-15 } $ && $ 8.3990 \cdot10^{-11 } $ && $ 7.3071 \cdot10^{-7 } $ & \\
% & $ 8800 $ && $ 8.4200 \cdot10^{-23 } $ && $ 7.4096 \cdot10^{-19 } $ && $ 6.5205 \cdot10^{-15 } $ && $ 5.7380 \cdot10^{-11 } $ && $ 5.0495 \cdot10^{-7 } $ & \\
% & $ 8900 $ && $ 5.5707 \cdot10^{-23 } $ && $ 4.9579 \cdot10^{-19 } $ && $ 4.4125 \cdot10^{-15 } $ && $ 3.9272 \cdot10^{-11 } $ && $ 3.4952 \cdot10^{-7 } $ & \\
& $ 9000 $ && $ 3.6935 \cdot10^{-23 } $ && $ 3.3241 \cdot10^{-19 } $ && $ 2.9917 \cdot10^{-15 } $ && $ 2.6926 \cdot10^{-11 } $ && $ 2.4233 \cdot10^{-7 } $ & \\
% & $ 9100 $ && $ 2.4541 \cdot10^{-23 } $ && $ 2.2332 \cdot10^{-19 } $ && $ 2.0322 \cdot10^{-15 } $ && $ 1.8493 \cdot10^{-11 } $ && $ 1.6829 \cdot10^{-7 } $ & \\
% & $ 9200 $ && $ 1.6340 \cdot10^{-23 } $ && $ 1.5032 \cdot10^{-19 } $ && $ 1.3830 \cdot10^{-15 } $ && $ 1.2724 \cdot10^{-11 } $ && $ 1.1706 \cdot10^{-7 } $ & \\
% & $ 9300 $ && $ 1.0902 \cdot10^{-23 } $ && $ 1.0139 \cdot10^{-19 } $ && $ 9.4285 \cdot10^{-16 } $ && $ 8.7685 \cdot10^{-12 } $ && $ 8.1547 \cdot10^{-8 } $ & \\
% & $ 9400 $ && $ 7.2878 \cdot10^{-24 } $ && $ 6.8506 \cdot10^{-20 } $ && $ 6.4395 \cdot10^{-16 } $ && $ 6.0532 \cdot10^{-12 } $ && $ 5.6900 \cdot10^{-8 } $ & \\
% & $ 9500 $ && $ 4.8819 \cdot10^{-24 } $ && $ 4.6378 \cdot10^{-20 } $ && $ 4.4059 \cdot10^{-16 } $ && $ 4.1856 \cdot10^{-12 } $ && $ 3.9763 \cdot10^{-8 } $ & \\
% & $ 9600 $ && $ 3.2766 \cdot10^{-24 } $ && $ 3.1456 \cdot10^{-20 } $ && $ 3.0197 \cdot10^{-16 } $ && $ 2.8990 \cdot10^{-12 } $ && $ 2.7830 \cdot10^{-8 } $ & \\
% & $ 9700 $ && $ 2.2035 \cdot10^{-24 } $ && $ 2.1374 \cdot10^{-20 } $ && $ 2.0733 \cdot10^{-16 } $ && $ 2.0111 \cdot10^{-12 } $ && $ 1.9507 \cdot10^{-8 } $ & \\
% & $ 9800 $ && $ 1.4846 \cdot10^{-24 } $ && $ 1.4549 \cdot10^{-20 } $ && $ 1.4258 \cdot10^{-16 } $ && $ 1.3973 \cdot10^{-12 } $ && $ 1.3694 \cdot10^{-8 } $ & \\
% & $ 9900 $ && $ 1.0022 \cdot10^{-24 } $ && $ 9.9212 \cdot10^{-21 } $ && $ 9.8219 \cdot10^{-17 } $ && $ 9.7237 \cdot10^{-13 } $ && $ 9.6265 \cdot10^{-9 } $ & \\
& $ 10000 $ && $ 6.7772 \cdot10^{-25 } $ && $ 6.7772 \cdot10^{-21 } $ && $ 6.7772 \cdot10^{-17 } $ && $ 6.7772 \cdot10^{-13 } $ && $ 6.7772 \cdot10^{-9 } $ & \\
% & $ 10100 $ && $ 4.5917 \cdot10^{-25 } $ && $ 4.6376 \cdot10^{-21 } $ && $ 4.6840 \cdot10^{-17 } $ && $ 4.7308 \cdot10^{-13 } $ && $ 4.7781 \cdot10^{-9 } $ & \\
% & $ 10200 $ && $ 3.1165 \cdot10^{-25 } $ && $ 3.1789 \cdot10^{-21 } $ && $ 3.2424 \cdot10^{-17 } $ && $ 3.3073 \cdot10^{-13 } $ && $ 3.3734 \cdot10^{-9 } $ & \\
% & $ 10300 $ && $ 2.1191 \cdot10^{-25 } $ && $ 2.1826 \cdot10^{-21 } $ && $ 2.2481 \cdot10^{-17 } $ && $ 2.3155 \cdot10^{-13 } $ && $ 2.3850 \cdot10^{-9 } $ & \\
% & $ 10400 $ && $ 1.4434 \cdot10^{-25 } $ && $ 1.5011 \cdot10^{-21 } $ && $ 1.5611 \cdot10^{-17 } $ && $ 1.6236 \cdot10^{-13 } $ && $ 1.6885 \cdot10^{-9 } $ & \\
% & $ 10500 $ && $ 9.8478 \cdot10^{-26 } $ && $ 1.0341 \cdot10^{-21 } $ && $ 1.0858 \cdot10^{-17 } $ && $ 1.1401 \cdot10^{-13 } $ && $ 1.1971 \cdot10^{-9 } $ & \\
% & $ 10600 $ && $ 6.7307 \cdot10^{-26 } $ && $ 7.1345 \cdot10^{-22 } $ && $ 7.5626 \cdot10^{-18 } $ && $ 8.0163 \cdot10^{-14 } $ && $ 8.4973 \cdot10^{-10 } $ & \\
% & $ 10700 $ && $ 4.6079 \cdot10^{-26 } $ && $ 4.9305 \cdot10^{-22 } $ && $ 5.2756 \cdot10^{-18 } $ && $ 5.6449 \cdot10^{-14 } $ && $ 6.0400 \cdot10^{-10 } $ & \\
% & $ 10800 $ && $ 3.1599 \cdot10^{-26 } $ && $ 3.4127 \cdot10^{-22 } $ && $ 3.6857 \cdot10^{-18 } $ && $ 3.9805 \cdot10^{-14 } $ && $ 4.2990 \cdot10^{-10 } $ & \\
% & $ 10900 $ && $ 2.1705 \cdot10^{-26 } $ && $ 2.3658 \cdot10^{-22 } $ && $ 2.5787 \cdot10^{-18 } $ && $ 2.8108 \cdot10^{-14 } $ && $ 3.0638 \cdot10^{-10 } $ & \\
& $ 11000 $ && $ 1.4933 \cdot10^{-26 } $ && $ 1.6426 \cdot10^{-22 } $ && $ 1.8068 \cdot10^{-18 } $ && $ 1.9875 \cdot10^{-14 } $ && $ 2.1862 \cdot10^{-10 } $ & \\
% & $ 11100 $ && $ 1.0290 \cdot10^{-26 } $ && $ 1.1422 \cdot10^{-22 } $ && $ 1.2678 \cdot10^{-18 } $ && $ 1.4073 \cdot10^{-14 } $ && $ 1.5620 \cdot10^{-10 } $ & \\
% & $ 11200 $ && $ 7.1014 \cdot10^{-27 } $ && $ 7.9535 \cdot10^{-23 } $ && $ 8.9079 \cdot10^{-19 } $ && $ 9.9769 \cdot10^{-15 } $ && $ 1.1175 \cdot10^{-10 } $ & \\
% & $ 11300 $ && $ 4.9087 \cdot10^{-27 } $ && $ 5.5469 \cdot10^{-23 } $ && $ 6.2680 \cdot10^{-19 } $ && $ 7.0828 \cdot10^{-15 } $ && $ 8.0035 \cdot10^{-11 } $ & \\
% & $ 11400 $ && $ 3.3983 \cdot10^{-27 } $ && $ 3.8741 \cdot10^{-23 } $ && $ 4.4164 \cdot10^{-19 } $ && $ 5.0347 \cdot10^{-15 } $ && $ 5.7396 \cdot10^{-11 } $ & \\
% & $ 11500 $ && $ 2.3562 \cdot10^{-27 } $ && $ 2.7097 \cdot10^{-23 } $ && $ 3.1161 \cdot10^{-19 } $ && $ 3.5835 \cdot10^{-15 } $ && $ 4.1210 \cdot10^{-11 } $ & \\
% & $ 11600 $ && $ 1.6361 \cdot10^{-27 } $ && $ 1.8979 \cdot10^{-23 } $ && $ 2.2016 \cdot10^{-19 } $ && $ 2.5538 \cdot10^{-15 } $ && $ 2.9624 \cdot10^{-11 } $ & \\
% & $ 11700 $ && $ 1.1378 \cdot10^{-27 } $ && $ 1.3312 \cdot10^{-23 } $ && $ 1.5575 \cdot10^{-19 } $ && $ 1.8223 \cdot10^{-15 } $ && $ 2.1321 \cdot10^{-11 } $ & \\
% & $ 11800 $ && $ 7.9236 \cdot10^{-28 } $ && $ 9.3499 \cdot10^{-24 } $ && $ 1.1033 \cdot10^{-19 } $ && $ 1.3019 \cdot10^{-15 } $ && $ 1.5363 \cdot10^{-11 } $ & \\
% & $ 11900 $ && $ 5.5262 \cdot10^{-28 } $ && $ 6.5762 \cdot10^{-24 } $ && $ 7.8256 \cdot10^{-20 } $ && $ 9.3125 \cdot10^{-16 } $ && $ 1.1082 \cdot10^{-11 } $ & \\
& $ 12000 $ && $ 3.8597 \cdot10^{-28 } $ && $ 4.6316 \cdot10^{-24 } $ && $ 5.5579 \cdot10^{-20 } $ && $ 6.6694 \cdot10^{-16 } $ && $ 8.0033 \cdot10^{-12 } $ & \\
% & $ 12100 $ && $ 2.6995 \cdot10^{-28 } $ && $ 3.2664 \cdot10^{-24 } $ && $ 3.9523 \cdot10^{-20 } $ && $ 4.7823 \cdot10^{-16 } $ && $ 5.7865 \cdot10^{-12 } $ & \\
% & $ 12200 $ && $ 1.8907 \cdot10^{-28 } $ && $ 2.3066 \cdot10^{-24 } $ && $ 2.8141 \cdot10^{-20 } $ && $ 3.4332 \cdot10^{-16 } $ && $ 4.1884 \cdot10^{-12 } $ & \\
% & $ 12300 $ && $ 1.3260 \cdot10^{-28 } $ && $ 1.6310 \cdot10^{-24 } $ && $ 2.0061 \cdot10^{-20 } $ && $ 2.4675 \cdot10^{-16 } $ && $ 3.0351 \cdot10^{-12 } $ & \\
% & $ 12400 $ && $ 9.3124 \cdot10^{-29 } $ && $ 1.1548 \cdot10^{-24 } $ && $ 1.4319 \cdot10^{-20 } $ && $ 1.7756 \cdot10^{-16 } $ && $ 2.2017 \cdot10^{-12 } $ & \\
% & $ 12500 $ && $ 6.5489 \cdot10^{-29 } $ && $ 8.1861 \cdot10^{-25 } $ && $ 1.0233 \cdot10^{-20 } $ && $ 1.2791 \cdot10^{-16 } $ && $ 1.5989 \cdot10^{-12 } $ & \\
% & $ 12600 $ && $ 4.6116 \cdot10^{-29 } $ && $ 5.8106 \cdot10^{-25 } $ && $ 7.3214 \cdot10^{-21 } $ && $ 9.2249 \cdot10^{-17 } $ && $ 1.1624 \cdot10^{-12 } $ & \\
% & $ 12700 $ && $ 3.2517 \cdot10^{-29 } $ && $ 4.1296 \cdot10^{-25 } $ && $ 5.2446 \cdot10^{-21 } $ && $ 6.6606 \cdot10^{-17 } $ && $ 8.4590 \cdot10^{-13 } $ & \\
% & $ 12800 $ && $ 2.2958 \cdot10^{-29 } $ && $ 2.9385 \cdot10^{-25 } $ && $ 3.7613 \cdot10^{-21 } $ && $ 4.8145 \cdot10^{-17 } $ && $ 6.1625 \cdot10^{-13 } $ & \\
% & $ 12900 $ && $ 1.6229 \cdot10^{-29 } $ && $ 2.0936 \cdot10^{-25 } $ && $ 2.7007 \cdot10^{-21 } $ && $ 3.4839 \cdot10^{-17 } $ && $ 4.4942 \cdot10^{-13 } $ & \\
& $ 13000 $ && $ 1.1488 \cdot10^{-29 } $ && $ 1.4934 \cdot10^{-25 } $ && $ 1.9414 \cdot10^{-21 } $ && $ 2.5238 \cdot10^{-17 } $ && $ 3.2809 \cdot10^{-13 } $ & \\
% & $ 13500 $ && $ 2.0792 \cdot10^{-30 } $ && $ 2.8070 \cdot10^{-26 } $ && $ 3.7894 \cdot10^{-22 } $ && $ 5.1157 \cdot10^{-18 } $ && $ 6.9061 \cdot10^{-14 } $ & \\
% & $ 13800 $ && $ 7.5451 \cdot10^{-31 } $ && $ 1.0413 \cdot10^{-26 } $ && $ 1.4371 \cdot10^{-22 } $ && $ 1.9833 \cdot10^{-18 } $ && $ 2.7370 \cdot10^{-14 } $ & \\
& $ 14000 $ && $ 3.8771 \cdot10^{-31 } $ && $ 5.4279 \cdot10^{-27 } $ && $ 7.5990 \cdot10^{-23 } $ && $ 1.0639 \cdot10^{-18 } $ && $ 1.4894 \cdot10^{-14 } $ & \\
& $ 15000 $ && $ 1.4652 \cdot10^{-32 } $ && $ 2.1978 \cdot10^{-28 } $ && $ 3.2967 \cdot10^{-24 } $ && $ 4.9450 \cdot10^{-20 } $ && $ 7.4175 \cdot10^{-16 } $ & \\
& $ 16000 $ && $ 6.1341 \cdot10^{-34 } $ && $ 9.8146 \cdot10^{-30 } $ && $ 1.5704 \cdot10^{-25 } $ && $ 2.5126 \cdot10^{-21 } $ && $ 4.0201 \cdot10^{-17 } $ & \\
& $ 17000 $ && $ 2.8193 \cdot10^{-35 } $ && $ 4.7928 \cdot10^{-31 } $ && $ 8.1478 \cdot10^{-27 } $ && $ 1.3852 \cdot10^{-22 } $ && $ 2.3547 \cdot10^{-18 } $ & \\
& $ 18000 $ && $ 1.4115 \cdot10^{-36 } $ && $ 2.5406 \cdot10^{-32 } $ && $ 4.5731 \cdot10^{-28 } $ && $ 8.2315 \cdot10^{-24 } $ && $ 1.4817 \cdot10^{-19 } $ & \\
& $ 19000 $ && $ 7.6449 \cdot10^{-38 } $ && $ 1.4526 \cdot10^{-33 } $ && $ 2.7599 \cdot10^{-29 } $ && $ 5.2437 \cdot10^{-25 } $ && $ 9.9630 \cdot10^{-21 } $ & \\
& $ 20000 $ && $ 4.4536 \cdot10^{-39 } $ && $ 8.9071 \cdot10^{-35 } $ && $ 1.7815 \cdot10^{-30 } $ && $ 3.5629 \cdot10^{-26 } $ && $ 7.1257 \cdot10^{-22 } $ & \\
& $ 21000 $ && $ 2.7759 \cdot10^{-40 } $ && $ 5.8293 \cdot10^{-36 } $ && $ 1.2242 \cdot10^{-31 } $ && $ 2.5708 \cdot10^{-27 } $ && $ 5.3985 \cdot10^{-23 } $ & \\
& $ 22000 $ && $ 1.8427 \cdot10^{-41 } $ && $ 4.0538 \cdot10^{-37 } $ && $ 8.9184 \cdot10^{-33 } $ && $ 1.9621 \cdot10^{-28 } $ && $ 4.3165 \cdot10^{-24 } $ & \\
& $ 23000 $ && $ 1.2974 \cdot10^{-42 } $ && $ 2.9839 \cdot10^{-38 } $ && $ 6.8629 \cdot10^{-34 } $ && $ 1.5785 \cdot10^{-29 } $ && $ 3.6305 \cdot10^{-25 } $ & \\
& $ 24000 $ && $ 9.6521 \cdot10^{-44 } $ && $ 2.3165 \cdot10^{-39 } $ && $ 5.5596 \cdot10^{-35 } $ && $ 1.3344 \cdot10^{-30 } $ && $ 3.2024 \cdot10^{-26 } $ & \\
& $ 25000 $ && $ 7.5635 \cdot10^{-45 } $ && $ 1.8909 \cdot10^{-40 } $ && $ 4.7272 \cdot10^{-36 } $ && $ 1.1818 \cdot10^{-31 } $ && $ 3.0000 \cdot10^{-27 } $ & \\
\hline
\hline
\end{longtable}
%Calculations shown in \path{Dropbox/Undergraduate-Research-Sharper-Psi-Theta-Summer2017/sharper-bounds-psi-theta-Broadbent-Wilk-Lumley-Kad-Ng/Programs_for_Sharper_bounds/June 2019/A-June-2019.pari}.
}
%\pagebreak

 \subsection{Sharper bounds for $\theta(x)$: $x$ in a middle range ($e^{20} \le x \le e^{25\,000}$)}
 We use Corollary \ref{Cor:Bk}:
% $|\theta(x)-x|<\frac{ B_{j,k}x}{(\log x)^k}$ for all $x \in [e^{b_j},e^{b_{j+1}})$, where $B_{j,k}$ is defined in \eqref{Bbbprime2}. 
   {\footnotesize
% The following table uses values of $\varepsilon(b_0)$ to bound $\left| \frac{\psi(x)-x}{x} \right|$ from \cite[Table 2]{But} and 
% \cite[Table 4]{KLN}, both of which use Wedeniwski's $H_0 = 2.445 \cdot 10^{12}$.
\begin{longtable}{rrrrrrrrrrrrrr} 
% \caption{Selected values for $ B_{j,k}$ in $|\theta(x)-x|<\frac{ B_{j,k}x}{(\log x)^k}$ calculated using the method described in \cref{Cor:Bk}, 
% Section \ref{section:xmiddle}.
%\caption{$|\theta(x)-x|<\frac{ B_k(b,b')x}{(\log x)^k}$ for all $x \in [e^b,e^{b'}$, where $B_k(b,b')$ is defined in \eqref{Bbbprime2}.}
\caption{Values for $ B_{k}(b,b')$ in $|\theta(x)-x|<\frac{  B_{k}(b,b')x}{(\log x)^k}$ calculated using the method described in Corollary \ref{Cor:Bk}.
Each $ B_{k}(b,b')$ is valid for $e^{b}\le x \le e^{b'}$. The last line is valid for $e^{24000} \le x \le e^{25000}$.}

\label{BkiTableWedeniwski}
\\ 
\hline
&&&&&&&&&&&&& \\[-1em]
\multicolumn{1}{c}{\phantom{}} &
\multicolumn{1}{r}{$b$  } &
\multicolumn{1}{c}{\phantom{a}} &
\multicolumn{1}{c}{$ B_1(b,b')$ } &
\multicolumn{1}{c}{\phantom{a}} &
\multicolumn{1}{c}{$ B_2(b,b')$ } &
\multicolumn{1}{c}{\phantom{a}} &
\multicolumn{1}{c}{$ B_3(b,b')$ } &
\multicolumn{1}{c}{\phantom{a}} &
\multicolumn{1}{c}{$ B_4(b,b')$ } &
\multicolumn{1}{c}{\phantom{a}} &
\multicolumn{1}{c}{$ B_5(b,b')$ } &
\multicolumn{1}{c}{\phantom{}} 
\\ \hline 
\endfirsthead
\multicolumn{13}{c}%
{\tablename\ \thetable{} -- continued from previous page} \\
\hline
&&&&&&&&&&&&& \\[-1em]
\multicolumn{1}{c}{\phantom{}} &
\multicolumn{1}{r}{$b$  } &
\multicolumn{1}{c}{\phantom{a}} &
\multicolumn{1}{c}{$ B_1(b,b')$ } &
\multicolumn{1}{c}{\phantom{a}} &
\multicolumn{1}{c}{$ B_2(b,b')$ } &
\multicolumn{1}{c}{\phantom{a}} &
\multicolumn{1}{c}{$ B_3(b,b')$ } &
\multicolumn{1}{c}{\phantom{a}} &
\multicolumn{1}{c}{$ B_4(b,b')$ } &
\multicolumn{1}{c}{\phantom{a}} &
\multicolumn{1}{c}{$ B_5(b,b')$ } &
\multicolumn{1}{c}{\phantom{}} 
\\ \hline 
\endhead
\hline \multicolumn{13}{r}{Continued on next page} \\ \hline
\endfoot
\hline 
\endlastfoot
&&&&&&&&&&&&& \\[-1em]
& $20 $ && $1.8077 \cdot 10^{-3} $ && $3.6154 \cdot 10^{-2} $ && $7.2309 \cdot 10^{-1} $ && $1.4462 \cdot 10^{1} $ && $2.9160 \cdot 10^{2} $ & \\
& $21 $ && $1.1458 \cdot 10^{-3} $ && $2.4062 \cdot 10^{-2} $ && $5.0530 \cdot 10^{-1} $ && $1.0611 \cdot 10^{1} $ && $2.2284 \cdot 10^{2} $ & \\
& $22 $ && $7.2527 \cdot 10^{-4} $ && $1.5956 \cdot 10^{-2} $ && $3.5103 \cdot 10^{-1} $ && $7.7226 \cdot 10^{0} $ && $1.6990 \cdot 10^{2} $ & \\
& $23 $ && $4.5848 \cdot 10^{-4} $ && $1.0545 \cdot 10^{-2} $ && $2.4254 \cdot 10^{-1} $ && $5.5783 \cdot 10^{0} $ && $1.2830 \cdot 10^{2} $ & \\
& $24 $ && $2.8945 \cdot 10^{-4} $ && $6.9468 \cdot 10^{-3} $ && $1.6672 \cdot 10^{-1} $ && $4.0013 \cdot 10^{0} $ && $9.6032 \cdot 10^{1} $ & \\
& $25 $ && $1.8251 \cdot 10^{-4} $ && $4.5626 \cdot 10^{-3} $ && $1.1407 \cdot 10^{-1} $ && $2.8516 \cdot 10^{0} $ && $7.1291 \cdot 10^{1} $ & \\
& $26 $ && $1.1493 \cdot 10^{-4} $ && $2.9882 \cdot 10^{-3} $ && $7.7694 \cdot 10^{-2} $ && $2.0200 \cdot 10^{0} $ && $5.2521 \cdot 10^{1} $ & \\
& $27 $ && $7.2293 \cdot 10^{-5} $ && $1.9519 \cdot 10^{-3} $ && $5.2702 \cdot 10^{-2} $ && $1.4229 \cdot 10^{0} $ && $3.8419 \cdot 10^{1} $ & \\
& $28 $ && $4.5421 \cdot 10^{-5} $ && $1.2718 \cdot 10^{-3} $ && $3.5610 \cdot 10^{-2} $ && $9.9708 \cdot 10^{-1} $ && $2.7918 \cdot 10^{1} $ & \\
& $29 $ && $2.8507 \cdot 10^{-5} $ && $8.2670 \cdot 10^{-4} $ && $2.3974 \cdot 10^{-2} $ && $6.9525 \cdot 10^{-1} $ && $2.0162 \cdot 10^{1} $ & \\
& $30 $ && $1.7873 \cdot 10^{-5} $ && $5.3619 \cdot 10^{-4} $ && $1.6086 \cdot 10^{-2} $ && $4.8257 \cdot 10^{-1} $ && $1.4477 \cdot 10^{1} $ & \\
& \vdots && && && && && & \\ 
& $43 $ && $8.5986 \cdot 10^{-7} $ && $3.7618 \cdot 10^{-5} $ && $1.6458 \cdot 10^{-3} $ && $7.2000 \cdot 10^{-2} $ && $3.1500 \cdot 10^{0} $ & \\
& $19 \log 10 $ && $8.6315 \cdot 10^{-7} $ && $3.7978 \cdot 10^{-5} $ && $1.6711 \cdot 10^{-3} $ && $7.3526 \cdot 10^{-2} $ && $3.2352 \cdot 10^{0} $ & \\
& $44 $ && $7.8162 \cdot 10^{-7} $ && $3.5173 \cdot 10^{-5} $ && $1.5828 \cdot 10^{-3} $ && $7.1225 \cdot 10^{-2} $ && $3.2052 \cdot 10^{0} $ & \\
& $45 $ && $5.0646 \cdot 10^{-7} $ && $2.3297 \cdot 10^{-5} $ && $1.0717 \cdot 10^{-3} $ && $4.9297 \cdot 10^{-2} $ && $2.2677 \cdot 10^{0} $ & \\
& $46 $ && $3.2935 \cdot 10^{-7} $ && $1.5479 \cdot 10^{-5} $ && $7.2752 \cdot 10^{-4} $ && $3.4194 \cdot 10^{-2} $ && $1.6071 \cdot 10^{0} $ & \\
& $47 $ && $2.1307 \cdot 10^{-7} $ && $1.0228 \cdot 10^{-5} $ && $4.9092 \cdot 10^{-4} $ && $2.3564 \cdot 10^{-2} $ && $1.1311 \cdot 10^{0} $ & \\
& \vdots && && && && && & \\ 
& $54 $ && $9.8777 \cdot 10^{-9} $ && $5.4328 \cdot 10^{-7} $ && $2.9880 \cdot 10^{-5} $ && $1.6434 \cdot 10^{-3} $ && $9.0388 \cdot 10^{-2} $ & \\
& $55 $ && $6.3417 \cdot 10^{-9} $ && $3.5514 \cdot 10^{-7} $ && $1.9888 \cdot 10^{-5} $ && $1.1137 \cdot 10^{-3} $ && $6.2367 \cdot 10^{-2} $ & \\
& $56 $ && $4.0668 \cdot 10^{-9} $ && $2.3181 \cdot 10^{-7} $ && $1.3213 \cdot 10^{-5} $ && $7.5315 \cdot 10^{-4} $ && $4.2929 \cdot 10^{-2} $ & \\
& \vdots && && && && && & \\ 
& $2275 $ && $4.4153 \cdot 10^{-9} $ && $1.0155 \cdot 10^{-5} $ && $2.3357 \cdot 10^{-2} $ && $5.3721 \cdot 10^{1} $ && $1.2356 \cdot 10^{5} $ & \\
& $2300 $ && $4.4627 \cdot 10^{-9} $ && $1.0376 \cdot 10^{-5} $ && $2.4124 \cdot 10^{-2} $ && $5.6088 \cdot 10^{1} $ && $1.3040 \cdot 10^{5} $ & \\
& $2325 $ && $4.4062 \cdot 10^{-9} $ && $1.0355 \cdot 10^{-5} $ && $2.4333 \cdot 10^{-2} $ && $5.7184 \cdot 10^{1} $ && $1.3438 \cdot 10^{5} $ & \\
& $2350 $ && $4.2245 \cdot 10^{-9} $ && $1.0033 \cdot 10^{-5} $ && $2.3829 \cdot 10^{-2} $ && $5.6593 \cdot 10^{1} $ && $1.3441 \cdot 10^{5} $ & \\
& $2375 $ && $4.0498 \cdot 10^{-9} $ && $9.7196 \cdot 10^{-6} $ && $2.3327 \cdot 10^{-2} $ && $5.5985 \cdot 10^{1} $ && $1.3436 \cdot 10^{5} $ & \\
& $2400 $ && $3.8820 \cdot 10^{-9} $ && $9.4139 \cdot 10^{-6} $ && $2.2829 \cdot 10^{-2} $ && $5.5360 \cdot 10^{1} $ && $1.3425 \cdot 10^{5} $ & \\
& \vdots && && && && && & \\ 
& $9800 $ && $8.4841 \cdot 10^{-25} $ && $8.3992 \cdot 10^{-21} $ && $8.3152 \cdot 10^{-17} $ && $8.2321 \cdot 10^{-13} $ && $8.1497 \cdot 10^{-9} $ & \\
& $9900 $ && $5.7395 \cdot 10^{-25} $ && $5.7395 \cdot 10^{-21} $ && $5.7395 \cdot 10^{-17} $ && $5.7395 \cdot 10^{-13} $ && $5.7395 \cdot 10^{-9} $ & \\
& $10000 $ && $3.8228 \cdot 10^{-25} $ && $3.8610 \cdot 10^{-21} $ && $3.8996 \cdot 10^{-17} $ && $3.9386 \cdot 10^{-13} $ && $3.9780 \cdot 10^{-9} $ & \\
% & $10100 $ && $2.5745 \cdot 10^{-25} $ && $2.6260 \cdot 10^{-21} $ && $2.6785 \cdot 10^{-17} $ && $2.7321 \cdot 10^{-13} $ && $2.7867 \cdot 10^{-9} $ & \\
% & $10200 $ && $1.7389 \cdot 10^{-25} $ && $1.7911 \cdot 10^{-21} $ && $1.8448 \cdot 10^{-17} $ && $1.9001 \cdot 10^{-13} $ && $1.9571 \cdot 10^{-9} $ & \\
% & $10300 $ && $1.1734 \cdot 10^{-25} $ && $1.2203 \cdot 10^{-21} $ && $1.2691 \cdot 10^{-17} $ && $1.3199 \cdot 10^{-13} $ && $1.3727 \cdot 10^{-9} $ & \\
% & $10400 $ && $7.9556 \cdot 10^{-26} $ && $8.3534 \cdot 10^{-22} $ && $8.7710 \cdot 10^{-18} $ && $9.2096 \cdot 10^{-14} $ && $9.6701 \cdot 10^{-10} $ & \\
% & $10500 $ && $5.4076 \cdot 10^{-26} $ && $5.7321 \cdot 10^{-22} $ && $6.0760 \cdot 10^{-18} $ && $6.4406 \cdot 10^{-14} $ && $6.8270 \cdot 10^{-10} $ & \\
% & $10600 $ && $3.6845 \cdot 10^{-26} $ && $3.9424 \cdot 10^{-22} $ && $4.2184 \cdot 10^{-18} $ && $4.5136 \cdot 10^{-14} $ && $4.8296 \cdot 10^{-10} $ & \\
% & $10700 $ && $2.5150 \cdot 10^{-26} $ && $2.7162 \cdot 10^{-22} $ && $2.9335 \cdot 10^{-18} $ && $3.1682 \cdot 10^{-14} $ && $3.4216 \cdot 10^{-10} $ & \\
% & $10800 $ && $1.7201 \cdot 10^{-26} $ && $1.8749 \cdot 10^{-22} $ && $2.0436 \cdot 10^{-18} $ && $2.2276 \cdot 10^{-14} $ && $2.4280 \cdot 10^{-10} $ & \\
% & $10900 $ && $1.1639 \cdot 10^{-26} $ && $1.2803 \cdot 10^{-22} $ && $1.4083 \cdot 10^{-18} $ && $1.5492 \cdot 10^{-14} $ && $1.7041 \cdot 10^{-10} $ & \\
% & $11000 $ && $7.9283 \cdot 10^{-27} $ && $8.8005 \cdot 10^{-23} $ && $9.7685 \cdot 10^{-19} $ && $1.0843 \cdot 10^{-14} $ && $1.2036 \cdot 10^{-10} $ & \\
& $11100 $ && $5.4156 \cdot 10^{-27} $ && $6.0654 \cdot 10^{-23} $ && $6.7933 \cdot 10^{-19} $ && $7.6085 \cdot 10^{-15} $ && $8.5215 \cdot 10^{-11} $ & \\
% & $11200 $ && $3.7120 \cdot 10^{-27} $ && $4.1945 \cdot 10^{-23} $ && $4.7398 \cdot 10^{-19} $ && $5.3560 \cdot 10^{-15} $ && $6.0522 \cdot 10^{-11} $ & \\
% & $11300 $ && $2.5509 \cdot 10^{-27} $ && $2.9080 \cdot 10^{-23} $ && $3.3151 \cdot 10^{-19} $ && $3.7792 \cdot 10^{-15} $ && $4.3083 \cdot 10^{-11} $ & \\
% & $11400 $ && $1.7569 \cdot 10^{-27} $ && $2.0205 \cdot 10^{-23} $ && $2.3235 \cdot 10^{-19} $ && $2.6721 \cdot 10^{-15} $ && $3.0729 \cdot 10^{-11} $ & \\
% & $11500 $ && $1.2102 \cdot 10^{-27} $ && $1.4039 \cdot 10^{-23} $ && $1.6285 \cdot 10^{-19} $ && $1.8890 \cdot 10^{-15} $ && $2.1913 \cdot 10^{-11} $ & \\
% & $11600 $ && $8.3444 \cdot 10^{-28} $ && $9.7630 \cdot 10^{-24} $ && $1.1423 \cdot 10^{-19} $ && $1.3365 \cdot 10^{-15} $ && $1.5637 \cdot 10^{-11} $ & \\
% & $11700 $ && $5.7692 \cdot 10^{-28} $ && $6.8076 \cdot 10^{-24} $ && $8.0330 \cdot 10^{-20} $ && $9.4789 \cdot 10^{-16} $ && $1.1185 \cdot 10^{-11} $ & \\
% & $11800 $ && $3.9987 \cdot 10^{-28} $ && $4.7584 \cdot 10^{-24} $ && $5.6625 \cdot 10^{-20} $ && $6.7384 \cdot 10^{-16} $ && $8.0187 \cdot 10^{-12} $ & \\
% & $11900 $ && $2.7776 \cdot 10^{-28} $ && $3.3331 \cdot 10^{-24} $ && $3.9997 \cdot 10^{-20} $ && $4.7996 \cdot 10^{-16} $ && $5.7595 \cdot 10^{-12} $ & \\
& $12000 $ && $1.9330 \cdot 10^{-28} $ && $2.3390 \cdot 10^{-24} $ && $2.8302 \cdot 10^{-20} $ && $3.4245 \cdot 10^{-16} $ && $4.1436 \cdot 10^{-12} $ & \\
% & $12100 $ && $1.3477 \cdot 10^{-28} $ && $1.6442 \cdot 10^{-24} $ && $2.0060 \cdot 10^{-20} $ && $2.4473 \cdot 10^{-16} $ && $2.9857 \cdot 10^{-12} $ & \\
% & $12200 $ && $9.3146 \cdot 10^{-29} $ && $1.1457 \cdot 10^{-24} $ && $1.4092 \cdot 10^{-20} $ && $1.7333 \cdot 10^{-16} $ && $2.1320 \cdot 10^{-12} $ & \\
% & $12300 $ && $6.5069 \cdot 10^{-29} $ && $8.0685 \cdot 10^{-25} $ && $1.0005 \cdot 10^{-20} $ && $1.2406 \cdot 10^{-16} $ && $1.5384 \cdot 10^{-12} $ & \\
% & $12400 $ && $4.5539 \cdot 10^{-29} $ && $5.6924 \cdot 10^{-25} $ && $7.1155 \cdot 10^{-21} $ && $8.8944 \cdot 10^{-17} $ && $1.1118 \cdot 10^{-12} $ & \\
% & $12500 $ && $3.1924 \cdot 10^{-29} $ && $4.0224 \cdot 10^{-25} $ && $5.0682 \cdot 10^{-21} $ && $6.3859 \cdot 10^{-17} $ && $8.0462 \cdot 10^{-13} $ & \\
% & $12600 $ && $2.2307 \cdot 10^{-29} $ && $2.8330 \cdot 10^{-25} $ && $3.5979 \cdot 10^{-21} $ && $4.5693 \cdot 10^{-17} $ && $5.8030 \cdot 10^{-13} $ & \\
% & $12700 $ && $1.5599 \cdot 10^{-29} $ && $1.9967 \cdot 10^{-25} $ && $2.5558 \cdot 10^{-21} $ && $3.2714 \cdot 10^{-17} $ && $4.1873 \cdot 10^{-13} $ & \\
% & $12800 $ && $1.0940 \cdot 10^{-29} $ && $1.4112 \cdot 10^{-25} $ && $1.8205 \cdot 10^{-21} $ && $2.3484 \cdot 10^{-17} $ && $3.0294 \cdot 10^{-13} $ & \\
% & $12900 $ && $7.6899 \cdot 10^{-30} $ && $9.9969 \cdot 10^{-26} $ && $1.2996 \cdot 10^{-21} $ && $1.6895 \cdot 10^{-17} $ && $2.1963 \cdot 10^{-13} $ & \\
& $13000 $ && $5.5830 \cdot 10^{-30} $ && $7.5370 \cdot 10^{-26} $ && $1.0175 \cdot 10^{-21} $ && $1.3736 \cdot 10^{-17} $ && $1.8544 \cdot 10^{-13} $ & \\
% & $13500 $ && $9.9578 \cdot 10^{-31} $ && $1.3743 \cdot 10^{-26} $ && $1.8966 \cdot 10^{-22} $ && $2.6174 \cdot 10^{-18} $ && $3.6122 \cdot 10^{-14} $ & \\
% & $13800.7464 $ && $3.5592 \cdot 10^{-31} $ && $4.9829 \cdot 10^{-27} $ && $6.9761 \cdot 10^{-23} $ && $9.7665 \cdot 10^{-19} $ && $1.3673 \cdot 10^{-14} $ & \\
& $14000 $ && $1.8398 \cdot 10^{-31} $ && $2.7597 \cdot 10^{-27} $ && $4.1396 \cdot 10^{-23} $ && $6.2094 \cdot 10^{-19} $ && $9.3141 \cdot 10^{-15} $ & \\
& $15000 $ && $6.5711 \cdot 10^{-33} $ && $1.0514 \cdot 10^{-28} $ && $1.6822 \cdot 10^{-24} $ && $2.6915 \cdot 10^{-20} $ && $4.3065 \cdot 10^{-16} $ & \\
& $16000 $ && $2.5738 \cdot 10^{-34} $ && $4.3755 \cdot 10^{-30} $ && $7.4384 \cdot 10^{-26} $ && $1.2645 \cdot 10^{-21} $ && $2.1497 \cdot 10^{-17} $ & \\
& $17000 $ && $1.1167 \cdot 10^{-35} $ && $2.0101 \cdot 10^{-31} $ && $3.6182 \cdot 10^{-27} $ && $6.5127 \cdot 10^{-23} $ && $1.1723 \cdot 10^{-18} $ & \\
& $18000 $ && $5.3738 \cdot 10^{-37} $ && $1.0210 \cdot 10^{-32} $ && $1.9400 \cdot 10^{-28} $ && $3.6859 \cdot 10^{-24} $ && $7.0032 \cdot 10^{-20} $ & \\
& $19000 $ && $2.7357 \cdot 10^{-38} $ && $5.4714 \cdot 10^{-34} $ && $1.0943 \cdot 10^{-29} $ && $2.1886 \cdot 10^{-25} $ && $4.3771 \cdot 10^{-21} $ & \\
& $20000 $ && $1.5040 \cdot 10^{-39} $ && $3.1585 \cdot 10^{-35} $ && $6.6328 \cdot 10^{-31} $ && $1.3929 \cdot 10^{-26} $ && $2.9251 \cdot 10^{-22} $ & \\
& $21000 $ && $9.0605 \cdot 10^{-41} $ && $1.9933 \cdot 10^{-36} $ && $4.3853 \cdot 10^{-32} $ && $9.6476 \cdot 10^{-28} $ && $2.1225 \cdot 10^{-23} $ & \\
& $22000 $ && $5.6101 \cdot 10^{-42} $ && $1.2903 \cdot 10^{-37} $ && $2.9677 \cdot 10^{-33} $ && $6.8258 \cdot 10^{-29} $ && $1.5699 \cdot 10^{-24} $ & \\
& $23000 $ && $3.7554 \cdot 10^{-43} $ && $9.0129 \cdot 10^{-39} $ && $2.1631 \cdot 10^{-34} $ && $5.1914 \cdot 10^{-30} $ && $1.2460 \cdot 10^{-25} $ & \\
& $24000 $ && $2.6755 \cdot 10^{-44} $ && $6.6888 \cdot 10^{-40} $ && $1.6722 \cdot 10^{-35} $ && $4.1805 \cdot 10^{-31} $ && $1.0451 \cdot 10^{-26} $ & \\
\hline
\end{longtable} 
%Calculations shown in \path{Computations-sharper-psi-theta/B-August-2018.pari}.
% The last line is valid for $e^{24\,000} \le x \le e^{25\,000}$
}

 {\footnotesize
% The following table uses Wedeniwski's $H_0 = 2.445 \cdot 10^{12}$ and was calculated using the $B_{j,k}$ values of Table \ref{BkiTableWedeniwski}, by taking the maximum $ B_{j,k}$ for each column for each $b_j$ to $13\,900$, making 
% each $\mathcal{B}_{j,k}$ value listed valid for $e^{b_j} < x < e^{b_{j+1}}$ (since the last value of Table \ref{BkiTableWedeniwski} is valid for $x \in [e^{10000}, e^{13\,900}]$).
\begin{longtable}{ccccccccccccccc}
% \caption{Selected values from the supplement for Table \ref{BkiTableWedeniwski}. All $\mathcal{B}_{j,k}$ values are valid for $x \in [e^{b_j},e^{13\,900}]$.}
%\caption{$|\theta(x)-x|<\frac{ \mathcal{B}_k(b_0)x}{(\log x)^k}$ for all $x \in [e^{b_0},e^{25000})$, where $\mathcal{B}_k(b_0)$ is defined in \eqref{MathcalBbbprime2}. }
\caption{Supplement for Table \ref{BkiTableWedeniwski}. Values for $ \mathcal{B}_{k}(b_0)$ in $|\theta(x)-x|<\frac{ \mathcal{B}_{k}(b_0)x}{(\log x)^k}$ calculated using the method described in Corollary \ref{Cor:Bk} and $\mathcal{B}_k(b_0)$ is defined in \eqref{MathcalBbbprime2}.  Each  $\mathcal{B}_{k}(b_0)$ is valid for $x \in [e^{b_0},e^{25000}]$. }

\label{BkMaxTableWedeniwski}
\\
\hline
&&&&&&&&&&&&& \\[-1em]
\phantom{a} & 
$b_0$         & 
% \phantom{a} &
% ${\mathcal{B}}_{j,0}$       &
\phantom{a} &
${\mathcal{B}}_1(b_0)$       &
\phantom{a} & 
${\mathcal{B}}_2(b_0)$        & 
\phantom{a} & 
${\mathcal{B}}_3(b_0)$        & 
\phantom{a} & 
${\mathcal{B}}_4(b_0)$       & 
\phantom{a} & 
${\mathcal{B}}_5(b_0)$      &
\phantom{a} \\
\hline
\endfirsthead
\multicolumn{13}{c}%
{\tablename\ \thetable{} -- continued from previous page} \\
\hline
&&&&&&&&&&&&& \\[-1em]
\phantom{a} & 
$b_0$         & 
% \phantom{a} &
% ${\mathcal{B}}_{j,0}$       &
\phantom{a} &
${\mathcal{B}}_1(b_0)$       &
\phantom{a} & 
${\mathcal{B}}_2(b_0)$        & 
\phantom{a} & 
${\mathcal{B}}_3(b_0)$        & 
\phantom{a} & 
${\mathcal{B}}_4(b_0)$       & 
\phantom{a} & 
${\mathcal{B}}_5(b_0)$      &
\phantom{a} \\
\hline
\endhead
\hline \multicolumn{13}{r}{Continued on next page} \\ \hline
\endfoot
\hline %\hline
\endlastfoot
&&&&&&&&&&&&& \\[-1em]
& $ 20 $ && $ 1.6844 \cdot 10^{-3 } $ && $ 3.3688 \cdot 10^{-2 } $ && $ 6.7375 \cdot 10^{-1 } $ && $ 5.7184 \cdot 10^{1 } $ && $ 1.3441 \cdot 10^{5 } $ & \\
& $ 21 $ && $ 1.0684 \cdot 10^{-3 } $ && $ 2.2435 \cdot 10^{-2 } $ && $ 4.7114 \cdot 10^{-1 } $ && $ 5.7184 \cdot 10^{1 } $ && $ 1.3441 \cdot 10^{5 } $ & \\
& $ 22 $ && $ 6.7654 \cdot 10^{-4 } $ && $ 1.4884 \cdot 10^{-2 } $ && $ 3.2746 \cdot 10^{-1 } $ && $ 5.7184 \cdot 10^{1 } $ && $ 1.3441 \cdot 10^{5 } $ & \\
& $ 23 $ && $ 4.2780 \cdot 10^{-4 } $ && $ 9.8392 \cdot 10^{-3 } $ && $ 2.2631 \cdot 10^{-1 } $ && $ 5.7184 \cdot 10^{1 } $ && $ 1.3441 \cdot 10^{5 } $ & \\
& $ 24 $ && $ 2.7011 \cdot 10^{-4 } $ && $ 6.4827 \cdot 10^{-3 } $ && $ 1.5559 \cdot 10^{-1 } $ && $ 5.7184 \cdot 10^{1 } $ && $ 1.3441 \cdot 10^{5 } $ & \\
& $ 25 $ && $ 1.7500 \cdot 10^{-4 } $ && $ 4.3750 \cdot 10^{-3 } $ && $ 1.0938 \cdot 10^{-1 } $ && $ 5.7184 \cdot 10^{1 } $ && $ 1.3441 \cdot 10^{5 } $ & \\
& $ 26 $ && $ 1.1022 \cdot 10^{-4 } $ && $ 2.8655 \cdot 10^{-3 } $ && $ 7.4503 \cdot 10^{-2 } $ && $ 5.7184 \cdot 10^{1 } $ && $ 1.3441 \cdot 10^{5 } $ & \\
& $ 27 $ && $ 6.9322 \cdot 10^{-5 } $ && $ 1.8717 \cdot 10^{-3 } $ && $ 5.0536 \cdot 10^{-2 } $ && $ 5.7184 \cdot 10^{1 } $ && $ 1.3441 \cdot 10^{5 } $ & \\
& $ 28 $ && $ 4.3555 \cdot 10^{-5 } $ && $ 1.2196 \cdot 10^{-3 } $ && $ 3.4148 \cdot 10^{-2 } $ && $ 5.7184 \cdot 10^{1 } $ && $ 1.3441 \cdot 10^{5 } $ & \\
& $ 29 $ && $ 2.7336 \cdot 10^{-5 } $ && $ 7.9272 \cdot 10^{-4 } $ && $ 2.4334 \cdot 10^{-2 } $ && $ 5.7184 \cdot 10^{1 } $ && $ 1.3441 \cdot 10^{5 } $ & \\
& $ 30 $ && $ 1.7139 \cdot 10^{-5 } $ && $ 5.1415 \cdot 10^{-4 } $ && $ 2.4334 \cdot 10^{-2 } $ && $ 5.7184 \cdot 10^{1 } $ && $ 1.3441 \cdot 10^{5 } $ & \\
&\vdots&& && && && && & \\
& $ 43 $ && $ 8.6315 \cdot 10^{-7 } $ && $ 3.7979 \cdot 10^{-5 } $ && $ 2.4334 \cdot 10^{-2 } $ && $ 5.7184 \cdot 10^{1 } $ && $ 1.3441 \cdot 10^{5 } $ & \\
& $ 43 $ && $ 8.6315 \cdot 10^{-7 } $ && $ 3.7979 \cdot 10^{-5 } $ && $ 2.4334 \cdot 10^{-2 } $ && $ 5.7184 \cdot 10^{1 } $ && $ 1.3441 \cdot 10^{5 } $ & \\
& $ 44 $ && $ 7.8163 \cdot 10^{-7 } $ && $ 3.5174 \cdot 10^{-5 } $ && $ 2.4334 \cdot 10^{-2 } $ && $ 5.7184 \cdot 10^{1 } $ && $ 1.3441 \cdot 10^{5 } $ & \\
& $ 45 $ && $ 5.0646 \cdot 10^{-7 } $ && $ 2.3298 \cdot 10^{-5 } $ && $ 2.4334 \cdot 10^{-2 } $ && $ 5.7184 \cdot 10^{1 } $ && $ 1.3441 \cdot 10^{5 } $ & \\
& $ 46 $ && $ 3.2935 \cdot 10^{-7 } $ && $ 1.5480 \cdot 10^{-5 } $ && $ 2.4334 \cdot 10^{-2 } $ && $ 5.7184 \cdot 10^{1 } $ && $ 1.3441 \cdot 10^{5 } $ & \\
& $ 47 $ && $ 2.1308 \cdot 10^{-7 } $ && $ 1.0376 \cdot 10^{-5 } $ && $ 2.4334 \cdot 10^{-2 } $ && $ 5.7184 \cdot 10^{1 } $ && $ 1.3441 \cdot 10^{5 } $ & \\
&\vdots&& && && && && & \\
& $ 54 $ && $ 9.8778 \cdot 10^{-9 } $ && $ 1.0376 \cdot 10^{-5 } $ && $ 2.4334 \cdot 10^{-2 } $ && $ 5.7184 \cdot 10^{1 } $ && $ 1.3441 \cdot 10^{5 } $ & \\
& $ 55 $ && $ 6.3417 \cdot 10^{-9 } $ && $ 1.0376 \cdot 10^{-5 } $ && $ 2.4334 \cdot 10^{-2 } $ && $ 5.7184 \cdot 10^{1 } $ && $ 1.3441 \cdot 10^{5 } $ & \\
& $ 56 $ && $ 4.4627 \cdot 10^{-9 } $ && $ 1.0376 \cdot 10^{-5 } $ && $ 2.4334 \cdot 10^{-2 } $ && $ 5.7184 \cdot 10^{1 } $ && $ 1.3441 \cdot 10^{5 } $ & \\
&\vdots&& && && && && & \\
& $ 2275 $ && $ 4.4627 \cdot 10^{-9 } $ && $ 1.0376 \cdot 10^{-5 } $ && $ 2.4334 \cdot 10^{-2 } $ && $ 5.7184 \cdot 10^{1 } $ && $ 1.3441 \cdot 10^{5 } $ & \\
& $ 2300 $ && $ 4.4627 \cdot 10^{-9 } $ && $ 1.0376 \cdot 10^{-5 } $ && $ 2.4334 \cdot 10^{-2 } $ && $ 5.7184 \cdot 10^{1 } $ && $ 1.3441 \cdot 10^{5 } $ & \\
& $ 2325 $ && $ 4.4063 \cdot 10^{-9 } $ && $ 1.0355 \cdot 10^{-5 } $ && $ 2.4334 \cdot 10^{-2 } $ && $ 5.7184 \cdot 10^{1 } $ && $ 1.3441 \cdot 10^{5 } $ & \\
& $ 2350 $ && $ 4.2245 \cdot 10^{-9 } $ && $ 1.0034 \cdot 10^{-5 } $ && $ 2.3829 \cdot 10^{-2 } $ && $ 5.6594 \cdot 10^{1 } $ && $ 1.3441 \cdot 10^{5 } $ & \\
& $ 2375 $ && $ 4.0499 \cdot 10^{-9 } $ && $ 9.7196 \cdot 10^{-6 } $ && $ 2.3328 \cdot 10^{-2 } $ && $ 5.5985 \cdot 10^{1 } $ && $ 1.3437 \cdot 10^{5 } $ & \\
& $ 2400 $ && $ 3.8821 \cdot 10^{-9 } $ && $ 9.4139 \cdot 10^{-6 } $ && $ 2.2829 \cdot 10^{-2 } $ && $ 5.5360 \cdot 10^{1 } $ && $ 1.3425 \cdot 10^{5 } $ & \\
&\vdots&& && && && && & \\
& $ 9800 $ && $ 8.4841 \cdot 10^{-25 } $ && $ 8.3993 \cdot 10^{-21 } $ && $ 8.3153 \cdot 10^{-17 } $ && $ 8.2321 \cdot 10^{-13 } $ && $ 8.1498 \cdot 10^{-9 } $ & \\
& $ 9900 $ && $ 5.7396 \cdot 10^{-25 } $ && $ 5.7396 \cdot 10^{-21 } $ && $ 5.7396 \cdot 10^{-17 } $ && $ 5.7396 \cdot 10^{-13 } $ && $ 5.7396 \cdot 10^{-9 } $ & \\
& $ 10000 $ && $ 3.8228 \cdot 10^{-25 } $ && $ 3.8610 \cdot 10^{-21 } $ && $ 3.8997 \cdot 10^{-17 } $ && $ 3.9387 \cdot 10^{-13 } $ && $ 3.9780 \cdot 10^{-9 } $ & \\
&\vdots&& && && && && & \\
% & $ 10100 $ && $ 2.5746 \cdot 10^{-25 } $ && $ 2.6260 \cdot 10^{-21 } $ && $ 2.6786 \cdot 10^{-17 } $ && $ 2.7321 \cdot 10^{-13 } $ && $ 2.7868 \cdot 10^{-9 } $ & \\
% & $ 10200 $ && $ 1.7389 \cdot 10^{-25 } $ && $ 1.7911 \cdot 10^{-21 } $ && $ 1.8448 \cdot 10^{-17 } $ && $ 1.9002 \cdot 10^{-13 } $ && $ 1.9572 \cdot 10^{-9 } $ & \\
% & $ 10300 $ && $ 1.1734 \cdot 10^{-25 } $ && $ 1.2204 \cdot 10^{-21 } $ && $ 1.2692 \cdot 10^{-17 } $ && $ 1.3199 \cdot 10^{-13 } $ && $ 1.3727 \cdot 10^{-9 } $ & \\
% & $ 10400 $ && $ 7.9556 \cdot 10^{-26 } $ && $ 8.3534 \cdot 10^{-22 } $ && $ 8.7711 \cdot 10^{-18 } $ && $ 9.2096 \cdot 10^{-14 } $ && $ 9.6701 \cdot 10^{-10 } $ & \\
% & $ 10500 $ && $ 5.4077 \cdot 10^{-26 } $ && $ 5.7321 \cdot 10^{-22 } $ && $ 6.0761 \cdot 10^{-18 } $ && $ 6.4406 \cdot 10^{-14 } $ && $ 6.8270 \cdot 10^{-10 } $ & \\
% & $ 10600 $ && $ 3.6845 \cdot 10^{-26 } $ && $ 3.9424 \cdot 10^{-22 } $ && $ 4.2184 \cdot 10^{-18 } $ && $ 4.5137 \cdot 10^{-14 } $ && $ 4.8296 \cdot 10^{-10 } $ & \\
% & $ 10700 $ && $ 2.5150 \cdot 10^{-26 } $ && $ 2.7162 \cdot 10^{-22 } $ && $ 2.9335 \cdot 10^{-18 } $ && $ 3.1682 \cdot 10^{-14 } $ && $ 3.4217 \cdot 10^{-10 } $ & \\
% & $ 10800 $ && $ 1.7201 \cdot 10^{-26 } $ && $ 1.8749 \cdot 10^{-22 } $ && $ 2.0437 \cdot 10^{-18 } $ && $ 2.2276 \cdot 10^{-14 } $ && $ 2.4281 \cdot 10^{-10 } $ & \\
% & $ 10900 $ && $ 1.1640 \cdot 10^{-26 } $ && $ 1.2803 \cdot 10^{-22 } $ && $ 1.4084 \cdot 10^{-18 } $ && $ 1.5492 \cdot 10^{-14 } $ && $ 1.7041 \cdot 10^{-10 } $ & \\
& $ 11000 $ && $ 7.9284 \cdot 10^{-27 } $ && $ 8.8005 \cdot 10^{-23 } $ && $ 9.7685 \cdot 10^{-19 } $ && $ 1.0844 \cdot 10^{-14 } $ && $ 1.2036 \cdot 10^{-10 } $ & \\
% & $ 11100 $ && $ 5.4156 \cdot 10^{-27 } $ && $ 6.0655 \cdot 10^{-23 } $ && $ 6.7933 \cdot 10^{-19 } $ && $ 7.6085 \cdot 10^{-15 } $ && $ 8.5215 \cdot 10^{-11 } $ & \\
% & $ 11200 $ && $ 3.7120 \cdot 10^{-27 } $ && $ 4.1946 \cdot 10^{-23 } $ && $ 4.7398 \cdot 10^{-19 } $ && $ 5.3560 \cdot 10^{-15 } $ && $ 6.0523 \cdot 10^{-11 } $ & \\
% & $ 11300 $ && $ 2.5509 \cdot 10^{-27 } $ && $ 2.9080 \cdot 10^{-23 } $ && $ 3.3152 \cdot 10^{-19 } $ && $ 3.7793 \cdot 10^{-15 } $ && $ 4.3084 \cdot 10^{-11 } $ & \\
% & $ 11400 $ && $ 1.7570 \cdot 10^{-27 } $ && $ 2.0205 \cdot 10^{-23 } $ && $ 2.3236 \cdot 10^{-19 } $ && $ 2.6721 \cdot 10^{-15 } $ && $ 3.0729 \cdot 10^{-11 } $ & \\
% & $ 11500 $ && $ 1.2103 \cdot 10^{-27 } $ && $ 1.4039 \cdot 10^{-23 } $ && $ 1.6285 \cdot 10^{-19 } $ && $ 1.8891 \cdot 10^{-15 } $ && $ 2.1913 \cdot 10^{-11 } $ & \\
% & $ 11600 $ && $ 8.3445 \cdot 10^{-28 } $ && $ 9.7630 \cdot 10^{-24 } $ && $ 1.1423 \cdot 10^{-19 } $ && $ 1.3365 \cdot 10^{-15 } $ && $ 1.5637 \cdot 10^{-11 } $ & \\
% & $ 11700 $ && $ 5.7692 \cdot 10^{-28 } $ && $ 6.8076 \cdot 10^{-24 } $ && $ 8.0330 \cdot 10^{-20 } $ && $ 9.4789 \cdot 10^{-16 } $ && $ 1.1186 \cdot 10^{-11 } $ & \\
% & $ 11800 $ && $ 3.9987 \cdot 10^{-28 } $ && $ 4.7584 \cdot 10^{-24 } $ && $ 5.6625 \cdot 10^{-20 } $ && $ 6.7384 \cdot 10^{-16 } $ && $ 8.0187 \cdot 10^{-12 } $ & \\
% & $ 11900 $ && $ 2.7776 \cdot 10^{-28 } $ && $ 3.3331 \cdot 10^{-24 } $ && $ 3.9997 \cdot 10^{-20 } $ && $ 4.7996 \cdot 10^{-16 } $ && $ 5.7596 \cdot 10^{-12 } $ & \\
& $ 12000 $ && $ 1.9331 \cdot 10^{-28 } $ && $ 2.3390 \cdot 10^{-24 } $ && $ 2.8302 \cdot 10^{-20 } $ && $ 3.4245 \cdot 10^{-16 } $ && $ 4.1437 \cdot 10^{-12 } $ & \\
% & $ 12100 $ && $ 1.3478 \cdot 10^{-28 } $ && $ 1.6443 \cdot 10^{-24 } $ && $ 2.0060 \cdot 10^{-20 } $ && $ 2.4473 \cdot 10^{-16 } $ && $ 2.9857 \cdot 10^{-12 } $ & \\
% & $ 12200 $ && $ 9.3147 \cdot 10^{-29 } $ && $ 1.1457 \cdot 10^{-24 } $ && $ 1.4093 \cdot 10^{-20 } $ && $ 1.7334 \cdot 10^{-16 } $ && $ 2.1320 \cdot 10^{-12 } $ & \\
% & $ 12300 $ && $ 6.5069 \cdot 10^{-29 } $ && $ 8.0686 \cdot 10^{-25 } $ && $ 1.0005 \cdot 10^{-20 } $ && $ 1.2407 \cdot 10^{-16 } $ && $ 1.5384 \cdot 10^{-12 } $ & \\
% & $ 12400 $ && $ 4.5540 \cdot 10^{-29 } $ && $ 5.6924 \cdot 10^{-25 } $ && $ 7.1155 \cdot 10^{-21 } $ && $ 8.8944 \cdot 10^{-17 } $ && $ 1.1118 \cdot 10^{-12 } $ & \\
% & $ 12500 $ && $ 3.1924 \cdot 10^{-29 } $ && $ 4.0224 \cdot 10^{-25 } $ && $ 5.0682 \cdot 10^{-21 } $ && $ 6.3860 \cdot 10^{-17 } $ && $ 8.0463 \cdot 10^{-13 } $ & \\
% & $ 12600 $ && $ 2.2307 \cdot 10^{-29 } $ && $ 2.8330 \cdot 10^{-25 } $ && $ 3.5979 \cdot 10^{-21 } $ && $ 4.5694 \cdot 10^{-17 } $ && $ 5.8031 \cdot 10^{-13 } $ & \\
% & $ 12700 $ && $ 1.5600 \cdot 10^{-29 } $ && $ 1.9967 \cdot 10^{-25 } $ && $ 2.5558 \cdot 10^{-21 } $ && $ 3.2714 \cdot 10^{-17 } $ && $ 4.1874 \cdot 10^{-13 } $ & \\
% & $ 12800 $ && $ 1.0940 \cdot 10^{-29 } $ && $ 1.4112 \cdot 10^{-25 } $ && $ 1.8205 \cdot 10^{-21 } $ && $ 2.3484 \cdot 10^{-17 } $ && $ 3.0294 \cdot 10^{-13 } $ & \\
% & $ 12900 $ && $ 7.6900 \cdot 10^{-30 } $ && $ 9.9970 \cdot 10^{-26 } $ && $ 1.2996 \cdot 10^{-21 } $ && $ 1.6895 \cdot 10^{-17 } $ && $ 2.1964 \cdot 10^{-13 } $ & \\
& $ 13000 $ && $ 5.5830 \cdot 10^{-30 } $ && $ 7.5371 \cdot 10^{-26 } $ && $ 1.0175 \cdot 10^{-21 } $ && $ 1.3737 \cdot 10^{-17 } $ && $ 1.8544 \cdot 10^{-13 } $ & \\
% & $ 13500 $ && $ 9.9578 \cdot 10^{-31 } $ && $ 1.3743 \cdot 10^{-26 } $ && $ 1.8966 \cdot 10^{-22 } $ && $ 2.6174 \cdot 10^{-18 } $ && $ 3.6123 \cdot 10^{-14 } $ & \\
% & $ 13800.7464 $ && $ 3.5593 \cdot 10^{-31 } $ && $ 4.9830 \cdot 10^{-27 } $ && $ 6.9761 \cdot 10^{-23 } $ && $ 9.7665 \cdot 10^{-19 } $ && $ 1.3674 \cdot 10^{-14 } $ & \\
& $ 14000 $ && $ 1.8399 \cdot 10^{-31 } $ && $ 2.7598 \cdot 10^{-27 } $ && $ 4.1396 \cdot 10^{-23 } $ && $ 6.2094 \cdot 10^{-19 } $ && $ 9.3141 \cdot 10^{-15 } $ & \\
& $ 15000 $ && $ 6.5712 \cdot 10^{-33 } $ && $ 1.0514 \cdot 10^{-28 } $ && $ 1.6823 \cdot 10^{-24 } $ && $ 2.6916 \cdot 10^{-20 } $ && $ 4.3065 \cdot 10^{-16 } $ & \\
& $ 16000 $ && $ 2.5739 \cdot 10^{-34 } $ && $ 4.3756 \cdot 10^{-30 } $ && $ 7.4384 \cdot 10^{-26 } $ && $ 1.2646 \cdot 10^{-21 } $ && $ 2.1497 \cdot 10^{-17 } $ & \\
& $ 17000 $ && $ 1.1168 \cdot 10^{-35 } $ && $ 2.0101 \cdot 10^{-31 } $ && $ 3.6182 \cdot 10^{-27 } $ && $ 6.5127 \cdot 10^{-23 } $ && $ 1.1723 \cdot 10^{-18 } $ & \\
& $ 18000 $ && $ 5.3739 \cdot 10^{-37 } $ && $ 1.0211 \cdot 10^{-32 } $ && $ 1.9400 \cdot 10^{-28 } $ && $ 3.6860 \cdot 10^{-24 } $ && $ 7.0033 \cdot 10^{-20 } $ & \\
& $ 19000 $ && $ 2.7357 \cdot 10^{-38 } $ && $ 5.4714 \cdot 10^{-34 } $ && $ 1.0943 \cdot 10^{-29 } $ && $ 2.1886 \cdot 10^{-25 } $ && $ 4.3772 \cdot 10^{-21 } $ & \\
& $ 20000 $ && $ 1.5041 \cdot 10^{-39 } $ && $ 3.1585 \cdot 10^{-35 } $ && $ 6.6329 \cdot 10^{-31 } $ && $ 1.3929 \cdot 10^{-26 } $ && $ 2.9251 \cdot 10^{-22 } $ & \\
& $ 21000 $ && $ 9.0606 \cdot 10^{-41 } $ && $ 1.9934 \cdot 10^{-36 } $ && $ 4.3853 \cdot 10^{-32 } $ && $ 9.6477 \cdot 10^{-28 } $ && $ 2.1225 \cdot 10^{-23 } $ & \\
& $ 22000 $ && $ 5.6101 \cdot 10^{-42 } $ && $ 1.2904 \cdot 10^{-37 } $ && $ 2.9678 \cdot 10^{-33 } $ && $ 6.8258 \cdot 10^{-29 } $ && $ 1.5700 \cdot 10^{-24 } $ & \\
& $ 23000 $ && $ 3.7554 \cdot 10^{-43 } $ && $ 9.0129 \cdot 10^{-39 } $ && $ 2.1631 \cdot 10^{-34 } $ && $ 5.1915 \cdot 10^{-30 } $ && $ 1.2460 \cdot 10^{-25 } $ & \\
& $ 24000 $ && $ 1.3804 \cdot 10^{-43 } $ && $ 3.4508 \cdot 10^{-39 } $ && $ 8.6269 \cdot 10^{-35 } $ && $ 2.1568 \cdot 10^{-30 } $ && $ 5.3919 \cdot 10^{-26 } $ & \\
& $ 25000 $ && $ 1.3804 \cdot 10^{-43 } $ && $ 3.4508 \cdot 10^{-39 } $ && $ 8.6269 \cdot 10^{-35 } $ && $ 2.1568 \cdot 10^{-30 } $ && $ 5.3919 \cdot 10^{-26 } $ & \\
\hline
\end{longtable}
%Calculations shown in \path{Computations-sharper-psi-theta/B-August-2018.pari}
}

\subsection{Lower bound for first values of $x$ ($x \le 10^{19}$)}\footnote{$10^{19} \simeq e^{43.749}$}
%%%%%%%%%%%
\subsubsection{Lower bound for first values of $x \in [e^{J_0},10^{19}]$}
\quad \\
{\footnotesize
% {\color{blue} This table updated January 14, 2019 by Kirsten}
% The following table does not depend on $H_0$.
\begin{longtable}{ccrrrrrr} 
% \caption{Values of $\mathcal{C}_{b,k}$ in $\theta(x) > x - \frac{\mathcal{C}_{b,k} x}{(\log x)^k}$ calculated using the method described in \cref{Cor:Ck}, Section \ref{section:xsmall}.
% Each $\mathcal{C}_{b,k}$ is valid for $e^b \le x \le 10^{19} \simeq e^{43.749}$.}
\caption{$\theta(x) -x >  - \frac{\mathcal{C}_{b,k} x}{(\log x)^k}$ for all $x \in [e^b, 10^{19} )$, where $\mathcal{C}_{b,k}$ is defined in \eqref{defn:mathcalCbk}.}
 \label{CkValues}
\\ 
\hline
\multicolumn{1}{c}{\phantom{}} &
\multicolumn{1}{c}{$b$  } &
% \multicolumn{1}{c}{\phantom{}} &
\multicolumn{1}{c}{$\mathcal{C}_{b,1}$ } &
% \multicolumn{1}{c}{\phantom{a}} &
\multicolumn{1}{c}{$\mathcal{C}_{b,2}$ } &
% \multicolumn{1}{c}{\phantom{a}} &
\multicolumn{1}{c}{$\mathcal{C}_{b,3}$ } &
% \multicolumn{1}{c}{\phantom{a}} &
\multicolumn{1}{c}{$\mathcal{C}_{b,4}$ } &
% \multicolumn{1}{c}{\phantom{a}} &
\multicolumn{1}{c}{$\mathcal{C}_{b,5}$ } &
\multicolumn{1}{c}{\phantom{}} 
\\ \hline 
\endfirsthead
\multicolumn{8}{c}%
{\tablename\ \thetable{} -- continued from previous page} \\
\hline
\multicolumn{1}{c}{\phantom{}} &
\multicolumn{1}{c}{$b$  } &
% \multicolumn{1}{c}{\phantom{}} &
\multicolumn{1}{c}{$\mathcal{C}_{b,1}$ } &
% \multicolumn{1}{c}{\phantom{a}} &
\multicolumn{1}{c}{$\mathcal{C}_{b,2}$ } &
% \multicolumn{1}{c}{\phantom{a}} &
\multicolumn{1}{c}{$\mathcal{C}_{b,3}$ } &
% \multicolumn{1}{c}{\phantom{a}} &
\multicolumn{1}{c}{$\mathcal{C}_{b,4}$ } &
% \multicolumn{1}{c}{\phantom{a}} &
\multicolumn{1}{c}{$\mathcal{C}_{b,5}$ } &
\multicolumn{1}{c}{\phantom{}} 
\\ \hline 
\endhead
\hline \multicolumn{8}{r}{Continued on next page} \\ \hline
\endfoot
\hline 
\endlastfoot
&&&&&&& \\[-1em]
\multicolumn{8}{c}{Calculated using $c=0.8$, $C=0.81$, each value valid up to $5\cdot10^{10}\simeq e^{24.635}$.}\\ \hline
&&&&&&& \\[-1em]
& $ 20 $ & $ 1.68440 \cdot 10^{-3} $ & $ 3.36880 \cdot 10^{-2} $ & $ 6.73750 \cdot 10^{-1} $ & $ 1.34750 \cdot 10^{1} $ & $ 2.69500 \cdot 10^{2 } $ & \\
& $ 21 $ & $ 1.06840 \cdot 10^{-3} $ & $ 2.24350 \cdot 10^{-2} $ & $ 4.71140 \cdot 10^{-1} $ & $ 9.89390 \cdot 10^{0} $ & $ 2.07780 \cdot 10^{2 } $ & \\
& $ 22 $ & $ 6.76540 \cdot 10^{-4} $ & $ 1.48840 \cdot 10^{-2} $ & $ 3.27450 \cdot 10^{-1} $ & $ 7.20380 \cdot 10^{0} $ & $ 1.58490 \cdot 10^{2 } $ & \\
& $ 23 $ & $ 4.27800 \cdot 10^{-4} $ & $ 9.83920 \cdot 10^{-3} $ & $ 2.26310 \cdot 10^{-1} $ & $ 5.20500 \cdot 10^{0} $ & $ 1.19720 \cdot 10^{2 } $ & \\
& $ 24 $ & $ 2.70120 \cdot 10^{-4} $ & $ 6.48290 \cdot 10^{-3} $ & $ 1.55590 \cdot 10^{-1} $ & $ 3.73410 \cdot 10^{0} $ & $ 8.96190 \cdot 10^{1 } $ & \\
\hline
&&&&&&& \\[-1em]
\multicolumn{8}{c}{Calculated using $c=0.88$, $C=0.86$, each value valid up to $32\cdot10^{12}\simeq e^{31.097}$.}\\ \hline
&&&&&&& \\[-1em]
& $ \log(5\cdot 10^{10}) $ & $ 2.01560 \cdot 10^{-4} $ & $ 4.96540 \cdot 10^{-3} $ & $ 1.22330 \cdot 10^{-1} $ & $ 3.01350 \cdot 10^{0} $ & $ 7.42380 \cdot 10^{1 } $ & \\
& $ 25 $ & $ 1.70330 \cdot 10^{-4}$ &$ 4.25830 \cdot 10^{-3} $ & $ 1.06460 \cdot 10^{-1} $ & $ 2.66140 \cdot 10^{0} $ & $ 6.65350 \cdot 10^{1 } $ & \\
& $ 26 $ & $ 1.10220 \cdot 10^{-4} $ & $ 2.86560 \cdot 10^{-3} $ & $ 7.45050 \cdot 10^{-2} $ & $ 1.93720 \cdot 10^{0} $ & $ 5.03650 \cdot 10^{1 } $ & \\
& $ 27 $ & $ 6.93270 \cdot 10^{-5} $ & $ 1.87190 \cdot 10^{-3} $ & $ 5.05400 \cdot 10^{-2} $ & $ 1.36460 \cdot 10^{0} $ & $ 3.68430 \cdot 10^{1 } $ & \\
& $ 28 $ & $ 4.35580 \cdot 10^{-5} $ & $ 1.21970 \cdot 10^{-3} $ & $ 3.41500 \cdot 10^{-2} $ & $ 9.56180 \cdot 10^{-1} $ & $ 2.67730 \cdot 10^{1 } $ & \\
& $ 29 $ & $ 2.73380 \cdot 10^{-5} $ & $ 7.92780 \cdot 10^{-4} $ & $ 2.29910 \cdot 10^{-2} $ & $ 6.66730 \cdot 10^{-1} $ & $ 1.93360 \cdot 10^{1 } $ & \\
& $ 30 $ & $ 1.71400 \cdot 10^{-5} $ & $ 5.14180 \cdot 10^{-4} $ & $ 1.54260 \cdot 10^{-2} $ & $ 4.62760 \cdot 10^{-1} $ & $ 1.38830 \cdot 10^{1 } $ & \\
& $ 31 $ & $ 1.07350 \cdot 10^{-5} $ & $ 3.32790 \cdot 10^{-4} $ & $ 1.03170 \cdot 10^{-2} $ & $ 3.19810 \cdot 10^{-1} $ & $ 9.91400 \cdot 10^{0 } $ & \\
\hline
&&&&&&& \\[-1em]
\multicolumn{8}{c}{Calculated using $c=C=0.94$, each value valid up to $10^{19}\simeq e^{43.749}$.}\\ \hline
&&&&&&& \\[-1em]
& $ \log(3.2\cdot10^{13}) $ & $ 1.02600 \cdot 10^{-5} $ & $ 3.19040 \cdot 10^{-4} $ & $ 9.92090 \cdot 10^{-3} $ & $ 3.08510 \cdot 10^{-1} $ & $ 9.59360 \cdot 10^{0 } $ & \\
& $ 32 $ & $ 6.71750 \cdot 10^{-6} $ & $ 2.14960 \cdot 10^{-4} $ & $ 6.87870 \cdot 10^{-3} $ & $ 2.20120 \cdot 10^{-1} $ & $ 7.04380 \cdot 10^{0 } $ & \\
& $ 33 $ & $ 4.38000 \cdot 10^{-6} $ & $ 1.44540 \cdot 10^{-4} $ & $ 4.76990 \cdot 10^{-3} $ & $ 1.57410 \cdot 10^{-1} $ & $ 5.19440 \cdot 10^{0 } $ & \\
& $ 34 $ & $ 2.73610 \cdot 10^{-6} $ & $ 9.30270 \cdot 10^{-5} $ & $ 3.16300 \cdot 10^{-3} $ & $ 1.07540 \cdot 10^{-1} $ & $ 3.65640 \cdot 10^{0 } $ & \\
& $ 35 $ & $ 1.70780 \cdot 10^{-6} $ & $ 5.97730 \cdot 10^{-5} $ & $ 2.09210 \cdot 10^{-3} $ & $ 7.32220 \cdot 10^{-2} $ & $ 2.56280 \cdot 10^{0 } $ & \\
& $ 36 $ & $ 1.06520 \cdot 10^{-6} $ & $ 3.83460 \cdot 10^{-5} $ & $ 1.38050 \cdot 10^{-3} $ & $ 4.96960 \cdot 10^{-2} $ & $ 1.78910 \cdot 10^{0 } $ & \\
& $ 37 $ & $ 6.63850 \cdot 10^{-7} $ & $ 2.45630 \cdot 10^{-5} $ & $ 9.08810 \cdot 10^{-4} $ & $ 3.36260 \cdot 10^{-2} $ & $ 1.24420 \cdot 10^{0 } $ & \\
& $ 38 $ & $ 4.13450 \cdot 10^{-7} $ & $ 1.57120 \cdot 10^{-5} $ & $ 5.97020 \cdot 10^{-4} $ & $ 2.26870 \cdot 10^{-2} $ & $ 8.62100 \cdot 10^{-1 } $ & \\
& $ 39 $ & $ 2.57330 \cdot 10^{-7} $ & $ 1.00360 \cdot 10^{-5} $ & $ 3.91400 \cdot 10^{-4} $ & $ 1.52650 \cdot 10^{-2} $ & $ 5.95320 \cdot 10^{-1 } $ & \\
& $ 40 $ & $ 1.60060 \cdot 10^{-7} $ & $ 6.40240 \cdot 10^{-6} $ & $ 2.56100 \cdot 10^{-4} $ & $ 1.02440 \cdot 10^{-2} $ & $ 4.09750 \cdot 10^{-1 } $ & \\
& $ 41 $ & $ 9.94970 \cdot 10^{-8} $ & $ 4.07940 \cdot 10^{-6} $ & $ 1.67260 \cdot 10^{-4} $ & $ 6.85740 \cdot 10^{-3} $ & $ 2.81160 \cdot 10^{-1 } $ & \\
& $ 42 $ & $ 6.18140 \cdot 10^{-8} $ & $ 2.59620 \cdot 10^{-6} $ & $ 1.09040 \cdot 10^{-4} $ & $ 4.57970 \cdot 10^{-3} $ & $ 1.92350 \cdot 10^{-1 } $ & \\
& $ 43 $ & $ 3.83820 \cdot 10^{-8} $ & $ 1.65050 \cdot 10^{-6} $ & $ 7.09680 \cdot 10^{-5} $ & $ 3.05170 \cdot 10^{-3} $ & $ 1.31220 \cdot 10^{-1 } $ & \\
\hline
\end{longtable} 
%Calculations shown in \path{Dropbox/Undergraduate-Research-Sharper-Psi-Theta-Summer2017/sharper-bounds-psi-theta-Broadbent-Wilk-Lumley-Kad-Ng/Programs_for_Sharper_bounds/June 2019/C-June-2019.pari}.
}
%%%%%%%
\subsubsection{Numerical Verification}

{\footnotesize
\begin{center}
\begin{longtable}{cc|cccccc}
\caption{$\theta(x) - x > -\frac{\mathcal{D}_{k}(a,b) x}{(\log x)^k}$ for $x \in [a,b)$, where $\mathcal{D}_k(a,b)$ is defined in \eqref{bound:mathcalD}.}
%  \caption{$\mathcal{D}_k(X_0,y_0)$ for selected ranges between $1$ and $7.0 \cdot 10^{11}$. \\
%  $\mathcal{D}_k(X_0,y_0) = \max_{n_0 \le n \le n_1} \left( \frac{(\log p_{n})^k \cdot (p_{n} - \theta(p_{n-1}))}{p_{n}} \right)$
%  and $p_n$ is the prime that provides the maximum $\mathcal{D}_k(X_0,y_0)$ in each subinterval.}
% \caption{Values for $\mathcal{D}_k(a,b)$ in 
% \\
%  $\mathcal{D}_k(a,b) = \max_{n_0 \le n \le n_1} \left( \frac{(\log p_{n})^k \cdot (p_{n} - \theta(p_{n-1}))}{p_{n}} \right)$
% \\
% calculated using the method described in \cref{lem:numerical}, \cref{sec:numerical}. We give values for selected ranges between $1$ and $7.0 \cdot 10^{11}$.}
\label{Table-Dk}
\\
\hline
&&&&&\\[-1em]
 $a$ & $b$ & $\mathcal{D}_0(a,b)$ & $\mathcal{D}_1(a,b)$ & $\mathcal{D}_2(a,b)$ & $\mathcal{D}_3(a,b)$ & $\mathcal{D}_4(a,b)$ & $\mathcal{D}_5(a,b)$ 
\\ \hline
\endfirsthead
\multicolumn{7}{c}%
{\tablename\ \thetable{} -- continued from previous page} \\
\hline
&&&&&\\[-1em]
 $a$ & $b$ & $\mathcal{D}_0(a,b)$ & $\mathcal{D}_1(a,b)$ & $\mathcal{D}_2(a,b)$ & $\mathcal{D}_3(a,b)$ & $\mathcal{D}_4(a,b)$ & $\mathcal{D}_5(a,b)$ 
\\ \hline
\endhead
\hline \multicolumn{7}{r}{Continued on next page} \\ \hline
\endfoot
\hline %\hline
\endlastfoot
$\Big.1$ & $ 1 \cdot10^{5}$ & $1.00000\cdot 10^{0}$ & $1.23228\cdot 10^{0}$ & $3.96481\cdot 10^{0}$ & $2.08282\cdot 10^{1}$ & $1.51224\cdot 10^{2}$ & $1.30475\cdot 10^{3}$ \\
$\Big. 1 \cdot10^{5}$ & $ 5 \cdot10^{5}$ & $4.73131\cdot 10^{-3}$ & $5.53160\cdot 10^{-2}$ & $6.46725\cdot 10^{-1}$ & $7.56118\cdot 10^{0}$ & $8.93458\cdot 10^{1}$ & $1.07895\cdot 10^{3}$ \\

$\Big. 5 \cdot10^{5}$ & $ 1 \cdot10^{6}$ & $1.99799\cdot 10^{-3}$ & $2.67236\cdot 10^{-2}$ & $3.57434\cdot 10^{-1}$ & $4.78076\cdot 10^{0}$ & $6.39437\cdot 10^{1}$ & $8.55260\cdot 10^{2}$ \\

$\Big.1 \cdot 10^{6}$ & $ 5 \cdot 10^{6}$ & $1.67162 \cdot 10^{-3}$ & $2.32393 \cdot 10^{-2}$ & $3.23081 \cdot 10^{-1}$ & $4.49158 \cdot 10^{0}$ & $6.24434 \cdot 10^{1}$ & $8.68209 \cdot 10^{2}$ \\

$\Big. 5 \cdot10^{6}$ & $ 1 \cdot10^{7}$ & $6.56794 \cdot 10^{-4}$ & $1.02357 \cdot 10^{-2}$ & $1.59515 \cdot 10^{-1}$ & $2.48592 \cdot 10^{0}$ & $3.87413 \cdot 10^{1}$ & $6.03754 \cdot 10^{2}$ \\

$\Big. 1 \cdot 10^{7}$ & $ 5 \cdot 10^{7}$ & $5.24943 \cdot 10^{-4}$ & $8.47248 \cdot 10^{-3}$ & $1.36744 \cdot 10^{-1}$ & $2.20703 \cdot 10^{0}$ & $3.56210 \cdot 10^{1}$ & $5.74917 \cdot 10^{2}$ \\

$\Big. 5 \cdot10^{7}$ & $ 1 \cdot10^{8}$ & $2.16306 \cdot 10^{-4}$ & $3.85492 \cdot 10^{-3}$ & $6.87010 \cdot 10^{-2}$ & $1.22436 \cdot 10^{0}$ & $2.18201 \cdot 10^{1}$ & $3.88870 \cdot 10^{2}$ \\

$\Big.1 \cdot 10^{8}$ & $ 1 \cdot 10^{9}$ & $1.48989 \cdot 10^{-4}$ & $2.74568 \cdot 10^{-3}$ & $5.06111 \cdot 10^{-2}$ & $9.4259\cdot 10^{-1}$ & $1.74450\cdot 10^{1}$ & $3.26946\cdot 10^{2}$ \\

$\Big.1 \cdot 10^{9}$ & $ 1 \cdot10^{10}$ & $ 4.57993 \cdot 10^{-5}$ & $ 9.59129 \cdot 10^{-4}$ & $ 2.00861 \cdot 10^{-2}$ & $ 4.20644 \cdot 10^{-1}$ & $ 8.80913 \cdot 10^{0}$ & $ 1.84481 \cdot 10^{2}$ \\

$\Big.1\cdot10^{10}$ & $ 2\cdot10^{10}$ & $1.63137 \cdot 10^{-5}$ &  $3.77870 \cdot 10^{-4}$ &  $8.75253\cdot 10^{-3}$ & $2.02733\cdot 10^{-1}$ & $4.69587\cdot 10^{0}$ & $1.08770\cdot 10^{2}$ \\

$\Big.19\, 035\, 709\, 163$ & $ 2\cdot10^{10}$ & $1.13110 \cdot 10^{-5}$ & $2.67726 \cdot 10^{-4}$ & $6.33697 \cdot 10^{-3}$ & $1.49993\cdot 10^{-1}$ & $3.55028\cdot 10^{0}$ & $8.40336\cdot 10^{1}$ \\
$\Big.2\cdot10^{10}$ & $ 5\cdot10^{10}$ & $ 1.03687 \cdot 10^{-5}$ & $ 2.46002 \cdot 10^{-4}$ & $ 5.83648 \cdot 10^{-3}$ & $ 1.38472 \cdot 10^{-1}$ & $ 3.28531 \cdot 10^{0}$ & $ 7.79452 \cdot 10^{1}$ \\
$\Big.5\cdot10^{10}$ & $ 10\cdot10^{10}$ & $ 7.53086 \cdot 10^{-6}$ & $ 1.85559 \cdot 10^{-4}$ & $ 4.57216 \cdot 10^{-3}$ & $ 1.12657 \cdot 10^{-1}$ & $ 2.77586 \cdot 10^{0}$ & $ 6.83969 \cdot 10^{1}$ \\
$\Big.10\cdot10^{10}$ & $ 20\cdot10^{10}$ & $ 5.26640 \cdot 10^{-6}$ & $ 1.33915 \cdot 10^{-4}$ & $ 3.40521 \cdot 10^{-3}$ & $ 8.65882 \cdot 10^{-2}$ & $ 2.20178 \cdot 10^{0}$ & $ 5.59872 \cdot 10^{1}$ \\
$\Big.20\cdot10^{10}$ & $ 30\cdot10^{10}$ & $ 3.00664 \cdot 10^{-6}$ & $ 7.86826 \cdot 10^{-5}$ & $ 2.05910 \cdot 10^{-3}$ & $ 5.38859 \cdot 10^{-2}$ & $ 1.41018 \cdot 10^{0}$ & $ 3.69038 \cdot 10^{1}$ \\
$\Big.30\cdot10^{10}$ & $ 40\cdot10^{10}$ & $ 2.41963 \cdot 10^{-6}$ & $ 6.39936 \cdot 10^{-5}$ & $ 1.69249 \cdot 10^{-3}$ & $ 4.47624 \cdot 10^{-2}$ & $ 1.18386 \cdot 10^{0}$ & $ 3.16250 \cdot 10^{1}$ \\
$\Big.40\cdot10^{10}$ & $ 50\cdot10^{10}$ & $ 2.62662 \cdot 10^{-6}$ & $ 7.01926 \cdot 10^{-5}$ & $ 1.87579 \cdot 10^{-3}$ & $ 5.01279 \cdot 10^{-2}$ & $ 1.33959 \cdot 10^{0}$ & $ 3.57987 \cdot 10^{1}$ \\
$\Big.50\cdot10^{10}$ & $ 60\cdot10^{10}$ & $ 1.89356 \cdot 10^{-6}$ & $ 5.10206 \cdot 10^{-5}$ & $ 1.37472 \cdot 10^{-3}$ & $ 3.70409 \cdot 10^{-2}$ & $ 9.98044 \cdot 10^{-1}$ & $ 2.68917 \cdot 10^{1}$ \\
$\Big.60\cdot10^{10}$ & $ 70\cdot10^{10}$ & $ 1.75478 \cdot 10^{-6}$ & $ 4.78305 \cdot 10^{-5}$ & $ 1.30373 \cdot 10^{-3}$ & $ 3.55359 \cdot 10^{-2}$ & $ 9.68610 \cdot 10^{-1}$ & $ 2.64016 \cdot 10^{1}$ \\
\end{longtable}
\end{center}
% Files: \path{Computations-sharper-psi-theta/DTableComputations/D-OneRow.pari}, \path{Computations-sharper-psi-theta/DTableComputations/D-Table.pari},
%  and the assorted \path{.out} files in\path{Computations-sharper-psi-theta/DTableComputations/ThetaAdd}.
}

\subsection{Final results for Theorem \ref{MainResult:theta}} \
\subsubsection{Values for Theorem \ref{MainResult:theta} for $k = 0$.} \

{\footnotesize
\begin{center}
\begin{table}[h!] 

\caption{$(1-m_0)x < \theta(x) < (1+M_0)x$ for all $x > X_0 = X_1$, where $m_0$ and $M_0$ are defined in \cref{section:k=0}.}
\label{ThetaEpsilon}
\begin{tabular}{|cc|ccc|}
\hline
& $\log X_0 = \log X_1$ & $M_0$ & $m_0$ & \\ \hline
&&&&\\[-1em]
& $20 $ & $ 4.2676 \cdot 10^{-5} $ & $ 9.1639 \cdot 10^{-5}$ & \\
% & $21 $ & $ 2.5884 \cdot 10^{-5 $ & $ 5.5407 \cdot 10^{-5}$ & \\
% & $22 $ & $ 1.5700 \cdot 10^{-5 $ & $ 3.3517 \cdot 10^{-5}$ & \\
% & $23 $ & $ 9.5223 \cdot 10^{-6 $ & $ 2.0285 \cdot 10^{-5}$ & \\
% & $24 $ & $ 5.7756 \cdot 10^{-6 $ & $ 1.2281 \cdot 10^{-5}$ & \\
& $25 $ & $ 3.5031 \cdot 10^{-6} $ & $ 7.4366 \cdot 10^{-6}$ & \\
% & $26 $ & $ 2.1247 \cdot 10^{-6 $ & $ 4.5046 \cdot 10^{-6}$ & \\
% & $27 $ & $ 1.2887 \cdot 10^{-6 $ & $ 2.7292 \cdot 10^{-6}$ & \\
% & $28 $ & $ 7.8164 \cdot 10^{-7 $ & $ 1.6538 \cdot 10^{-6}$ & \\
% & $29 $ & $ 4.7409 \cdot 10^{-7 $ & $ 1.0023 \cdot 10^{-6}$ & \\
& $30 $ & $ 2.8755 \cdot 10^{-7} $ & $ 6.0751 \cdot 10^{-7}$ & \\
% & $31 $ & $ 1.7441 \cdot 10^{-7 $ & $ 3.6827 \cdot 10^{-7}$ & \\
% & $32 $ & $ 1.0578 \cdot 10^{-7 $ & $ 2.2327 \cdot 10^{-7}$ & \\
% & $33 $ & $ 6.4161 \cdot 10^{-8 $ & $ 1.3537 \cdot 10^{-7}$ & \\
% & $34 $ & $ 3.8916 \cdot 10^{-8 $ & $ 8.2073 \cdot 10^{-8}$ & \\
& $35 $ & $ 2.3603 \cdot 10^{-8} $ & $ 4.9766 \cdot 10^{-8}$ & \\
% & $36 $ & $ 1.9338 \cdot 10^{-8 $ & $ 3.5199 \cdot 10^{-8}$ & \\
% & $37 $ & $ 1.9338 \cdot 10^{-8 $ & $ 2.8955 \cdot 10^{-8}$ & \\
% & $38 $ & $ 1.9338 \cdot 10^{-8 $ & $ 2.5169 \cdot 10^{-8}$ & \\
% & $39 $ & $ 1.9338 \cdot 10^{-8 $ & $ 2.2874 \cdot 10^{-8}$ & \\
& $40 $ & $ 1.9338 \cdot 10^{-8} $ & $ 2.1482 \cdot 10^{-8}$ & \\
% & $41 $ & $ 1.9338 \cdot 10^{-8 $ & $ 2.0638 \cdot 10^{-8}$ & \\
% & $42 $ & $ 1.9338 \cdot 10^{-8 $ & $ 2.0127 \cdot 10^{-8}$ & \\
% & $43 $ & $ 1.9338 \cdot 10^{-8 $ & $ 1.9816 \cdot 10^{-8}$ & \\
& $19\log 10 $ & $ 1.9338 \cdot 10^{-8} $ & $ 1.9667 \cdot 10^{-8}$ & \\
% & $44 $ & $ 1.7200 \cdot 10^{-8 $ & $ 1.7491 \cdot 10^{-8}$ & \\
& $45 $ & $ 1.0907 \cdot 10^{-8} $ & $ 1.1084 \cdot 10^{-8}$ & \\
% & $46 $ & $ 6.9451 \cdot 10^{-9 $ & $ 7.0518 \cdot 10^{-9}$ & \\
% & $47 $ & $ 4.4012 \cdot 10^{-9 $ & $ 4.4660 \cdot 10^{-9}$ & \\
% & $48 $ & $ 2.7914 \cdot 10^{-9 $ & $ 2.8307 \cdot 10^{-9}$ & \\
% & $49 $ & $ 1.7689 \cdot 10^{-9 $ & $ 1.7927 \cdot 10^{-9}$ & \\
& $50 $ & $ 1.1199 \cdot 10^{-9} $ & $ 1.1344 \cdot 10^{-9}$ & \\
% & $51 $ & $ 7.0924 \cdot 10^{-10 $ & $ 7.1800 \cdot 10^{-10}$ & \\
% & $52 $ & $ 4.4781 \cdot 10^{-10 $ & $ 4.5312 \cdot 10^{-10}$ & \\
% & $53 $ & $ 2.8280 \cdot 10^{-10 $ & $ 2.8602 \cdot 10^{-10}$ & \\
% & $54 $ & $ 1.7846 \cdot 10^{-10 $ & $ 1.8041 \cdot 10^{-10}$ & \\
% & $55 $ & $ 1.1255 \cdot 10^{-10 $ & $ 1.1374 \cdot 10^{-10}$ & \\
% & $56 $ & $ 7.0928 \cdot 10^{-11 $ & $ 7.1647 \cdot 10^{-11}$ & \\
% & $57 $ & $ 4.4655 \cdot 10^{-11 $ & $ 4.5091 \cdot 10^{-11}$ & \\
% & $58 $ & $ 2.8328 \cdot 10^{-11 $ & $ 2.8593 \cdot 10^{-11}$ & \\
% & $59 $ & $ 1.8440 \cdot 10^{-11 $ & $ 1.8601 \cdot 10^{-11}$ & \\
& $60 $ & $ 1.2215 \cdot 10^{-11} $ & $ 1.2312 \cdot 10^{-11}$ & \\
% & $65 $ & $ 3.5713 \cdot 10^{-12 $ & $ 3.5793 \cdot 10^{-12}$ & \\
& $70 $ & $ 2.7923 \cdot 10^{-12} $ & $ 2.7930 \cdot 10^{-12}$ & \\
% & $75 $ & $ 2.7036 \cdot 10^{-12 $ & $ 2.7037 \cdot 10^{-12}$ & \\
& $80 $ & $ 2.6108 \cdot 10^{-12} $ & $ 2.6108 \cdot 10^{-12}$ & \\
% & $85 $ & $ 2.5692 \cdot 10^{-12 $ & $ 2.5693 \cdot 10^{-12}$ & \\
& $90 $ & $ 2.5213 \cdot 10^{-12} $ & $ 2.5213 \cdot 10^{-12}$ & \\
% & $95 $ & $ 2.4919 \cdot 10^{-12 $ & $ 2.4920 \cdot 10^{-12}$ & \\
& $100 $ & $ 2.4530 \cdot 10^{-12} $ & $ 2.4530 \cdot 10^{-12}$ & \\
& $200 $ & $ 2.1815 \cdot 10^{-12} $ & $ 2.1816 \cdot 10^{-12}$ & \\
& $300 $ & $ 2.0902 \cdot 10^{-12} $ & $ 2.0903 \cdot 10^{-12}$ & \\
& $400 $ & $ 2.0398 \cdot 10^{-12} $ & $ 2.0399 \cdot 10^{-12}$ & \\
& $500 $ & $ 1.9999 \cdot 10^{-12} $ & $ 1.9999 \cdot 10^{-12}$ & \\
\hline
\end{tabular}
\begin{tabular}{|cc|ccc|}
\hline
& $\log X_0 = \log X_1$ & $M_0$ & $m_0$ & \\ \hline
&&&&\\[-1em]
% & $600 $ & $ 1.9889 \cdot 10^{-12} $ & $ 1.9890 \cdot 10^{-12}$ & \\
& $700 $ & $ 1.9764 \cdot 10^{-12} $ & $ 1.9765 \cdot 10^{-12}$ & \\
% & $800 $ & $ 1.9671 \cdot 10^{-12} $ & $ 1.9671 \cdot 10^{-12}$ & \\
% & $900 $ & $ 1.9599 \cdot 10^{-12} $ & $ 1.9599 \cdot 10^{-12}$ & \\
& $1000 $ & $ 1.9475 \cdot 10^{-12} $ & $ 1.9476 \cdot 10^{-12}$ & \\
% & $1500 $ & $ 1.9368 \cdot 10^{-12} $ & $ 1.9368 \cdot 10^{-12}$ & \\
% & $1600 $ & $ 1.9292 \cdot 10^{-12} $ & $ 1.9293 \cdot 10^{-12}$ & \\
% & $1700 $ & $ 1.9273 \cdot 10^{-12} $ & $ 1.9274 \cdot 10^{-12}$ & \\
% & $1725 $ & $ 1.9269 \cdot 10^{-12} $ & $ 1.9269 \cdot 10^{-12}$ & \\
% & $1750 $ & $ 1.9265 \cdot 10^{-12} $ & $ 1.9265 \cdot 10^{-12}$ & \\
% & $1775 $ & $ 1.9260 \cdot 10^{-12} $ & $ 1.9261 \cdot 10^{-12}$ & \\
% & $1800 $ & $ 1.9256 \cdot 10^{-12} $ & $ 1.9257 \cdot 10^{-12}$ & \\
% & $1825 $ & $ 1.9253 \cdot 10^{-12} $ & $ 1.9253 \cdot 10^{-12}$ & \\
% & $1850 $ & $ 1.9249 \cdot 10^{-12} $ & $ 1.9249 \cdot 10^{-12}$ & \\
% & $1875 $ & $ 1.9245 \cdot 10^{-12} $ & $ 1.9245 \cdot 10^{-12}$ & \\
% & $1900 $ & $ 1.9241 \cdot 10^{-12} $ & $ 1.9242 \cdot 10^{-12}$ & \\
% & $1925 $ & $ 1.9238 \cdot 10^{-12} $ & $ 1.9238 \cdot 10^{-12}$ & \\
% & $1950 $ & $ 1.9234 \cdot 10^{-12} $ & $ 1.9235 \cdot 10^{-12}$ & \\
% & $1975 $ & $ 1.9231 \cdot 10^{-12} $ & $ 1.9232 \cdot 10^{-12}$ & \\
& $2000 $ & $ 1.9228 \cdot 10^{-12} $ & $ 1.9228 \cdot 10^{-12}$ & \\
& $3000 $ & $ 4.5997 \cdot 10^{-14} $ & $ 4.5998 \cdot 10^{-14}$ & \\
& $4000 $ & $ 1.4263 \cdot 10^{-16} $ & $ 1.4264 \cdot 10^{-16}$ & \\
& $5000 $ & $ 5.6303 \cdot 10^{-19} $ & $ 5.6303 \cdot 10^{-19}$ & \\
% & $5100 $ & $ 3.2859 \cdot 10^{-19} $ & $ 3.2859 \cdot 10^{-19}$ & \\
% & $5200 $ & $ 1.9217 \cdot 10^{-19} $ & $ 1.9217 \cdot 10^{-19}$ & \\
% & $5300 $ & $ 1.1292 \cdot 10^{-19} $ & $ 1.1292 \cdot 10^{-19}$ & \\
% & $5400 $ & $ 6.6312 \cdot 10^{-20} $ & $ 6.6312 \cdot 10^{-20}$ & \\
% & $5500 $ & $ 3.9135 \cdot 10^{-20} $ & $ 3.9135 \cdot 10^{-20}$ & \\
% & $5600 $ & $ 2.3188 \cdot 10^{-20} $ & $ 2.3188 \cdot 10^{-20}$ & \\
% & $5700 $ & $ 1.3807 \cdot 10^{-20} $ & $ 1.3807 \cdot 10^{-20}$ & \\
% & $5800 $ & $ 8.2234 \cdot 10^{-21} $ & $ 8.2235 \cdot 10^{-21}$ & \\
% & $5900 $ & $ 4.9137 \cdot 10^{-21} $ & $ 4.9137 \cdot 10^{-21}$ & \\
% & $6000 $ & $ 2.9429 \cdot 10^{-21} $ & $ 2.9429 \cdot 10^{-21}$ & \\
% & $6100 $ & $ 1.7721 \cdot 10^{-21} $ & $ 1.7722 \cdot 10^{-21}$ & \\
% & $6200 $ & $ 1.0663 \cdot 10^{-21} $ & $ 1.0664 \cdot 10^{-21}$ & \\
% & $6300 $ & $ 6.4480 \cdot 10^{-22} $ & $ 6.4481 \cdot 10^{-22}$ & \\
% & $6400 $ & $ 3.9201 \cdot 10^{-22} $ & $ 3.9202 \cdot 10^{-22}$ & \\
% & $6500 $ & $ 2.3849 \cdot 10^{-22} $ & $ 2.3850 \cdot 10^{-22}$ & \\
% & $6600 $ & $ 1.4520 \cdot 10^{-22} $ & $ 1.4521 \cdot 10^{-22}$ & \\
% & $6700 $ & $ 8.8893 \cdot 10^{-23} $ & $ 8.8893 \cdot 10^{-23}$ & \\
% & $6800 $ & $ 5.4567 \cdot 10^{-23} $ & $ 5.4567 \cdot 10^{-23}$ & \\
% & $6900 $ & $ 3.3649 \cdot 10^{-23} $ & $ 3.3649 \cdot 10^{-23}$ & \\
& $7000 $ & $ 2.0765 \cdot 10^{-23} $ & $ 2.0766 \cdot 10^{-23}$ & \\
& $10000 $ & $ 3.7849 \cdot 10^{-29} $ & $ 3.7850 \cdot 10^{-29}$ & \\
% & $10100 $ & $ 2.5240 \cdot 10^{-29} $ & $ 2.5241 \cdot 10^{-29}$ & \\
% & $10200 $ & $ 1.6882 \cdot 10^{-29} $ & $ 1.6883 \cdot 10^{-29}$ & \\
% & $10300 $ & $ 1.1282 \cdot 10^{-29} $ & $ 1.1283 \cdot 10^{-29}$ & \\
% & $10400 $ & $ 7.5767 \cdot 10^{-30} $ & $ 7.5768 \cdot 10^{-30}$ & \\
% & $10500 $ & $ 5.1015 \cdot 10^{-30} $ & $ 5.1016 \cdot 10^{-30}$ & \\
% & $10600 $ & $ 3.4434 \cdot 10^{-30} $ & $ 3.4435 \cdot 10^{-30}$ & \\
% & $10700 $ & $ 2.3287 \cdot 10^{-30} $ & $ 2.3287 \cdot 10^{-30}$ & \\
% & $10800 $ & $ 1.5781 \cdot 10^{-30} $ & $ 1.5781 \cdot 10^{-30}$ & \\
% & $10900 $ & $ 1.0581 \cdot 10^{-30} $ & $ 1.0581 \cdot 10^{-30}$ & \\
& $11000 $ & $ 7.1426 \cdot 10^{-31} $ & $ 7.1427 \cdot 10^{-31}$ & \\
% & $11100 $ & $ 4.8353 \cdot 10^{-31} $ & $ 4.8354 \cdot 10^{-31}$ & \\
% & $11200 $ & $ 3.2849 \cdot 10^{-31} $ & $ 3.2850 \cdot 10^{-31}$ & \\
% & $11300 $ & $ 2.2376 \cdot 10^{-31} $ & $ 2.2377 \cdot 10^{-31}$ & \\
% & $11400 $ & $ 1.5278 \cdot 10^{-31} $ & $ 1.5278 \cdot 10^{-31}$ & \\
% & $11500 $ & $ 1.0433 \cdot 10^{-31} $ & $ 1.0433 \cdot 10^{-31}$ & \\
% & $11600 $ & $ 7.1320 \cdot 10^{-32} $ & $ 7.1320 \cdot 10^{-32}$ & \\
% & $11700 $ & $ 4.8891 \cdot 10^{-32} $ & $ 4.8892 \cdot 10^{-32}$ & \\
% & $11800 $ & $ 3.3602 \cdot 10^{-32} $ & $ 3.3603 \cdot 10^{-32}$ & \\
% & $11900 $ & $ 2.3146 \cdot 10^{-32} $ & $ 2.3147 \cdot 10^{-32}$ & \\
& $12000 $ & $ 1.5975 \cdot 10^{-32} $ & $ 1.5976 \cdot 10^{-32}$ & \\
% & $12100 $ & $ 1.1047 \cdot 10^{-32} $ & $ 1.1048 \cdot 10^{-32}$ & \\
% & $12200 $ & $ 7.5729 \cdot 10^{-33} $ & $ 7.5729 \cdot 10^{-33}$ & \\
% & $12300 $ & $ 5.2475 \cdot 10^{-33} $ & $ 5.2475 \cdot 10^{-33}$ & \\
% & $12400 $ & $ 3.6431 \cdot 10^{-33} $ & $ 3.6432 \cdot 10^{-33}$ & \\
% & $12500 $ & $ 2.5336 \cdot 10^{-33} $ & $ 2.5337 \cdot 10^{-33}$ & \\
% & $12600 $ & $ 1.7564 \cdot 10^{-33} $ & $ 1.7565 \cdot 10^{-33}$ & \\
% & $12700 $ & $ 1.2187 \cdot 10^{-33} $ & $ 1.2187 \cdot 10^{-33}$ & \\
% & $12800 $ & $ 8.4802 \cdot 10^{-34} $ & $ 8.4803 \cdot 10^{-34}$ & \\
% & $12900 $ & $ 5.9153 \cdot 10^{-34} $ & $ 5.9154 \cdot 10^{-34}$ & \\
& $13000 $ & $ 4.1355 \cdot 10^{-34} $ & $ 4.1356 \cdot 10^{-34}$ & \\
% & $13500 $ & $ 7.2154 \cdot 10^{-35} $ & $ 7.2154 \cdot 10^{-35}$ & \\
& $13800.7464 $ & $ 2.5423 \cdot 10^{-35} $ & $ 2.5424 \cdot 10^{-35}$ & \\
% & $14000 $ & $ 1.2265 \cdot 10^{-35} $ & $ 1.2266 \cdot 10^{-35}$ & \\
& $15000 $ & $ 4.1070 \cdot 10^{-37} $ & $ 4.1070 \cdot 10^{-37}$ & \\
% & $16000 $ & $ 1.5140 \cdot 10^{-38} $ & $ 1.5141 \cdot 10^{-38}$ & \\
& $17000 $ & $ 6.2040 \cdot 10^{-40} $ & $ 6.2040 \cdot 10^{-40}$ & \\
% & $18000 $ & $ 2.8283 \cdot 10^{-41} $ & $ 2.8284 \cdot 10^{-41}$ & \\
% & $19000 $ & $ 1.3678 \cdot 10^{-42} $ & $ 1.3679 \cdot 10^{-42}$ & \\
& $20000 $ & $ 7.1621 \cdot 10^{-44} $ & $ 7.1621 \cdot 10^{-44}$ & \\
% & $21000 $ & $ 4.1184 \cdot 10^{-45} $ & $ 4.1185 \cdot 10^{-45}$ & \\
& $22000 $ & $ 2.4392 \cdot 10^{-46} $ & $ 2.4392 \cdot 10^{-46}$ & \\
% & $23000 $ & $ 1.5647 \cdot 10^{-47} $ & $ 1.5648 \cdot 10^{-47}$ & \\
% & $24000 $ & $ 1.0702 \cdot 10^{-48} $ & $ 1.0703 \cdot 10^{-48}$ & \\
& $25000 $ & $ 7.5724 \cdot 10^{-50} $ & $ 7.5724 \cdot 10^{-50}$ & \\
\hline
\end{tabular}
\end{table}
\end{center}
% Files: \path{Computations-sharper-psi-theta/ExplicitK0.pari}.
}

 \pagebreak

\subsubsection{Values for Theorem \ref{MainResult:theta} for $k \in \{1,2,3,4,5 \}$.} \

{\footnotesize
% {\color{red!30} Updated August 11, 2020 with new CombinedABCTable.pari. Command used:CreateMasterResults(19*log(10),log(7*10^11),bList,SmallCombined) }
\begin{center}
\begin{longtable}{cc|ccccccccccc}
%  \caption{Combined results for Theorem \ref{MainResult:theta}. For the small range, $M_k$ values for $\log X_0 \le 19\log 10$ can be reduced to the $19\log 10$ row. 
%  In addition, $X_1$ can be reduced to $0$.}
\caption{For $k \in \{1,2,3,4,5\}$, 
$ -\frac{m_k x}{(\log x)^k} < \theta(x) - x$ for $x > X_0 $
and
 $ \theta(x) - x < \frac{M_k x}{(\log x)^k}$ for  $x > X_1 $
,where $m_k$ and $M_k$ are defined in Section \ref{thetalemmataproofs}. }
\label{Table:MasterResults}
 \\ \hline  
 &&&&&&&&&&&&\\[-1em]
\phantom{} & $\log X_0 = \log X_1$ && $m_1 = M_1$ && $m_2 = M_2$ && $m_3 = M_3$ && $m_4 = M_4$ && $m_5 = M_5$ & \\
\hline
\endfirsthead
\multicolumn{6}{c}%
{\tablename\ \thetable{} -- continued from previous page} \\
\hline
&&&&&&&&&&&&\\[-1em]
\phantom{} & $\log X_0 = \log X_1$ && $m_1 = M_1$ && $m_2 = M_2$ && $m_3 = M_3$ && $m_4 = M_4$ && $m_5 = M_5$ & \\
\hline 
\endhead
\hline \multicolumn{6}{r}{Continued on next page} \\ \hline
\endfoot
\hline 
\endlastfoot
 \hline 
 &&&&&&&&&&&&\\[-1em]
& $ 0 $ && $ 1.2323 \cdot10^{0 } $ && $ 3.9649 \cdot10^{0 } $ && $ 2.0829 \cdot10^{1 } $ && $ 1.5123 \cdot10^{2 } $ && $ 1.3441 \cdot10^{5 } $ & \\
& $ \log(10^5) $ && $ 5.5316 \cdot10^{-2 } $ && $ 6.4673 \cdot10^{-1 } $ && $ 7.5612 \cdot10^{0 } $ && $ 8.9346 \cdot10^{1 } $ && $ 1.3441 \cdot10^{5 } $ & \\
& $ \log(5\cdot10^{5}) $ && $ 2.6724 \cdot10^{-2 } $ && $ 3.5744 \cdot10^{-1 } $ && $ 4.7808 \cdot10^{0 } $ && $ 6.3944 \cdot10^{1 } $ && $ 1.3441 \cdot10^{5 } $ & \\
& $ \log(10^{6}) $ && $ 2.3240 \cdot10^{-2 } $ && $ 3.2309 \cdot10^{-1 } $ && $ 4.4916 \cdot10^{0 } $ && $ 6.2444 \cdot10^{1 } $ && $ 1.3441 \cdot10^{5 } $ & \\
& $ \log(5\cdot10^{6}) $ && $ 1.0236 \cdot10^{-2 } $ && $ 1.5952 \cdot10^{-1 } $ && $ 2.4860 \cdot10^{0 } $ && $ 5.7184 \cdot10^{1 } $ && $ 1.3441 \cdot10^{5 } $ & \\
& $ \log(10^{7}) $ && $ 8.4725 \cdot10^{-3 } $ && $ 1.3675 \cdot10^{-1 } $ && $ 2.2071 \cdot10^{0 } $ && $ 5.7184 \cdot10^{1 } $ && $ 1.3441 \cdot10^{5 } $ & \\
& $ \log(5\cdot10^{7}) $ && $ 3.8550 \cdot10^{-3 } $ && $ 6.8701 \cdot10^{-2 } $ && $ 1.2244 \cdot10^{0 } $ && $ 5.7184 \cdot10^{1 } $ && $ 1.3441 \cdot10^{5 } $ & \\
& $ \log(10^{8}) $ && $ 2.7457 \cdot10^{-3 } $ && $ 5.0612 \cdot10^{-2 } $ && $ 9.4259 \cdot10^{-1 } $ && $ 5.7184 \cdot10^{1 } $ && $ 1.3441 \cdot10^{5 } $ & \\
& $ \log(10^{9}) $ && $ 9.5913 \cdot10^{-4 } $ && $ 2.0087 \cdot10^{-2 } $ && $ 4.2065 \cdot10^{-1 } $ && $ 5.7184 \cdot10^{1 } $ && $ 1.3441 \cdot10^{5 } $ & \\
& $ \log(10^{10}) $ && $ 3.7787 \cdot10^{-4 } $ && $ 8.7526 \cdot10^{-3 } $ && $ 2.0274 \cdot10^{-1 } $ && $ 5.7184 \cdot10^{1 } $ && $ 1.3441 \cdot10^{5 } $ & \\
& $ \log(19035709163) $ && $ 2.6773 \cdot10^{-4 } $ && $ 6.3370 \cdot10^{-3 } $ && $ 1.5000 \cdot10^{-1 } $ && $ 5.7184 \cdot10^{1 } $ && $ 1.3441 \cdot10^{5 } $ & \\
& $ \log(2\cdot 10^{10}) $ && $ 2.4601 \cdot10^{-4 } $ && $ 5.8365 \cdot10^{-3 } $ && $ 1.3848 \cdot10^{-1 } $ && $ 5.7184 \cdot10^{1 } $ && $ 1.3441 \cdot10^{5 } $ & \\
& $ \log(5\cdot 10^{10}) $ && $ 1.8556 \cdot10^{-4 } $ && $ 4.5722 \cdot10^{-3 } $ && $ 1.1266 \cdot10^{-1 } $ && $ 5.7184 \cdot10^{1 } $ && $ 1.3441 \cdot10^{5 } $ & \\
& $ \log(10^{11}) $ && $ 1.3392 \cdot10^{-4 } $ && $ 3.4053 \cdot10^{-3 } $ && $ 8.6589 \cdot10^{-2 } $ && $ 5.7184 \cdot10^{1 } $ && $ 1.3441 \cdot10^{5 } $ & \\
& $ \log(2\cdot 10^{11}) $ && $ 7.8683 \cdot10^{-5 } $ && $ 2.0591 \cdot10^{-3 } $ && $ 5.3886 \cdot10^{-2 } $ && $ 5.7184 \cdot10^{1 } $ && $ 1.3441 \cdot10^{5 } $ & \\
& $ \log(3\cdot 10^{11}) $ && $ 7.0193 \cdot10^{-5 } $ && $ 1.8758 \cdot10^{-3 } $ && $ 5.0536 \cdot10^{-2 } $ && $ 5.7184 \cdot10^{1 } $ && $ 1.3441 \cdot10^{5 } $ & \\
& $ \log(4\cdot 10^{11}) $ && $ 7.0193 \cdot10^{-5 } $ && $ 1.8758 \cdot10^{-3 } $ && $ 5.0536 \cdot10^{-2 } $ && $ 5.7184 \cdot10^{1 } $ && $ 1.3441 \cdot10^{5 } $ & \\
& $ \log(5\cdot10^{11}) $ && $ 6.9322 \cdot10^{-5 } $ && $ 1.8717 \cdot10^{-3 } $ && $ 5.0536 \cdot10^{-2 } $ && $ 5.7184 \cdot10^{1 } $ && $ 1.3441 \cdot10^{5 } $ & \\
& $ \log(6\cdot10^{11}) $ && $ 6.9322 \cdot10^{-5 } $ && $ 1.8717 \cdot10^{-3 } $ && $ 5.0536 \cdot10^{-2 } $ && $ 5.7184 \cdot10^{1 } $ && $ 1.3441 \cdot10^{5 } $ & \\
& $ 28 $ && $ 4.3555 \cdot10^{-5 } $ && $ 1.2196 \cdot10^{-3 } $ && $ 3.4148 \cdot10^{-2 } $ && $ 5.7184 \cdot10^{1 } $ && $ 1.3441 \cdot10^{5 } $ & \\
& $ 29 $ && $ 2.7336 \cdot10^{-5 } $ && $ 7.9272 \cdot10^{-4 } $ && $ 2.4334 \cdot10^{-2 } $ && $ 5.7184 \cdot10^{1 } $ && $ 1.3441 \cdot10^{5 } $ & \\
& $ 30 $ && $ 1.7139 \cdot10^{-5 } $ && $ 5.1415 \cdot10^{-4 } $ && $ 2.4334 \cdot10^{-2 } $ && $ 5.7184 \cdot10^{1 } $ && $ 1.3441 \cdot10^{5 } $ & \\
& $ 31 $ && $ 1.0735 \cdot10^{-5 } $ && $ 3.3277 \cdot10^{-4 } $ && $ 2.4334 \cdot10^{-2 } $ && $ 5.7184 \cdot10^{1 } $ && $ 1.3441 \cdot10^{5 } $ & \\
& $ 32 $ && $ 7.0053 \cdot10^{-6 } $ && $ 2.2417 \cdot10^{-4 } $ && $ 2.4334 \cdot10^{-2 } $ && $ 5.7184 \cdot10^{1 } $ && $ 1.3441 \cdot10^{5 } $ & \\
& $ 33 $ && $ 4.3798 \cdot10^{-6 } $ && $ 1.4454 \cdot10^{-4 } $ && $ 2.4334 \cdot10^{-2 } $ && $ 5.7184 \cdot10^{1 } $ && $ 1.3441 \cdot10^{5 } $ & \\
& $ 34 $ && $ 2.7360 \cdot10^{-6 } $ && $ 9.3023 \cdot10^{-5 } $ && $ 2.4334 \cdot10^{-2 } $ && $ 5.7184 \cdot10^{1 } $ && $ 1.3441 \cdot10^{5 } $ & \\
& $ 35 $ && $ 1.7078 \cdot10^{-6 } $ && $ 5.9771 \cdot10^{-5 } $ && $ 2.4334 \cdot10^{-2 } $ && $ 5.7184 \cdot10^{1 } $ && $ 1.3441 \cdot10^{5 } $ & \\
& $ 36 $ && $ 1.0652 \cdot10^{-6 } $ && $ 3.8345 \cdot10^{-5 } $ && $ 2.4334 \cdot10^{-2 } $ && $ 5.7184 \cdot10^{1 } $ && $ 1.3441 \cdot10^{5 } $ & \\
& $ 37 $ && $ 8.6315 \cdot10^{-7 } $ && $ 3.7979 \cdot10^{-5 } $ && $ 2.4334 \cdot10^{-2 } $ && $ 5.7184 \cdot10^{1 } $ && $ 1.3441 \cdot10^{5 } $ & \\
& $ 38 $ && $ 8.6315 \cdot10^{-7 } $ && $ 3.7979 \cdot10^{-5 } $ && $ 2.4334 \cdot10^{-2 } $ && $ 5.7184 \cdot10^{1 } $ && $ 1.3441 \cdot10^{5 } $ & \\
& $ 39 $ && $ 8.6315 \cdot10^{-7 } $ && $ 3.7979 \cdot10^{-5 } $ && $ 2.4334 \cdot10^{-2 } $ && $ 5.7184 \cdot10^{1 } $ && $ 1.3441 \cdot10^{5 } $ & \\
& $ 40 $ && $ 8.6315 \cdot10^{-7 } $ && $ 3.7979 \cdot10^{-5 } $ && $ 2.4334 \cdot10^{-2 } $ && $ 5.7184 \cdot10^{1 } $ && $ 1.3441 \cdot10^{5 } $ & \\
& $ 41 $ && $ 8.6315 \cdot10^{-7 } $ && $ 3.7979 \cdot10^{-5 } $ && $ 2.4334 \cdot10^{-2 } $ && $ 5.7184 \cdot10^{1 } $ && $ 1.3441 \cdot10^{5 } $ & \\
& $ 42 $ && $ 8.6315 \cdot10^{-7 } $ && $ 3.7979 \cdot10^{-5 } $ && $ 2.4334 \cdot10^{-2 } $ && $ 5.7184 \cdot10^{1 } $ && $ 1.3441 \cdot10^{5 } $ & \\
& $ 43 $ && $ 8.6315 \cdot10^{-7 } $ && $ 3.7979 \cdot10^{-5 } $ && $ 2.4334 \cdot10^{-2 } $ && $ 5.7184 \cdot10^{1 } $ && $ 1.3441 \cdot10^{5 } $ & \\
& $ 19\log 10 $ && $ 8.6315 \cdot10^{-7 } $ && $ 3.7979 \cdot10^{-5 } $ && $ 2.4334 \cdot10^{-2 } $ && $ 5.7184 \cdot10^{1 } $ && $ 1.3441 \cdot10^{5 } $ & \\
& $ 44 $ && $ 7.8163 \cdot10^{-7 } $ && $ 3.5174 \cdot10^{-5 } $ && $ 2.4334 \cdot10^{-2 } $ && $ 5.7184 \cdot10^{1 } $ && $ 1.3441 \cdot10^{5 } $ & \\
& $ 45 $ && $ 5.0646 \cdot10^{-7 } $ && $ 2.3298 \cdot10^{-5 } $ && $ 2.4334 \cdot10^{-2 } $ && $ 5.7184 \cdot10^{1 } $ && $ 1.3441 \cdot10^{5 } $ & \\
& $ 46 $ && $ 3.2935 \cdot10^{-7 } $ && $ 1.5480 \cdot10^{-5 } $ && $ 2.4334 \cdot10^{-2 } $ && $ 5.7184 \cdot10^{1 } $ && $ 1.3441 \cdot10^{5 } $ & \\
& $ 47 $ && $ 2.1308 \cdot10^{-7 } $ && $ 1.0376 \cdot10^{-5 } $ && $ 2.4334 \cdot10^{-2 } $ && $ 5.7184 \cdot10^{1 } $ && $ 1.3441 \cdot10^{5 } $ & \\
& $ 48 $ && $ 1.3791 \cdot10^{-7 } $ && $ 1.0376 \cdot10^{-5 } $ && $ 2.4334 \cdot10^{-2 } $ && $ 5.7184 \cdot10^{1 } $ && $ 1.3441 \cdot10^{5 } $ & \\
& $ 49 $ && $ 8.9140 \cdot10^{-8 } $ && $ 1.0376 \cdot10^{-5 } $ && $ 2.4334 \cdot10^{-2 } $ && $ 5.7184 \cdot10^{1 } $ && $ 1.3441 \cdot10^{5 } $ & \\
& $ 50 $ && $ 5.7545 \cdot10^{-8 } $ && $ 1.0376 \cdot10^{-5 } $ && $ 2.4334 \cdot10^{-2 } $ && $ 5.7184 \cdot10^{1 } $ && $ 1.3441 \cdot10^{5 } $ & \\
% & $ 51 $ && $ 3.7147 \cdot10^{-8 } $ && $ 1.0376 \cdot10^{-5 } $ && $ 2.4334 \cdot10^{-2 } $ && $ 5.7184 \cdot10^{1 } $ && $ 1.3441 \cdot10^{5 } $ & \\
% & $ 52 $ && $ 2.3898 \cdot10^{-8 } $ && $ 1.0376 \cdot10^{-5 } $ && $ 2.4334 \cdot10^{-2 } $ && $ 5.7184 \cdot10^{1 } $ && $ 1.3441 \cdot10^{5 } $ & \\
% & $ 53 $ && $ 1.5373 \cdot10^{-8 } $ && $ 1.0376 \cdot10^{-5 } $ && $ 2.4334 \cdot10^{-2 } $ && $ 5.7184 \cdot10^{1 } $ && $ 1.3441 \cdot10^{5 } $ & \\
% & $ 54 $ && $ 9.8778 \cdot10^{-9 } $ && $ 1.0376 \cdot10^{-5 } $ && $ 2.4334 \cdot10^{-2 } $ && $ 5.7184 \cdot10^{1 } $ && $ 1.3441 \cdot10^{5 } $ & \\
& $ 55 $ && $ 6.3417 \cdot10^{-9 } $ && $ 1.0376 \cdot10^{-5 } $ && $ 2.4334 \cdot10^{-2 } $ && $ 5.7184 \cdot10^{1 } $ && $ 1.3441 \cdot10^{5 } $ & \\
% & $ 56 $ && $ 4.4627 \cdot10^{-9 } $ && $ 1.0376 \cdot10^{-5 } $ && $ 2.4334 \cdot10^{-2 } $ && $ 5.7184 \cdot10^{1 } $ && $ 1.3441 \cdot10^{5 } $ & \\
% & $ 57 $ && $ 4.4627 \cdot10^{-9 } $ && $ 1.0376 \cdot10^{-5 } $ && $ 2.4334 \cdot10^{-2 } $ && $ 5.7184 \cdot10^{1 } $ && $ 1.3441 \cdot10^{5 } $ & \\
% & $ 58 $ && $ 4.4627 \cdot10^{-9 } $ && $ 1.0376 \cdot10^{-5 } $ && $ 2.4334 \cdot10^{-2 } $ && $ 5.7184 \cdot10^{1 } $ && $ 1.3441 \cdot10^{5 } $ & \\
% & $ 59 $ && $ 4.4627 \cdot10^{-9 } $ && $ 1.0376 \cdot10^{-5 } $ && $ 2.4334 \cdot10^{-2 } $ && $ 5.7184 \cdot10^{1 } $ && $ 1.3441 \cdot10^{5 } $ & \\
& $ 60 $ && $ 4.4627 \cdot10^{-9 } $ && $ 1.0376 \cdot10^{-5 } $ && $ 2.4334 \cdot10^{-2 } $ && $ 5.7184 \cdot10^{1 } $ && $ 1.3441 \cdot10^{5 } $ & \\
& $ 65 $ && $ 4.4627 \cdot10^{-9 } $ && $ 1.0376 \cdot10^{-5 } $ && $ 2.4334 \cdot10^{-2 } $ && $ 5.7184 \cdot10^{1 } $ && $ 1.3441 \cdot10^{5 } $ & \\
& $ 70 $ && $ 4.4627 \cdot10^{-9 } $ && $ 1.0376 \cdot10^{-5 } $ && $ 2.4334 \cdot10^{-2 } $ && $ 5.7184 \cdot10^{1 } $ && $ 1.3441 \cdot10^{5 } $ & \\
% & $ 75 $ && $ 4.4627 \cdot10^{-9 } $ && $ 1.0376 \cdot10^{-5 } $ && $ 2.4334 \cdot10^{-2 } $ && $ 5.7184 \cdot10^{1 } $ && $ 1.3441 \cdot10^{5 } $ & \\
& $ 80 $ && $ 4.4627 \cdot10^{-9 } $ && $ 1.0376 \cdot10^{-5 } $ && $ 2.4334 \cdot10^{-2 } $ && $ 5.7184 \cdot10^{1 } $ && $ 1.3441 \cdot10^{5 } $ & \\
% & $ 85 $ && $ 4.4627 \cdot10^{-9 } $ && $ 1.0376 \cdot10^{-5 } $ && $ 2.4334 \cdot10^{-2 } $ && $ 5.7184 \cdot10^{1 } $ && $ 1.3441 \cdot10^{5 } $ & \\
& $ 90 $ && $ 4.4627 \cdot10^{-9 } $ && $ 1.0376 \cdot10^{-5 } $ && $ 2.4334 \cdot10^{-2 } $ && $ 5.7184 \cdot10^{1 } $ && $ 1.3441 \cdot10^{5 } $ & \\
% & $ 95 $ && $ 4.4627 \cdot10^{-9 } $ && $ 1.0376 \cdot10^{-5 } $ && $ 2.4334 \cdot10^{-2 } $ && $ 5.7184 \cdot10^{1 } $ && $ 1.3441 \cdot10^{5 } $ & \\
& $ 100 $ && $ 4.4627 \cdot10^{-9 } $ && $ 1.0376 \cdot10^{-5 } $ && $ 2.4334 \cdot10^{-2 } $ && $ 5.7184 \cdot10^{1 } $ && $ 1.3441 \cdot10^{5 } $ & \\
& $ 200 $ && $ 4.4627 \cdot10^{-9 } $ && $ 1.0376 \cdot10^{-5 } $ && $ 2.4334 \cdot10^{-2 } $ && $ 5.7184 \cdot10^{1 } $ && $ 1.3441 \cdot10^{5 } $ & \\
& $ 300 $ && $ 4.4627 \cdot10^{-9 } $ && $ 1.0376 \cdot10^{-5 } $ && $ 2.4334 \cdot10^{-2 } $ && $ 5.7184 \cdot10^{1 } $ && $ 1.3441 \cdot10^{5 } $ & \\
& $ 400 $ && $ 4.4627 \cdot10^{-9 } $ && $ 1.0376 \cdot10^{-5 } $ && $ 2.4334 \cdot10^{-2 } $ && $ 5.7184 \cdot10^{1 } $ && $ 1.3441 \cdot10^{5 } $ & \\
& $ 500 $ && $ 4.4627 \cdot10^{-9 } $ && $ 1.0376 \cdot10^{-5 } $ && $ 2.4334 \cdot10^{-2 } $ && $ 5.7184 \cdot10^{1 } $ && $ 1.3441 \cdot10^{5 } $ & \\
& $ 600 $ && $ 4.4627 \cdot10^{-9 } $ && $ 1.0376 \cdot10^{-5 } $ && $ 2.4334 \cdot10^{-2 } $ && $ 5.7184 \cdot10^{1 } $ && $ 1.3441 \cdot10^{5 } $ & \\
& $ 700 $ && $ 4.4627 \cdot10^{-9 } $ && $ 1.0376 \cdot10^{-5 } $ && $ 2.4334 \cdot10^{-2 } $ && $ 5.7184 \cdot10^{1 } $ && $ 1.3441 \cdot10^{5 } $ & \\
& $ 800 $ && $ 4.4627 \cdot10^{-9 } $ && $ 1.0376 \cdot10^{-5 } $ && $ 2.4334 \cdot10^{-2 } $ && $ 5.7184 \cdot10^{1 } $ && $ 1.3441 \cdot10^{5 } $ & \\
& $ 900 $ && $ 4.4627 \cdot10^{-9 } $ && $ 1.0376 \cdot10^{-5 } $ && $ 2.4334 \cdot10^{-2 } $ && $ 5.7184 \cdot10^{1 } $ && $ 1.3441 \cdot10^{5 } $ & \\
& $ 1000 $ && $ 4.4627 \cdot10^{-9 } $ && $ 1.0376 \cdot10^{-5 } $ && $ 2.4334 \cdot10^{-2 } $ && $ 5.7184 \cdot10^{1 } $ && $ 1.3441 \cdot10^{5 } $ & \\
& $ 1500 $ && $ 4.4627 \cdot10^{-9 } $ && $ 1.0376 \cdot10^{-5 } $ && $ 2.4334 \cdot10^{-2 } $ && $ 5.7184 \cdot10^{1 } $ && $ 1.3441 \cdot10^{5 } $ & \\
& $ 2000 $ && $ 4.4627 \cdot10^{-9 } $ && $ 1.0376 \cdot10^{-5 } $ && $ 2.4334 \cdot10^{-2 } $ && $ 5.7184 \cdot10^{1 } $ && $ 1.3441 \cdot10^{5 } $ & \\
& $ 2500 $ && $ 2.2885 \cdot10^{-9 } $ && $ 5.7783 \cdot10^{-6 } $ && $ 1.4591 \cdot10^{-2 } $ && $ 3.6840 \cdot10^{1 } $ && $ 9.3021 \cdot10^{4 } $ & \\
& $ 3000 $ && $ 1.3915 \cdot10^{-10 } $ && $ 4.2091 \cdot10^{-7 } $ && $ 1.2733 \cdot10^{-3 } $ && $ 3.8516 \cdot10^{0 } $ && $ 1.1651 \cdot10^{4 } $ & \\
& $ 3500 $ && $ 8.7646 \cdot10^{-12 } $ && $ 3.0896 \cdot10^{-8 } $ && $ 1.0891 \cdot10^{-4 } $ && $ 3.8390 \cdot10^{-1 } $ && $ 1.3533 \cdot10^{3 } $ & \\
& $ 4000 $ && $ 5.7410 \cdot10^{-13 } $ && $ 2.3108 \cdot10^{-9 } $ && $ 9.3007 \cdot10^{-6 } $ && $ 3.7436 \cdot10^{-2 } $ && $ 1.5068 \cdot10^{2 } $ & \\
& $ 5000 $ && $ 2.8715 \cdot10^{-15 } $ && $ 1.4645 \cdot10^{-11 } $ && $ 7.4687 \cdot10^{-8 } $ && $ 3.8091 \cdot10^{-4 } $ && $ 1.9426 \cdot10^{0 } $ & \\
% & $ 5100 $ && $ 1.7087 \cdot10^{-15 } $ && $ 8.8850 \cdot10^{-12 } $ && $ 4.6202 \cdot10^{-8 } $ && $ 2.4025 \cdot10^{-4 } $ && $ 1.2493 \cdot10^{0 } $ & \\
% & $ 5200 $ && $ 1.0185 \cdot10^{-15 } $ && $ 5.3981 \cdot10^{-12 } $ && $ 2.8610 \cdot10^{-8 } $ && $ 1.5164 \cdot10^{-4 } $ && $ 8.0365 \cdot10^{-1 } $ & \\
% & $ 5300 $ && $ 6.0977 \cdot10^{-16 } $ && $ 3.2928 \cdot10^{-12 } $ && $ 1.7781 \cdot10^{-8 } $ && $ 9.6016 \cdot10^{-5 } $ && $ 5.1849 \cdot10^{-1 } $ & \\
% & $ 5400 $ && $ 3.6472 \cdot10^{-16 } $ && $ 2.0060 \cdot10^{-12 } $ && $ 1.1033 \cdot10^{-8 } $ && $ 6.0680 \cdot10^{-5 } $ && $ 3.3374 \cdot10^{-1 } $ & \\
% & $ 5500 $ && $ 2.1916 \cdot10^{-16 } $ && $ 1.2273 \cdot10^{-12 } $ && $ 6.8727 \cdot10^{-9 } $ && $ 3.8488 \cdot10^{-5 } $ && $ 2.1553 \cdot10^{-1 } $ & \\
% & $ 5600 $ && $ 1.3218 \cdot10^{-16 } $ && $ 7.5337 \cdot10^{-13 } $ && $ 4.2943 \cdot10^{-9 } $ && $ 2.4477 \cdot10^{-5 } $ && $ 1.3952 \cdot10^{-1 } $ & \\
% & $ 5700 $ && $ 8.0079 \cdot10^{-17 } $ && $ 4.6446 \cdot10^{-13 } $ && $ 2.6939 \cdot10^{-9 } $ && $ 1.5625 \cdot10^{-5 } $ && $ 9.0621 \cdot10^{-2 } $ & \\
% & $ 5800 $ && $ 4.8519 \cdot10^{-17 } $ && $ 2.8626 \cdot10^{-13 } $ && $ 1.6890 \cdot10^{-9 } $ && $ 9.9647 \cdot10^{-6 } $ && $ 5.8792 \cdot10^{-2 } $ & \\
% & $ 5900 $ && $ 2.9483 \cdot10^{-17 } $ && $ 1.7690 \cdot10^{-13 } $ && $ 1.0614 \cdot10^{-9 } $ && $ 6.3682 \cdot10^{-6 } $ && $ 3.8209 \cdot10^{-2 } $ & \\
& $ 6000 $ && $ 1.7952 \cdot10^{-17 } $ && $ 1.0951 \cdot10^{-13 } $ && $ 6.6798 \cdot10^{-10 } $ && $ 4.0747 \cdot10^{-6 } $ && $ 2.4856 \cdot10^{-2 } $ & \\
% & $ 6100 $ && $ 1.0988 \cdot10^{-17 } $ && $ 6.8120 \cdot10^{-14 } $ && $ 4.2235 \cdot10^{-10 } $ && $ 2.6186 \cdot10^{-6 } $ && $ 1.6235 \cdot10^{-2 } $ & \\
% & $ 6200 $ && $ 6.7178 \cdot10^{-18 } $ && $ 4.2322 \cdot10^{-14 } $ && $ 2.6663 \cdot10^{-10 } $ && $ 1.6798 \cdot10^{-6 } $ && $ 1.0583 \cdot10^{-2 } $ & \\
% & $ 6300 $ && $ 4.1268 \cdot10^{-18 } $ && $ 2.6412 \cdot10^{-14 } $ && $ 1.6904 \cdot10^{-10 } $ && $ 1.0818 \cdot10^{-6 } $ && $ 6.9236 \cdot10^{-3 } $ & \\
% & $ 6400 $ && $ 2.5481 \cdot10^{-18 } $ && $ 1.6563 \cdot10^{-14 } $ && $ 1.0766 \cdot10^{-10 } $ && $ 6.9977 \cdot10^{-7 } $ && $ 4.5485 \cdot10^{-3 } $ & \\
% & $ 6500 $ && $ 1.5741 \cdot10^{-18 } $ && $ 1.0389 \cdot10^{-14 } $ && $ 6.8566 \cdot10^{-11 } $ && $ 4.5254 \cdot10^{-7 } $ && $ 2.9868 \cdot10^{-3 } $ & \\
% & $ 6600 $ && $ 9.7287 \cdot10^{-19 } $ && $ 6.5182 \cdot10^{-15 } $ && $ 4.3672 \cdot10^{-11 } $ && $ 2.9261 \cdot10^{-7 } $ && $ 1.9605 \cdot10^{-3 } $ & \\
% & $ 6700 $ && $ 6.0447 \cdot10^{-19 } $ && $ 4.1104 \cdot10^{-15 } $ && $ 2.7951 \cdot10^{-11 } $ && $ 1.9007 \cdot10^{-7 } $ && $ 1.2925 \cdot10^{-3 } $ & \\
% & $ 6800 $ && $ 3.7652 \cdot10^{-19 } $ && $ 2.5980 \cdot10^{-15 } $ && $ 1.7926 \cdot10^{-11 } $ && $ 1.2369 \cdot10^{-7 } $ && $ 8.5345 \cdot10^{-4 } $ & \\
% & $ 6900 $ && $ 2.3555 \cdot10^{-19 } $ && $ 1.6488 \cdot10^{-15 } $ && $ 1.1542 \cdot10^{-11 } $ && $ 8.0791 \cdot10^{-8 } $ && $ 5.6554 \cdot10^{-4 } $ & \\
& $ 7000 $ && $ 1.4744 \cdot10^{-19 } $ && $ 1.0468 \cdot10^{-15 } $ && $ 7.4322 \cdot10^{-12 } $ && $ 5.2769 \cdot10^{-8 } $ && $ 3.7466 \cdot10^{-4 } $ & \\
% & $ 7100 $ && $ 9.2826 \cdot10^{-20 } $ && $ 6.6835 \cdot10^{-16 } $ && $ 4.8121 \cdot10^{-12 } $ && $ 3.4647 \cdot10^{-8 } $ && $ 2.4946 \cdot10^{-4 } $ & \\
% & $ 7200 $ && $ 5.8601 \cdot10^{-20 } $ && $ 4.2779 \cdot10^{-16 } $ && $ 3.1229 \cdot10^{-12 } $ && $ 2.2797 \cdot10^{-8 } $ && $ 1.6642 \cdot10^{-4 } $ & \\
% & $ 7300 $ && $ 3.7014 \cdot10^{-20 } $ && $ 2.7391 \cdot10^{-16 } $ && $ 2.0269 \cdot10^{-12 } $ && $ 1.4999 \cdot10^{-8 } $ && $ 1.1100 \cdot10^{-4 } $ & \\
% & $ 7400 $ && $ 2.3489 \cdot10^{-20 } $ && $ 1.7617 \cdot10^{-16 } $ && $ 1.3213 \cdot10^{-12 } $ && $ 9.9092 \cdot10^{-9 } $ && $ 7.4319 \cdot10^{-5 } $ & \\
% & $ 7500 $ && $ 1.4908 \cdot10^{-20 } $ && $ 1.1330 \cdot10^{-16 } $ && $ 8.6105 \cdot10^{-13 } $ && $ 6.5440 \cdot10^{-9 } $ && $ 4.9735 \cdot10^{-5 } $ & \\
% & $ 7600 $ && $ 9.4822 \cdot10^{-21 } $ && $ 7.3013 \cdot10^{-17 } $ && $ 5.6220 \cdot10^{-13 } $ && $ 4.3290 \cdot10^{-9 } $ && $ 3.3333 \cdot10^{-5 } $ & \\
% & $ 7700 $ && $ 6.0570 \cdot10^{-21 } $ && $ 4.7245 \cdot10^{-17 } $ && $ 3.6851 \cdot10^{-13 } $ && $ 2.8744 \cdot10^{-9 } $ && $ 2.2420 \cdot10^{-5 } $ & \\
% & $ 7800 $ && $ 3.8812 \cdot10^{-21 } $ && $ 3.0661 \cdot10^{-17 } $ && $ 2.4223 \cdot10^{-13 } $ && $ 1.9136 \cdot10^{-9 } $ && $ 1.5117 \cdot10^{-5 } $ & \\
% & $ 7900 $ && $ 2.4928 \cdot10^{-21 } $ && $ 1.9943 \cdot10^{-17 } $ && $ 1.5954 \cdot10^{-13 } $ && $ 1.2764 \cdot10^{-9 } $ && $ 1.0211 \cdot10^{-5 } $ & \\
& $ 8000 $ && $ 1.6007 \cdot10^{-21 } $ && $ 1.2966 \cdot10^{-17 } $ && $ 1.0502 \cdot10^{-13 } $ && $ 8.5065 \cdot10^{-10 } $ && $ 6.8903 \cdot10^{-6 } $ & \\
% & $ 8100 $ && $ 1.0327 \cdot10^{-21 } $ && $ 8.4675 \cdot10^{-18 } $ && $ 6.9433 \cdot10^{-14 } $ && $ 5.6935 \cdot10^{-10 } $ && $ 4.6687 \cdot10^{-6 } $ & \\
% & $ 8200 $ && $ 6.6829 \cdot10^{-22 } $ && $ 5.5468 \cdot10^{-18 } $ && $ 4.6038 \cdot10^{-14 } $ && $ 3.8212 \cdot10^{-10 } $ && $ 3.1716 \cdot10^{-6 } $ & \\
% & $ 8300 $ && $ 4.3257 \cdot10^{-22 } $ && $ 3.6336 \cdot10^{-18 } $ && $ 3.0523 \cdot10^{-14 } $ && $ 2.5639 \cdot10^{-10 } $ && $ 2.1537 \cdot10^{-6 } $ & \\
% & $ 8400 $ && $ 2.8105 \cdot10^{-22 } $ && $ 2.3889 \cdot10^{-18 } $ && $ 2.0306 \cdot10^{-14 } $ && $ 1.7260 \cdot10^{-10 } $ && $ 1.4671 \cdot10^{-6 } $ & \\
% & $ 8500 $ && $ 1.8316 \cdot10^{-22 } $ && $ 1.5752 \cdot10^{-18 } $ && $ 1.3547 \cdot10^{-14 } $ && $ 1.1650 \cdot10^{-10 } $ && $ 1.0019 \cdot10^{-6 } $ & \\
% & $ 8600 $ && $ 1.1981 \cdot10^{-22 } $ && $ 1.0424 \cdot10^{-18 } $ && $ 9.0682 \cdot10^{-15 } $ && $ 7.8894 \cdot10^{-11 } $ && $ 6.8638 \cdot10^{-7 } $ & \\
% & $ 8700 $ && $ 7.8221 \cdot10^{-23 } $ && $ 6.8834 \cdot10^{-19 } $ && $ 6.0574 \cdot10^{-15 } $ && $ 5.3305 \cdot10^{-11 } $ && $ 4.6909 \cdot10^{-7 } $ & \\
% & $ 8800 $ && $ 5.1585 \cdot10^{-23 } $ && $ 4.5911 \cdot10^{-19 } $ && $ 4.0861 \cdot10^{-15 } $ && $ 3.6366 \cdot10^{-11 } $ && $ 3.2366 \cdot10^{-7 } $ & \\
% & $ 8900 $ && $ 3.3883 \cdot10^{-23 } $ && $ 3.0495 \cdot10^{-19 } $ && $ 2.7445 \cdot10^{-15 } $ && $ 2.4701 \cdot10^{-11 } $ && $ 2.2231 \cdot10^{-7 } $ & \\
& $ 9000 $ && $ 2.2253 \cdot10^{-23 } $ && $ 2.0250 \cdot10^{-19 } $ && $ 1.8428 \cdot10^{-15 } $ && $ 1.6769 \cdot10^{-11 } $ && $ 1.5260 \cdot10^{-7 } $ & \\
% & $ 9100 $ && $ 1.4642 \cdot10^{-23 } $ && $ 1.3470 \cdot10^{-19 } $ && $ 1.2393 \cdot10^{-15 } $ && $ 1.1401 \cdot10^{-11 } $ && $ 1.0489 \cdot10^{-7 } $ & \\
% & $ 9200 $ && $ 9.6778 \cdot10^{-24 } $ && $ 9.0003 \cdot10^{-20 } $ && $ 8.3703 \cdot10^{-16 } $ && $ 7.7844 \cdot10^{-12 } $ && $ 7.2395 \cdot10^{-8 } $ & \\
% & $ 9300 $ && $ 6.4195 \cdot10^{-24 } $ && $ 6.0344 \cdot10^{-20 } $ && $ 5.6723 \cdot10^{-16 } $ && $ 5.3320 \cdot10^{-12 } $ && $ 5.0121 \cdot10^{-8 } $ & \\
% & $ 9400 $ && $ 4.2972 \cdot10^{-24 } $ && $ 4.0823 \cdot10^{-20 } $ && $ 3.8782 \cdot10^{-16 } $ && $ 3.6843 \cdot10^{-12 } $ && $ 3.5001 \cdot10^{-8 } $ & \\
% & $ 9500 $ && $ 2.8513 \cdot10^{-24 } $ && $ 2.7372 \cdot10^{-20 } $ && $ 2.6277 \cdot10^{-16 } $ && $ 2.5226 \cdot10^{-12 } $ && $ 2.4217 \cdot10^{-8 } $ & \\
% & $ 9600 $ && $ 1.9059 \cdot10^{-24 } $ && $ 1.8487 \cdot10^{-20 } $ && $ 1.7933 \cdot10^{-16 } $ && $ 1.7395 \cdot10^{-12 } $ && $ 1.6873 \cdot10^{-8 } $ & \\
% & $ 9700 $ && $ 1.2689 \cdot10^{-24 } $ && $ 1.2436 \cdot10^{-20 } $ && $ 1.2187 \cdot10^{-16 } $ && $ 1.1943 \cdot10^{-12 } $ && $ 1.1704 \cdot10^{-8 } $ & \\
% & $ 9800 $ && $ 8.4841 \cdot10^{-25 } $ && $ 8.3993 \cdot10^{-21 } $ && $ 8.3153 \cdot10^{-17 } $ && $ 8.2321 \cdot10^{-13 } $ && $ 8.1498 \cdot10^{-9 } $ & \\
% & $ 9900 $ && $ 5.7396 \cdot10^{-25 } $ && $ 5.7396 \cdot10^{-21 } $ && $ 5.7396 \cdot10^{-17 } $ && $ 5.7396 \cdot10^{-13 } $ && $ 5.7396 \cdot10^{-9 } $ & \\
& $ 10000 $ && $ 3.8228 \cdot10^{-25 } $ && $ 3.8610 \cdot10^{-21 } $ && $ 3.8997 \cdot10^{-17 } $ && $ 3.9387 \cdot10^{-13 } $ && $ 3.9780 \cdot10^{-9 } $ & \\
% & $ 10100 $ && $ 2.5746 \cdot10^{-25 } $ && $ 2.6260 \cdot10^{-21 } $ && $ 2.6786 \cdot10^{-17 } $ && $ 2.7321 \cdot10^{-13 } $ && $ 2.7868 \cdot10^{-9 } $ & \\
% & $ 10200 $ && $ 1.7389 \cdot10^{-25 } $ && $ 1.7911 \cdot10^{-21 } $ && $ 1.8448 \cdot10^{-17 } $ && $ 1.9002 \cdot10^{-13 } $ && $ 1.9572 \cdot10^{-9 } $ & \\
% & $ 10300 $ && $ 1.1734 \cdot10^{-25 } $ && $ 1.2204 \cdot10^{-21 } $ && $ 1.2692 \cdot10^{-17 } $ && $ 1.3199 \cdot10^{-13 } $ && $ 1.3727 \cdot10^{-9 } $ & \\
% & $ 10400 $ && $ 7.9556 \cdot10^{-26 } $ && $ 8.3534 \cdot10^{-22 } $ && $ 8.7711 \cdot10^{-18 } $ && $ 9.2096 \cdot10^{-14 } $ && $ 9.6701 \cdot10^{-10 } $ & \\
% & $ 10500 $ && $ 5.4077 \cdot10^{-26 } $ && $ 5.7321 \cdot10^{-22 } $ && $ 6.0761 \cdot10^{-18 } $ && $ 6.4406 \cdot10^{-14 } $ && $ 6.8270 \cdot10^{-10 } $ & \\
% & $ 10600 $ && $ 3.6845 \cdot10^{-26 } $ && $ 3.9424 \cdot10^{-22 } $ && $ 4.2184 \cdot10^{-18 } $ && $ 4.5137 \cdot10^{-14 } $ && $ 4.8296 \cdot10^{-10 } $ & \\
% & $ 10700 $ && $ 2.5150 \cdot10^{-26 } $ && $ 2.7162 \cdot10^{-22 } $ && $ 2.9335 \cdot10^{-18 } $ && $ 3.1682 \cdot10^{-14 } $ && $ 3.4217 \cdot10^{-10 } $ & \\
% & $ 10800 $ && $ 1.7201 \cdot10^{-26 } $ && $ 1.8749 \cdot10^{-22 } $ && $ 2.0437 \cdot10^{-18 } $ && $ 2.2276 \cdot10^{-14 } $ && $ 2.4281 \cdot10^{-10 } $ & \\
% & $ 10900 $ && $ 1.1640 \cdot10^{-26 } $ && $ 1.2803 \cdot10^{-22 } $ && $ 1.4084 \cdot10^{-18 } $ && $ 1.5492 \cdot10^{-14 } $ && $ 1.7041 \cdot10^{-10 } $ & \\
& $ 11000 $ && $ 7.9284 \cdot10^{-27 } $ && $ 8.8005 \cdot10^{-23 } $ && $ 9.7685 \cdot10^{-19 } $ && $ 1.0844 \cdot10^{-14 } $ && $ 1.2036 \cdot10^{-10 } $ & \\
% & $ 11100 $ && $ 5.4156 \cdot10^{-27 } $ && $ 6.0655 \cdot10^{-23 } $ && $ 6.7933 \cdot10^{-19 } $ && $ 7.6085 \cdot10^{-15 } $ && $ 8.5215 \cdot10^{-11 } $ & \\
% & $ 11200 $ && $ 3.7120 \cdot10^{-27 } $ && $ 4.1946 \cdot10^{-23 } $ && $ 4.7398 \cdot10^{-19 } $ && $ 5.3560 \cdot10^{-15 } $ && $ 6.0523 \cdot10^{-11 } $ & \\
% & $ 11300 $ && $ 2.5509 \cdot10^{-27 } $ && $ 2.9080 \cdot10^{-23 } $ && $ 3.3152 \cdot10^{-19 } $ && $ 3.7793 \cdot10^{-15 } $ && $ 4.3084 \cdot10^{-11 } $ & \\
% & $ 11400 $ && $ 1.7570 \cdot10^{-27 } $ && $ 2.0205 \cdot10^{-23 } $ && $ 2.3236 \cdot10^{-19 } $ && $ 2.6721 \cdot10^{-15 } $ && $ 3.0729 \cdot10^{-11 } $ & \\
% & $ 11500 $ && $ 1.2103 \cdot10^{-27 } $ && $ 1.4039 \cdot10^{-23 } $ && $ 1.6285 \cdot10^{-19 } $ && $ 1.8891 \cdot10^{-15 } $ && $ 2.1913 \cdot10^{-11 } $ & \\
% & $ 11600 $ && $ 8.3445 \cdot10^{-28 } $ && $ 9.7630 \cdot10^{-24 } $ && $ 1.1423 \cdot10^{-19 } $ && $ 1.3365 \cdot10^{-15 } $ && $ 1.5637 \cdot10^{-11 } $ & \\
% & $ 11700 $ && $ 5.7692 \cdot10^{-28 } $ && $ 6.8076 \cdot10^{-24 } $ && $ 8.0330 \cdot10^{-20 } $ && $ 9.4789 \cdot10^{-16 } $ && $ 1.1186 \cdot10^{-11 } $ & \\
% & $ 11800 $ && $ 3.9987 \cdot10^{-28 } $ && $ 4.7584 \cdot10^{-24 } $ && $ 5.6625 \cdot10^{-20 } $ && $ 6.7384 \cdot10^{-16 } $ && $ 8.0187 \cdot10^{-12 } $ & \\
% & $ 11900 $ && $ 2.7776 \cdot10^{-28 } $ && $ 3.3331 \cdot10^{-24 } $ && $ 3.9997 \cdot10^{-20 } $ && $ 4.7996 \cdot10^{-16 } $ && $ 5.7596 \cdot10^{-12 } $ & \\
& $ 12000 $ && $ 1.9331 \cdot10^{-28 } $ && $ 2.3390 \cdot10^{-24 } $ && $ 2.8302 \cdot10^{-20 } $ && $ 3.4245 \cdot10^{-16 } $ && $ 4.1437 \cdot10^{-12 } $ & \\
% & $ 12100 $ && $ 1.3478 \cdot10^{-28 } $ && $ 1.6443 \cdot10^{-24 } $ && $ 2.0060 \cdot10^{-20 } $ && $ 2.4473 \cdot10^{-16 } $ && $ 2.9857 \cdot10^{-12 } $ & \\
% & $ 12200 $ && $ 9.3147 \cdot10^{-29 } $ && $ 1.1457 \cdot10^{-24 } $ && $ 1.4093 \cdot10^{-20 } $ && $ 1.7334 \cdot10^{-16 } $ && $ 2.1320 \cdot10^{-12 } $ & \\
% & $ 12300 $ && $ 6.5069 \cdot10^{-29 } $ && $ 8.0686 \cdot10^{-25 } $ && $ 1.0005 \cdot10^{-20 } $ && $ 1.2407 \cdot10^{-16 } $ && $ 1.5384 \cdot10^{-12 } $ & \\
% & $ 12400 $ && $ 4.5540 \cdot10^{-29 } $ && $ 5.6924 \cdot10^{-25 } $ && $ 7.1155 \cdot10^{-21 } $ && $ 8.8944 \cdot10^{-17 } $ && $ 1.1118 \cdot10^{-12 } $ & \\
% & $ 12500 $ && $ 3.1924 \cdot10^{-29 } $ && $ 4.0224 \cdot10^{-25 } $ && $ 5.0682 \cdot10^{-21 } $ && $ 6.3860 \cdot10^{-17 } $ && $ 8.0463 \cdot10^{-13 } $ & \\
% & $ 12600 $ && $ 2.2307 \cdot10^{-29 } $ && $ 2.8330 \cdot10^{-25 } $ && $ 3.5979 \cdot10^{-21 } $ && $ 4.5694 \cdot10^{-17 } $ && $ 5.8031 \cdot10^{-13 } $ & \\
% & $ 12700 $ && $ 1.5600 \cdot10^{-29 } $ && $ 1.9967 \cdot10^{-25 } $ && $ 2.5558 \cdot10^{-21 } $ && $ 3.2714 \cdot10^{-17 } $ && $ 4.1874 \cdot10^{-13 } $ & \\
% & $ 12800 $ && $ 1.0940 \cdot10^{-29 } $ && $ 1.4112 \cdot10^{-25 } $ && $ 1.8205 \cdot10^{-21 } $ && $ 2.3484 \cdot10^{-17 } $ && $ 3.0294 \cdot10^{-13 } $ & \\
% & $ 12900 $ && $ 7.6900 \cdot10^{-30 } $ && $ 9.9970 \cdot10^{-26 } $ && $ 1.2996 \cdot10^{-21 } $ && $ 1.6895 \cdot10^{-17 } $ && $ 2.1964 \cdot10^{-13 } $ & \\
& $ 13000 $ && $ 5.5830 \cdot10^{-30 } $ && $ 7.5371 \cdot10^{-26 } $ && $ 1.0175 \cdot10^{-21 } $ && $ 1.3737 \cdot10^{-17 } $ && $ 1.8544 \cdot10^{-13 } $ & \\
% & $ 13500 $ && $ 9.9578 \cdot10^{-31 } $ && $ 1.3743 \cdot10^{-26 } $ && $ 1.8966 \cdot10^{-22 } $ && $ 2.6174 \cdot10^{-18 } $ && $ 3.6123 \cdot10^{-14 } $ & \\
% & $ 13800 $ && $ 3.5593 \cdot10^{-31 } $ && $ 4.9830 \cdot10^{-27 } $ && $ 6.9761 \cdot10^{-23 } $ && $ 9.7665 \cdot10^{-19 } $ && $ 1.3674 \cdot10^{-14 } $ & \\
& $ 14000 $ && $ 1.8399 \cdot10^{-31 } $ && $ 2.7598 \cdot10^{-27 } $ && $ 4.1396 \cdot10^{-23 } $ && $ 6.2094 \cdot10^{-19 } $ && $ 9.3141 \cdot10^{-15 } $ & \\
& $ 15000 $ && $ 6.5712 \cdot10^{-33 } $ && $ 1.0514 \cdot10^{-28 } $ && $ 1.6823 \cdot10^{-24 } $ && $ 2.6916 \cdot10^{-20 } $ && $ 4.3065 \cdot10^{-16 } $ & \\
& $ 16000 $ && $ 2.5739 \cdot10^{-34 } $ && $ 4.3756 \cdot10^{-30 } $ && $ 7.4384 \cdot10^{-26 } $ && $ 1.2646 \cdot10^{-21 } $ && $ 2.1497 \cdot10^{-17 } $ & \\
& $ 17000 $ && $ 1.1168 \cdot10^{-35 } $ && $ 2.0101 \cdot10^{-31 } $ && $ 3.6182 \cdot10^{-27 } $ && $ 6.5127 \cdot10^{-23 } $ && $ 1.1723 \cdot10^{-18 } $ & \\
& $ 18000 $ && $ 5.3739 \cdot10^{-37 } $ && $ 1.0211 \cdot10^{-32 } $ && $ 1.9400 \cdot10^{-28 } $ && $ 3.6860 \cdot10^{-24 } $ && $ 7.0033 \cdot10^{-20 } $ & \\
& $ 19000 $ && $ 2.7357 \cdot10^{-38 } $ && $ 5.4714 \cdot10^{-34 } $ && $ 1.0943 \cdot10^{-29 } $ && $ 2.1886 \cdot10^{-25 } $ && $ 4.3772 \cdot10^{-21 } $ & \\
& $ 20000 $ && $ 1.5041 \cdot10^{-39 } $ && $ 3.1585 \cdot10^{-35 } $ && $ 6.6329 \cdot10^{-31 } $ && $ 1.3929 \cdot10^{-26 } $ && $ 2.9251 \cdot10^{-22 } $ & \\
& $ 21000 $ && $ 9.0606 \cdot10^{-41 } $ && $ 1.9934 \cdot10^{-36 } $ && $ 4.3853 \cdot10^{-32 } $ && $ 9.6477 \cdot10^{-28 } $ && $ 2.1225 \cdot10^{-23 } $ & \\
& $ 22000 $ && $ 5.6101 \cdot10^{-42 } $ && $ 1.2904 \cdot10^{-37 } $ && $ 2.9678 \cdot10^{-33 } $ && $ 6.8258 \cdot10^{-29 } $ && $ 1.5700 \cdot10^{-24 } $ & \\
& $ 23000 $ && $ 3.7554 \cdot10^{-43 } $ && $ 9.0129 \cdot10^{-39 } $ && $ 2.1631 \cdot10^{-34 } $ && $ 5.1915 \cdot10^{-30 } $ && $ 1.2460 \cdot10^{-25 } $ & \\
& $ 24000 $ && $ 2.6756 \cdot10^{-44 } $ && $ 6.6889 \cdot10^{-40 } $ && $ 1.6723 \cdot10^{-35 } $ && $ 4.1806 \cdot10^{-31 } $ && $ 1.0452 \cdot10^{-26 } $ & \\
& $ 25000 $ && $ 7.5635 \cdot10^{-45 } $ && $ 1.8909 \cdot10^{-40 } $ && $ 4.7272 \cdot10^{-36 } $ && $ 1.1818 \cdot10^{-31 } $ && $ 2.9545 \cdot10^{-27 } $ & \\
\end{longtable}
\end{center}
%Calculations shown in \path{Dropbox/Undergraduate-Research-Sharper-Psi-Theta-Summer2017/sharper-bounds-psi-theta-Broadbent-Wilk-Lumley-Kad-Ng/Programs_for_Sharper_bounds/June_2020/CombinedABCTable-June-2019.pari}.
}


\begin{thebibliography}{99}

% \bibitem{Axler16}
% \textsc{C. Axler,} {\it New bounds for the prime counting function},  Integers 16 (2016), Paper No. A22, 15 pp. 

% \bibitem{Axler17_1}
% \textsc{C. Axler,} {\it Estimates for $\pi (x)$ for large values of $x$ and Ramanujan's prime counting inequality}, arXiv:1703.02407v2 (March 29, 2017).

\bibitem{Axler17_2}
\textsc{C. Axler,} {\it New estimates for some functions defined over primes}, Integers 18 (2018), Paper No. A52, 21 pp.

% \bibitem{BH}
% \textsc{C. Bays and R. Hudson}, 
% {\it A new bound for the smallest x with $\pi(x)>li(x)$}, Math. Comp. 69 (2000), no. 231, 1285--1296. 

\bibitem{BKLNW}
\textsc{S. Broadbent, A. Fiori, H. Kadiri, A. Lumley, N. Ng, J. Swidinsky, K. Wilk}, 
{\it Tables of values of the Chebyshev functions $\theta(x)$ and $\psi(x)$}, 37 pages.  

\bibitem{But}
\textsc{J. B\"uthe,} {\it Estimating $\pi (x)$ and related functions under partial RH assumptions},  Math. Comp. 85 (2016), no. 301, 2483--2498. 

% \bibitem{But1} %- Not called
% \textsc{J. B\"uthe,} {\it An improved analytic method for calculating $\pi (x)$},  Manuscripta Math. 151 (2016), no. 3-4, 329--352. 

\bibitem{But2}
\textsc{J. B\"uthe,} {\it An analytic method for bounding $\psi (x)$}, Math. Comp. 87 (2018), no. 312, 1991--2009.

\bibitem{Cheb1}
\textsc{P.L. Chebyshev,} {\it M\'emoire sur les nombres premiers}, J. Math. Pures Appl., 17 (1852), p. 366--390.
% \path{http://sites.mathdoc.fr/JMPA/PDF/JMPA_1852_1_17_A19_0.pdf}

%\bibitem{Cheb2}
%P.L. Chebyshev, {\it Sur la totalit\'e des nombres premiers inf\'erieurs \`a une limite donn\'ee}, J. Math. Pures Appl., 17 (1852), p. 341--365.
% \path{http://sites.mathdoc.fr/JMPA/PDF/JMPA_1852_1_17_A18_0.pdf}

\bibitem{Cos}
\textsc{N. Costa Pereira,} {\it Estimates for the Chebychev Function $\psi (x) - \theta (x)$},  Math. Comp. 44 (1985), no. 169, 211--221. 

\bibitem{Cos2}
\textsc{N. Costa Pereira,} {\it Elementary estimates for the Chebyshev functions $\psi(x)$ and for M\"obius function $m(x)$,} Acta Arith 52 (1989), 307--337.

\bibitem{Dav}
\textsc{H. Davenport},
\emph{Multiplicative Number Theory: Third Edition}, Springer, New York (2000).

\bibitem{DST}
P. Demichel, Y. Saouter, T. Trudgian, 
{\it A still sharper region where $\pi(x) - li(x)$ is positive}, 
Math. Comp. 84 (2015), no. 295, 2433?2446. 

\bibitem{Des}
\textsc{J. M. Deshouillers,} {\it Retour sur la methode de Chebyshev}, Bull. Soc. Math de France 49--50 (1977), 41--45.

\bibitem{Dud}
A. W Dudek, {\it An explicit result for primes between cubes}, Funct. et Approx., 55(2), 177--197 (2016).

\bibitem{Dus}
\textsc{P. Dusart,} {\it In\'egalit\'es explicites pour $\psi(x),\theta(x),\pi(x)$ et les nombres premiers} ({\it Explicit inequalities for $\psi(x),\theta(x),\pi(x)$ and the primes})
C. R . Math. Acad. Sci. Soc. R . Can. 21 (1999), no. 2, 53--59. 

% \bibitem{Dus1} %- Not called
% \textsc{P. Dusart,} {\it The $k^{th}$ prime is greater than $k(\log k+\log\log k-1)$ for $k\ge2$},
% Math. Comp. 68 (1999), no. 225, 411--415.

% \bibitem{Dus2}
% \textsc{P. Dusart,} {\it Sur la conjecture $\pi(x+y) \le \pi(x) + \pi(y)$} ({\it On the conjecture $\pi(x+y) \le \pi(x) + \pi(y)$}),  Acta Arith. 102 (2002), no. 4, 295--308. 

\bibitem{Dus5}
\textsc{P. Dusart,} {\it Estimates of some functions over primes without R.H.}, arXiv:1002.0442 (February 2, 2010).

\bibitem{Dus3}
\textsc{P. Dusart,} {\it Estimates of $\psi, \theta$ for large values of $x$ without the Riemann Hypothesis},
Math. Comp. 85 (2016), no. 298, 875--888.

\bibitem{Dus4}
\textsc{P. Dusart,} {\it Explicit estimates of some functions over primes}, Ramanujan J. 45 (2018), no. 1, 227--251.

\bibitem{Erdos}
\textsc{P. Erd\"os}, {\it Beweis eines Satzes von Tschebyschef}, Acta Litt. Sci. Szeged 5 (1932), 194--198.

% \bibitem{Est} %- Not called
% \textsc{F. Esteki,} {\it Approximations for some functions of primes}, M. Sc. Thesis, Univ. of Lethbridge, Lethbridge Alta. (2012).

\bibitem{FaKa}  
\textsc{L. Faber and H. Kadiri,} {\it New bounds for $\psi(x)$}, Math. Comp. 84 (2015), no. 293, 1339--1357. 

\bibitem{FaKaCorr}
\textsc{L. Faber and H. Kadiri,} {\it Corrigendum to New bounds for $\psi(x)$}, Math. Comp. 87 (2018), no. 311, 1451--1455.

\bibitem{Gou}
\textsc{X. Gourdon}, 
\emph{The $10^{13}$ First Zeros of the Riemann Zeta Function, and Zeros Computation at Very Large Height},
\url{http://numbers.computation.free.fr/Constants/Miscellaneous/ zetazeros1e13-1e24.pdf}
%

% \bibitem{Gram}%- Not called
% \textsc{J.P. Gram,} {\it Unders$\phi $geler angaaende Maengden af Primtal under en given Graense}, L. Danske Vidensk. 
% Selskabs Skrifter, Naturv. og Math. Afd., ser. 6, vol. 2 (1881 - 86), pp.183--308.

\bibitem{GrmHan}
\textsc{W. E. L. Grimson and D. Hanson}, {\it Estimates for the product of the primes not exceeding x}, Univ. Manitoba, Winnipeg, Man. (1977), 407--416.

\bibitem{Had}
\textsc{J. Hadamard,} {\it Sur la distribution des z\'eros de la fonction $\zeta(s)$ et ses cons\'equences arithm\'etiques}, Bull. Soc. Math. France 24 (1896), 199--220. 
% \path{http://fg2fy8yh7d.search.serialssolutions.com/?sid=AMS%3AMathSciNet&atitle=Sur%20la%20distribution%20des%20z%C3%A9ros%20de%20la%20fonction%20%24%5Czeta%28s%29%24%20et%20ses%20cons%C3%A9quences%20arithm%C3%A9tiques&aufirst=J.&auinit1=J&aulast=Hadamard&date=1896&epage=220&genre=article&issn=0037-9484&pages=199-220&spage=199&stitle=Bull.%20Soc.%20Math.%20France&title=Bulletin%20de%20la%20Soci%C3%A9t%C3%A9%20Math%C3%A9matique%20de%20France&volume=24}

\bibitem{Han}
\textsc{D. Hanson}, {\it On the product of primes,} Canad. Math. Bull 15 (1972), 33--37.

\bibitem{Kad}
\textsc{H. Kadiri}, 
\emph{Une R\'egion explicite sans z\'eros pour la fonction $\zeta$ de Riemann}, 
Acta Arith. 117 (2005), no. 4, 303--339.
%

\bibitem{Kad2}
\textsc{H. Kadiri,} {\it A zero density result for the Riemann zeta function}, Acta Arith. 160 (2013), no. 2, 185--200.

\bibitem{KLN}
\textsc{H. Kadiri, A. Lumley, and N. Ng,} {\it Explicit zero-density for the Riemann zeta function},
J. Math. Anal. Appl. 465 (2018), no. 1, 22--46. 

\bibitem{Koch}
\textsc{H. von Koch}, {\it Sur la distribution des nombres premiers}, Acta Math. 24 (1901), no. 1, 159--182.

% \bibitem{Kor} %- Only calling of this reference is commented out
% \textsc{N. M. Korobov}, {\it Estimates of trigonometric sums and their applications} (Russian), Uspehi Mat. Nauk 13 (1958), no. 4 (82), 185--192. 

\bibitem{Lehm}%- Not called
\textsc{D.N. Lehmer,} {\it List of prime numbers from $1$ to $10\ 006\ 721$}, Carnegie Institution of Washington, Publication No. 165, 1914.

\bibitem{Log88}
\textsc{B. F. Logan,} {\it Bounds for the tails of sharp-cutoff filter kernels}, SIAM J. Math. Anal 19 (1988), no.2, 372--376.

\bibitem{vdL}
\textsc{J. van de Lune, H. J. J. te Riele, D. T. Winter} {\it On the zeros of the Riemann zeta function in the critical strip. IV}, Math. Comp. 46 (1986), no. 174, 667--681.

% \bibitem{MasRob}%- Not called
% \textsc{J.-P. Massias, G. Robin,} {\it Bornes effectives pour certaines fonctions concernant les nombres premiers} (\textit{Effective bounds for certain functions relating to prime numbers})
%  J. Th\'eorie Nombres Bordeaux 8 (1996), no. 1, 215--242.

\bibitem{Mont}
\textsc{H.L. Montgomery,} {\em The zeta function and prime numbers}, Proceedings of the Queen's Number Theory Conference, 1979,
Queen's Univ., Kingston, Ont., 1980, 1--31.

\bibitem{TrudMoss}
\textsc{M. Mossinghoff and T. Trudgian,} \textit{Nonnegative trigonometric polynomials and a zero-free region for the Riemann zeta-function},
Journal of Number Theory 157 (2015) 329--349.

\bibitem{Pintz}
\textsc{J. Pintz}, 
\emph{On the remainder term of the prime number formula II. On a theorem of Ingham}, Acta. Arith. 37 (1980),
209-220. 

\bibitem{Pla0} 
\textsc{D.J. Platt}, 
\emph{Computing degree $1$ $L$-functions rigorously}, 
Ph.D. Thesis, University of Bristol (2011).
%
\bibitem{Pla} 
\textsc{D.J. Platt}, 
\emph{Computing $\pi(x)$ analytically}, 
Math. Comp. 84 (2015), no. 293, 1521--1535.

\bibitem{Pla-zeta} 
\textsc{D.J. Platt}, 
\emph{Isolating some non-trivial zeros of zeta}, 
 Math. Comp. 86 (2017), no. 307, 2449--2467.

\bibitem{PT2016}
\textsc{D.J. Platt, T.S. Trudgian}, {\it 
On the first sign change of $\theta(x)-x$}. 
Math. Comp. 85 (2016), no. 299, 1539--1547. 

\bibitem{PT2019}
D. Platt and T. Trudgian, {\it  The error term in the prime number theorem}, Math. Comp. 90 (2021), no. 328, 871-881. 

\bibitem{PT2020}
D. Platt and T. Trudgian, {\it  
The Riemann hypothesis is true up to $3 \cdot 10^{12}$},  to appear in  Bulletin of London Mathematical Society,  https://doi.org/10.1112/blms.12460.
%arXiv:2004.09765. 


\bibitem{Ram}
\textsc{O. Ramar\'e,} {\it An explicit density estimate for Dirichlet L-series}, Math. Comp. 85 (2016), 325--356.
 
% \bibitem{Rob}%- Not called
% \textsc{G. Robin,} {\it Estimation de la fonction de Tchebychef $\theta$ sur le k-i\`{e}me nombre permier et grandes valeurs de la fonction $\omega (n)$ nombre de diviseurs premiers de n}
% ({\it Estimation of the Chebyshef function $\theta$ on the k-th prime number and large values of the function $\omega (n)$ number of prime divisors of n})
% Acta Arith. 42 (1983), no. 4, 367--389.

\bibitem{Ros} 
\textsc{J.B. Rosser,} {\it Explicit bounds for some functions of prime numbers}, Amer. J. Math. 63 (1941), 211--232. 

\bibitem{Ros2}
\textsc{J.B. Rosser,} {\it Explicit remainder terms for some asymptotic series}, Journal of Rational Mechanics and Analysis (J. Math. Mech.), 
vol. 4 (1955), pp. 595--626

\bibitem{RS1} 
\textsc{J.B. Rosser and L. Schoenfeld,} {\it Approximate formulas for some functions of prime numbers}, Illinois. J. Math. 6 (1962), 64--94. 

\bibitem{RS2}
\textsc{J.B. Rosser and L. Schoenfeld,} {\it Sharper bounds for Chebyshev functions $\theta(x)$ and $\psi(x)$} I, Math. Comp. 29 (1975), 243--269. 

\bibitem{RubSar}
\textsc{M. Rubinstein and P. Sarnak,} {\it Chebyshev's bias}, Experiment. Math. 3 (1994), no. 3, 173--197.

\bibitem{SaTrDe}
% \textsc{Y. Saouter, T. Trudgian, and P. Demichel,} {\it A still sharper region where $\pi(x) ? \mathrm{li}(x)$ is positive},
% Math. Comp. 84 (2015), no. 295, 2433--2446.
\textsc{Y. Saouter, T. Trudgian, D. Patrick,} {\it A still sharper region where $\pi(x)-li(x)$ is positive},  
Math. Comp. 84 (2015), no. 295, 2433--2446. 

\bibitem{Sch}
\textsc{L. Schoenfeld,} {\it Sharper bounds for Chebyshev functions $\theta(x)$ and $\psi(x)$} II, Math. Comp. 30 (1976), 337--360.

\bibitem{Ste}
\textsc{S.B. Stechkin} {\it Zeros of the Riemann zeta-function}, Math. Notes 8 (1970) 706--711.\\
(Translated from Matematicheskie Zametki, Vol. 8, No. 4, pp. 419--429, October, 1970.)

\bibitem{Trudgian}
\textsc{T. Trudgian} {\it An Improved Upper Bound for the argument of the Riemann zeta function on the critical line.} Math. Comp. 81 (2012), 1053--1061. 

\bibitem{Trudgian2}
\textsc{T. Trudgian} {\it An Improved Upper Bound for the argument of the Riemann zeta function on the critical line II.} J. Number Theory 134 (2014),280--292. 

\bibitem{Tru}
\textsc{T. Trudgian,} {\it Updating the error term in the prime number theorem},  Ramanujan J. 39 (2016), no. 2, 225--234. 

\bibitem{dlVP}
\textsc{C.J. de la Vall\'ee Poussin} {\it Recherches analytiques sur la th\'eorie des nombres premiers}, Ann. Soc. Scient. Bruxelles, deuxi\'eme partie 20,
(1896), pp. 183--256 

% \path{https://archive.org/stream/annalesdelasoci05bruxgoog#page/n393/mode/2up}

% \bibitem{Vin}%- Only reference is commented out
% \textsc{I. M. Vinogradov}, {\it A new estimate for $\zeta(1+ it)$} (Russian), Izv. Akad. Nauk SSSR. Ser. Mat. 22 (1958), 161--164.

\bibitem{Wed} 
\textsc{S. Wedeniwski - ZETAGRID}, 
\emph{Computational verification of the Riemann hypothesis,}
Conference in Number Theory in Honour of Professor H.C. Williams, Alberta, Canada, May 2003.\\
\url{http://www.zetagrid.net/zeta/math/zeta.result.100billion.zeros.html}

\end{thebibliography}
\end{document}